\documentclass[leqn,10pt]{amsart}
\usepackage[utf8]{inputenc}
\usepackage{a4}
\usepackage{amsmath,amssymb,amscd, bbm}
\usepackage{enumerate}
\usepackage[mathscr]{eucal}
\usepackage[all]{xy}
\usepackage{hyperref}

\usepackage{pgflibrarysnakes}
\usepackage{pgflibraryarrows}
\usepackage{tikz}
\usetikzlibrary{decorations.pathmorphing, patterns,mindmap,trees}

\newtheorem{thm}{Theorem}[section]
\newtheorem{lemma}[thm]{Lemma}

\newtheorem{cor}[thm]{Corollary}

\theoremstyle{definition}
\newtheorem{defi}[thm]{Definition}
\newtheorem{exa}[thm]{Example}
\newtheorem{exas}[thm]{Examples}

\theoremstyle{remark}
\newtheorem{rem}[thm]{Remark}
\newtheorem{rems}[thm]{Remarks}
\numberwithin{equation}{section}

\newcommand{\nocontentsline}[3]{}
\newcommand{\tocless}[2]{\bgroup\let\addcontentsline=\nocontentsline#1{#2}\egroup}

\newcommand{\extend}[1]{#1}
\newcommand{\Jei}{U}
\newcommand{\gebiet}{O}
\newcommand{\subord}{\precsim}
\newcommand{\strongeq}{\approx}
\allowdisplaybreaks

\newcommand{\ELLONE} {$\ell_1$\,}

\newcommand{\embeds}{\hookrightarrow}

\usepackage{stmaryrd}

\newcommand{\Bignorm}[1]{\Bigl\| #1 \Bigr\|}

\newcommand {\sfrac}[2] { {{}^{#1}\!\!/\!{}_{#2}}}
\newcommand {\pihalbe}{\sfrac{\pi}2}
\newcommand {\einhalb}{\sfrac{1}2}
\newcommand{\emdf}{\em}

\newcommand{\Id}{\mathrm{I}}

\newcommand{\vanish}[1]{\relax}

\newcommand{\suchthat}{\,\,|\,\,}

\newcommand{\scrF}{\mathscr{F}}

\DeclareMathOperator{\calE}{\mathcal{E}}
\DeclareMathOperator{\calF}{\mathcal{F}}

\newcommand{\N}{\mathbb{N}}
\newcommand{\Z}{\mathbb{Z}}

\newcommand{\R}{\mathbb{R}}
\newcommand{\C}{\mathbb{C}}

\newcommand{\K}{\mathbb{K}}

\newcommand{\ud}{\mathrm{d}}
\newcommand{\ue}{\mathrm{e}}
\newcommand{\ui}{\mathrm{i}}

\newcommand{\We}{\mathrm{W}}
\newcommand{\Ha}{\mathrm{H}}

\newcommand{\vphi}{\varphi}
\newcommand{\veps}{\varepsilon}

\newcommand{\sector}[1]{\mathrm{S}_{#1}}

\newcommand{\strip}[1]{\mathrm{St}_{#1}}

\newcommand{\torus}{\mathbb{T}}

\newcommand{\mfrac}[2]{#1/#2}

\newcommand{\Laplace}{\Delta}

\newcommand{\restrict}{\big|}

\newcommand{\Sum}[2][\relax]{%
 \ifx#1\relax \sideset{}{_{#2}}\sum 
 \else \sideset{}{^{#1}_{#2}}\sum
 \fi}

\newcommand{\car}{\mathbf{1}}
\DeclareMathOperator{\re}{Re}
\DeclareMathOperator{\im}{Im}
\newcommand{\konj}[1]{\overline{#1}}

\newcommand{\abs}[1]{\left| #1 \right|}
\renewcommand{\abs}[1]{\left\vert#1\right\vert}

\newcommand{\Ce}{\mathrm{C}}

\newcommand{\ce}{\mathrm{c}}

\newcommand{\Ell}[1]{\mathrm{L}_{#1}}

\newcommand{\resolv}{\varrho}

\newcommand{\ohne}{\setminus}

\newcommand{\leer}{\emptyset}

\newcommand{\dann}{\rightarrow}

\newcommand{\Iff}{\Leftrightarrow}

\newcommand{\nach}{\circ}
\newcommand{\pfeil}{\longrightarrow}
\newcommand{\tpfeil}{\longmapsto}
\DeclareMathOperator{\dom}{dom}
\DeclareMathOperator{\ran}{ran}

\newcommand{\cls}[1]{\overline{#1}}

\newcommand{\rand}{\partial}
\DeclareMathOperator{\supp}{supp}

\DeclareMathOperator{\diag}{diag}

\newcommand{\spann}{\mathrm{span}}
\DeclareMathOperator{\tr}{tr}

\newcommand{\tensor}{\otimes}

\newcommand{\absconv}{\mathrm{absconv}}
\newcommand{\cabsconv}{\overline{\absconv}}

\DeclareMathOperator{\Lin}{\mathcal{L}}

\newcommand{\norm}[2][\relax]{%
   \ifx#1\relax \ensuremath{\left\Vert#2\right\Vert}
   \else \ensuremath{\left\Vert#2\right\Vert_{#1}}
   \fi}

\makeatletter
\newcommand{\sprod}[2]{\ensuremath{%
  \setbox0=\hbox{\ensuremath{#2}}
  \dimen@\ht0
  \advance\dimen@ by \dp0
  \left(\left.#1\rule[-\dp0]{0pt}{\dimen@}\,\right|#2\hspace{1pt}\right)}}
 \makeatother
\newcommand{\dprod}[2]{\ensuremath{\left<#1,#2\right>}}
\newcommand{\eprod}[3][\relax]{%
 \ifx#1\relax #2\diamond #3 
 \else #2\diamond_{#1} #3
 \fi}

\newcommand{\Fourier}{\mathcal{F}}
\newcommand{\fourier}[1]{\widehat{#1}}

\newcounter{aufzi}
\newenvironment{aufzi}{\begin{list}{ {\upshape\alph{aufzi})}}{
        \usecounter{aufzi}
        \topsep1ex
        \parsep0cm
        \itemsep1ex
        \leftmargin0.8cm
        \labelwidth0.5cm
        \labelsep0.3cm
}}
{\end{list}}

\newcounter{aufzii}
\newenvironment{aufzii}{\begin{list}{\hfill {\upshape 
(\roman{aufzii})}}{
        \usecounter{aufzii}
        \topsep1ex
        \parsep0cm
        \itemsep1ex
        \leftmargin0.8cm
        \labelwidth0.5cm
        \labelsep0.3cm
         \itemindent0cm
}}
{\end{list}}

\newcounter{aufziii}
\newenvironment{aufziii}{\begin{list}{ {\upshape\arabic{aufziii})}}{
        \usecounter{aufziii}
        \topsep1ex
        \parsep0cm
        \itemsep1ex
        \leftmargin0.8cm
        \labelwidth0.5cm
        \labelsep0.3cm
}}
{\end{list}}

\let\upi\pi

\def\tp{\mathrm{t}}
\def\ct{\mathrm{c}}
\def\bbD{\mathbb D}
\def\elementary{\calE}
\def\De{\mathrm{D}}

\def\Exp{\mathbb{E}}
\def\Pe{\mathrm{P}}
\def\co{\mathrm{c}_0}
\def\ui{\mathrm{i}}
\def\Nuc{\mathrm{Nu}}
\def\nuc{\mathrm{nu}}
\def\dann{\,\Rightarrow\,}
\def\LeMerdy{Le$\,$Merdy}

\date{\today}

\begin{document}

\title[Square function estimates and functional calculi]{Square
function estimates and functional calculi}

\author[Bernhard H. Haak]{Bernhard H. Haak}
\address{%
Institut de Math\'ematiques de Bordeaux\\Universit\'e Bordeaux 1\\351, cours de la Lib\'eration\\33405 Talence CEDEX\\FRANCE}
\email{bernhard.haak@math.u-bordeaux1.fr}

\author{Markus Haase}
\address{Delft Institute of Applied Mathematics, Delft University of Technology,
P.O. Box 5031, 2600 GA Delft, The Netherlands}
\email{m.h.a.haase@tudelft.nl}

\subjclass{Primary 47A60; Secondary 34G10 47D06 47N20}
\renewcommand{\subjclassname}{\textup{2000} Mathematics Subject
    Classification}

\thanks{The first mentioned author kindly acknowledges the support from the ANR projects
 ANR-09-BLAN-0058-01 and ANR-12-BS01-0013-02}
\date{\today}

\dedicatory{To the memory of Nigel Kalton (1946--2010)}

\begin{abstract}
In this paper the notion of an abstract square function (estimate)
is introduced as an operator $X \to \gamma(H;Y)$, where
$X,Y$ are Banach spaces, $H$ is a Hilbert space, and
$\gamma(H;Y)$ is the space
of $\gamma$-radonifying operators. By the seminal work of Kalton
and Weis, this definition is a coherent generalisation of the
classical notion of square function appearing in the theory of singular 
integrals. Given an abstract functional calculus $(\calE, \calF, \Phi)$
on a Banach space $X$, where $\calF(\gebiet)$ is an algebra of scalar-valued
functions on a set $\gebiet$, we define a square function $\Phi_\gamma(f)$ 
for certain {\em $H$-valued} functions $f$ on $\gebiet$.
The assignment $f \mapsto \Phi_\gamma(f)$
then becomes a {\em vectorial} functional calculus, and a 
``square function estimate'' for $f$ simply means 
the boundedness of $\Phi_\gamma(f)$. In this view, all
results linking  square function estimates
with the boundedness of a certain (usually the $\Ha^\infty$-) functional 
calculus simply assert that certain square function estimates
imply other square function estimates. In the present paper
several results of this type are proved in an abstract setting, 
based on the principles of {\em subordination}, 
{\em integral representation}, and a new  
boundedness concept for subsets of Hilbert spaces,  the so-called
{\em \ELLONE-frame-boundedness}.
These abstract
results are then applied to the $\Ha^\infty$-calculus for
sectorial and strip type operators. For example, it is proved
that any strip type operator with bounded scalar 
$\Ha^\infty$-calculus on a strip over a Banach space with finite
cotype has a bounded vectorial $\Ha^\infty$-calculus on every larger
strip. 
\end{abstract}

\maketitle

\tableofcontents

\section{Introduction}\label{s.intro}

Square functions and square function estimates are a classical topic and
a central tool in harmonic analysis, in particular in the so-called
Littlewood--Paley theory. Their history can be traced back to almost a century ago,
see \cite{Ste1958} for a historical account and \cite{Ste1961a,Ste1961b,Stein:HA}
for the development from  the 1960s on.  
One of the classical instances of a square function is 
\begin{equation}  \label{eq:square-function}
  (S_\phi f)(x) := \Bigl( \int_0^\infty \bigl| (\phi_t \ast f)(x) \bigr|^2 \,
\tfrac{\ud{t}}t\Bigr)^\einhalb
\end{equation}
where $\phi \in \Ell{2}(\R^d)$ decays reasonably fast at infinity
and $\phi_t(x) = t^{-d} \phi(x/t)$ for $x\in \R^d$ and $t > 0$. A
``square function estimate'' then reads
\begin{equation}  \label{eq:square-function-est}
  \norm{S_\phi f}_{\Ell{p}} = 
\Big\| \Bigl( \int_0^\infty \bigl| (\phi_t \ast f)(x) \bigr|^2 \,\tfrac{\ud{t}}t\Bigr)^\einhalb
\Big\|_{\Ell{p}} \lesssim \norm{f}_{\Ell{p}}.
\end{equation}
In many situations, $\phi$ is radial. Then its Fourier transform
is radial, too, and can be written as $\fourier{\phi}(\xi) = \psi(\abs{\xi})$
for $\xi \in \R^d$. Hence,
\[ 
\phi_t * f  = \scrF^{-1}(\widehat \phi(t \xi)  \cdot \widehat{f\,}(\xi) ) 
           = \scrF^{-1}(\widehat{\psi}(|t \xi|) \cdot \widehat{f\,}(\xi)) =  
\psi(t\,\sqrt{- \Laplace})\, f,
\]
where we employ the functional calculus for the Laplace, or better, the Poisson operator.
Hence, the abstract form of \eqref{eq:square-function-est} is
\begin{equation}  \label{eq:sfe-abs} 
\Big\| \Bigl( \int_0^\infty \abs{ \psi(tA)f}^2 \,\tfrac{\ud{t}}t\Bigr)^\einhalb
\Big\|_{\Ell{p}} \lesssim \norm{f}_{\Ell{p}},
\end{equation}
where $A := \sqrt{-\Laplace}$; and 
taking $\psi(z) := z \ue^{-z}$ we recover the classical Littlewood-Paley $g$-function.

From the mid 1980's on, the theory of functional calculus for sectorial
operators was developed by several people. Building on the seminal
works \cite{McI1986} and \cite{Cowling} and inspired by \cite{BoyDeL1992} 
Cowling, Doust, McIntosh and Yagi in \cite{CowDouMcIYag1996} 
established a strong link between
the boundedness of the $\Ha^\infty$-calculus for sectorial operators $A$ on 
(closed subspaces of) $\Ell{p}$-spaces 
and square functions of  the form \eqref{eq:sfe-abs}. 
Kalton and Weis in an unpublished and unfortunately never completed
manuscript \cite{KalWei2004} then showed how one could
pass from $\Ell{p}$-spaces to  general Banach spaces.
Their manuscript subsequently circulated and inspired
a considerable amount of research, e.g. 
\cite{BatMubVor2012,FroWei2006,HaaNeeVer2007,HaaKun2006,HytNeePor2008,%
Hyt2007,KaiWei2008,KalKunWei2006,%
KalNeeVerWei2008,KriLeM2010,%
LeM2004,LeMXu2012,NeeWei2005,%
NeeWei2006,NeeVerWei2007,Haa2011}.
It is also the starting point
of the present article.

The main novelty in Kalton and Weis' approach from \cite{KalWei2004} 
was to employ the class of so-called $\gamma$-radonifying operators in order
to define square functions. This step 
is motivated by  two observations. On the one hand, 
\[
\Bigl(\int_0^\infty \abs{ \psi(tA)f}^2 \,\tfrac{\ud{t}}t\Bigr)^\einhalb
  = \Bigl( \sum_{k=1}^\infty  \abs{[Tf]e_n}^2 \Bigr)^\einhalb
\]
where $(e_n)_{n\in \N}$ is an orthonormal basis of 
$H := \Ell{2}(\R_+;\sfrac{\ud{t}}{t})$ and 
$Tf : H \to X := \Ell{p}(\R^d)$ is the operator defined by
\[ [Tf]\,h := \int_{^0}^\infty h(t)\psi(tA)f\, \tfrac{\ud{t}}{t} \qquad (h \in H).
\]
(This holds true in each Banach lattice $X$, see
Section~\ref{ss:banach-lattices} below.) The second, more decisive
step, is based on the norm equivalence
\[   \Big\| \Bigl( {\sum}_k \abs{x_k}^2 \Bigr)^\einhalb \Big\|_X
 \sim  
\Bigl(    \Exp \big\|{\sum}_k \gamma_k  \tensor x_k\big\|_X^2 \Bigr)^\einhalb 
\]
where $(\gamma_k)_k$ is an independent sequence of standard Gaussian
random variables.  (This equivalence holds true in any Banach lattice $X$
of finite cotype, see Theorem~\ref{acsf.t.blfc-discrete}.)  Hence, the
square function estimate \eqref{eq:sfe-abs} can be reformulated as
\begin{equation}  \label{eq:sfe-abs-gamma} 
\Bigl(    \Exp \big\|{\sum}_k \gamma_k  \tensor [Tf]e_k\big\|_X^2 \Bigr)^\einhalb 
 \lesssim \norm{f}_X
\end{equation}
with $Tf$ as above. In other words, $Tf \in \gamma(H;X)$,  the space of
{\em $\gamma$-radonifying} operators, and $\norm{Tf}_{\gamma} \lesssim
\norm{f}_X$ (see Section~\ref{ss:banach-lattices} below).  In this
formulation of the square function estimate the 
lattice structure of $X= \Ell{p}$ does not appear any more
and hence it can be used
to define square function estimates over general Banach spaces $X$.

In the present paper we follow this approach, but transcend it in two points.
The first, minor, point is that we propose a {\em definition}
of a general square function as any linear operator
\[ T: \dom(T) \to \gamma(H;Y)
\]
where $X,Y$ are Banach spaces and $\dom(T)\subseteq X$ is a linear subspace.
A {\em square function estimate} for the square function $T$ then just
asserts its boundedness 
\[ \norm{Tx}_{\gamma(H;Y)} \lesssim \norm{x}_X.
\] 
If we admit finite dimensional Hilbert spaces $H$ here,
any (bounded) operator can be viewed as a trivial
square function (estimate).

The second and more important point of our approach is that for the
square functions of functional calculus type as before, we want to
systematise the dependence on the function $\psi$, e.g., in order
to cover square functions associated with expressions
of the form
\[ \psi(t,A) \quad \text{instead of}\quad \psi(tA).
\]
(We work with the functional calculus for sectorial operators for the time being,
as did Kalton and Weis).
Although we do not claim that this could not be done by the conservative
(Kalton-Weis) approach, we find it more natural to follow
the  {\em basic ideology of functional calculus}, i.e.,  to replace
working with operators by working with functions. This leads to
a re-reading of the operator $Tf$ from above as
\begin{equation}\label{intro:eq:twist}
 [Tf]\,h = \int_0^\infty h(t) \psi(tA)f \tfrac{\ud{t}}{t} = 
\Bigl( \int_0^\infty h(t) \psi(tz) \tfrac{\ud{t}}{t}\Bigr)(A)f.
\end{equation} 
(For good functions $\psi$ this is possible by the definition 
of $\psi(tA)$ as a Cauchy integral and Fubini's theorem.) 
The function $\psi(t,z) = \psi(tz)$ can now be viewed as a function of
two parameters, but since only integration with respect to $t$ is
performed here, it should better be viewed as a mapping
\[ \Psi:  \sector{\omega} \to H = 
\Ell{2}(\R_+;\sfrac{\ud{t}}{t}),\qquad 
\Psi(z) = (t  \mapsto \psi(t,z)) 
\]
The square function $T$ from above then appears as resulting
from inserting $A$ into the $H$-valued $\Ha^\infty$-function $\Psi$,
by defining
\[ \Psi(A)x := h \mapsto (z \mapsto \dprod{h}{\Psi(z)})(A)x : H \to X
\]
In this way,  the problem of a square function estimate turns into the problem
of the boundedness of the operator $\Psi(A)$. As such,
it is recognised as just another instance of the  central problem of functional
calculus, namely whether applying an unbounded functional calculus to a certain
function leads to a bounded operator or not. The only 
difference now is that the functions we are considering have to be
$H$-valued, and one should think of a functional calculus
as a module rather than an algebra homomorphism.

In this view, the classical line of research on the connection of square function estimates
and the boundedness of a certain (usually the $\Ha^\infty$-) functional calculus
changes its face, and the so far differently looking theorems become just 
instances of one type: namely how certain square function estimates
imply other square function estimates. 

\smallskip

In the present paper we have analysed
theorems of this type and could reduce them to {\em three basic principles}.
The first one is {\em subordination}, by which we mean 
that the square functions are connected via a bounded  operator
between  the underlying Hilbert spaces. 
The second is an abstract version of
how square function estimates can be proved via 
{\em integral representation} theorems. Here a deep result from the
theory of $\gamma$-radonifying operators plays a central role. 
The third one 
is based on a new (natural, but still rather enigmatic)
boundedness concept for subsets of Hilbert spaces, the so-called
{\em \ELLONE-frame-boundedness.} Basically, it asserts that 
if the $H$-valued function  $\Psi$ as above has \ELLONE-frame-bounded
range, then the associated square function $\Psi(A)$ 
is bounded (Theorem~\ref{lob.t.l1-sfe}).

The abstract results are then applied to operators of 
strip type (and hence to sectorial operators via the 
$\exp/\log$-correspondence). One of the main results here, Theorem 
\ref{exas.t:bdd-vectorial-calc}, states
that a densely defined operator (on a Banach space with finite cotype) 
with a bounded scalar $\Ha^\infty$-calculus
on a strip has a bounded vectorial $\Ha^\infty$-calculus on each larger
strip. (This result has been obtained independently of us by 
\LeMerdy\  in  \cite{LeMerdy:sqf-Ritt}.) Our proof is  based 
on a simple representation formula for holomorphic functions
on strips, see Section~\ref{ss:cauchy-gauss-repres}.

Subsequently we consider several other integral representation formulae for
analytic functions on strips 
and interpret their application in the light of our theory \
(Sections~\ref{ss:poisson-repres}--\ref{ss:sing-int-repres}).

\smallskip

By performing the twist in \eqref{intro:eq:twist}, our theory of square function
estimates becomes just a natural part of functional calculus theory. 
The technicalities involving integrals 
over vector-valued functions (like $t\mapsto \psi(tA)f$ as above)
are avoided. As a consequence (and maybe surprisingly)
the concepts of $\gamma$- or $R$-boundedness
do not appear in the present paper. We shall devote a 
future work to the detailed study of the role these concepts play
for square function estimates, and how it can be incorporated into the general
theory we develop here.

\bigskip
\noindent
{\em Notation and Terminology}\\
Banach spaces are denoted by $X,Y,Z$ and understood to be {\em complex}
unless otherwise noted.
For a closed linear operator $A$ on a complex Banach space $X$ we
denote by $\dom(A),$ $\ran(A)$, $\ker(A)$, $\sigma(A)$ and $\resolv(A)$  the
{\em domain}, the {\em range}, the {\em kernel}, the {\em spectrum} 
 and the {\em resolvent set} of $A$, respectively.
The norm-closure of the range is
written as $\cls{\ran}(A)$. The space of bounded linear operators
on $X$ is denoted by $\Lin(X)$.  For two possibly unbounded linear
operators $A,B$ on $X$ their {\em product} $AB$ is defined on
its natural domain $\dom(AB) := \{ x \in \dom(B) \suchthat Bx \in \dom(A)\}$.
An inclusion $A\subseteq B$ denotes inclusion of graphs, i.e., it means
that $B$ extends $A$. 

The inner product of two elements $f,g$ of a Hilbert space $H$ 
is generically written as $\sprod{f}{g}$ or $\sprod{f}{g}_H$. 
The duality between
a Banach space $X$ and its dual space $X'$ is denoted by 
$\dprod{\cdot}{\cdot}$ or $\dprod{\cdot}{\cdot}_{X, X'}$. 
We usually do {\em not} identify a Hilbert space $H$ with its dual space 
$H'$, except in the case that $H$ is given concretely as
$H= \Ell{2}(\Omega)$, in which case we identify $H'$ with  $H$
via the canonical duality. 

For an open subset $\gebiet \subseteq \C$ of the complex plane we
let $\Ha^\infty(\gebiet)$ be the algebra of bounded holomorphic functions
on $\gebiet$ with norm $\norm{f}_{\Ha^\infty} := \sup\{ \abs{f(z)} \suchthat 
z\in \gebiet\}$.

Unless explicitly noted  otherwise, the real line
$\R$ carries Lebesgue measure $\ud{t}$, and the
set   $(0, \infty)$ of positive reals carries the measure 
$\sfrac{\ud{t}}{t}$. We abbreviate
\[ \Ell{p}^*(0, \infty) := \Ell{p}((0, \infty); \sfrac{\ud{t}}{t}) \qquad 
(0 < p \le \infty).
\]
The {\em Fourier transform} of a function $f\in \Ell{1}(\R)$ is
\[ \Fourier(f)(t) = \fourier{f}(t) = \int_\R f(s) \ue^{\ui s t} \, \ud{s}
\qquad (t\in \R).
\]
The {\em inverse Fourier transform} is then given  by the formula
\[ (\Fourier^{-1}g)(s) = g^\vee(s) = \frac{1}{2\pi} \int_\R g(t)
\ue^{\ui s t}\, \ud{t} \qquad(s\in \R) 
\]
for $g \in \Ell{1}(\R)$.

\section{$\gamma$-Summing and $\gamma$-radonifying operators}
\label{s.gamma-radonifying} 

In this chapter we review and develop the theory of $\gamma$-summing
and $\gamma$-radonifying operators to an extent that serves our
purposes. At many places we shall simply refer to the excellent recent
article \cite{Nee2010} of van Neerven that contains also historical
remarks and an extensive bibliography on the topic. We include proofs
either for convenience or when we deviate from or go beyond van
Neerven's work.

Essentially, all presented results in this chapter are from or
inspired by the unpublished and actually never completed preprint
\cite{KalWei2004} by Kalton and Weis. Our contribution consists mostly
in presenting the results with full and concise proofs, and we give
full credits to these authors for the results themselves.

However, we want to stress the fact that whereas the two mentioned
works deal exclusively with {\em real} Banach spaces, we develop the
theory for complex spaces. The reason is that we are interested in
functional calculus questions, where contour integrals are ubiquitous.
For the theory we need the notion of a {\em complex} standard Gaussian
random variable, by which we mean a random variable $\gamma$ of the
form
\[ 
   \gamma = \gamma_r + \ui\,\gamma_i
\]
where $\gamma_r$ and $\gamma_i$ are independent standard {\em real}
Gaussians. Basically, the whole theory for real spaces carries over to
complex spaces when real Gaussians are replaced by complex ones.

\subsection{Definition and the ideal property}\label{gamma.ss.def}

Let $H$ be a Hilbert space and $X$ a Banach space over the scalar
field $\K\in \{\R,\C\}$. A linear operator $T: H \to X$ is called {\em
  $\gamma$-summing} if
\[ 
\norm{T}_{\gamma} := \sup_F \, \Exp\left(\norm{\Sum{e\in F} \gamma_e
    \tensor Te}_X^2\right)^{\einhalb} < \infty,
\]
where the supremum is taken over all finite orthonormal systems $F
\subseteq H$ and $(\gamma_e)_{e\in F}$ is an independent collection of
$\K$-valued standard Gaussian random variables on some probability
space. We let
\[ 
   \gamma_\infty(H;X) := \{ T: H \pfeil X \suchthat \text{$T$ is
  $\gamma$-summing}\}
\]
the space of $\gamma$-summing operators of $H$ into $X$. It is clear
that each $\gamma$-summing operator is bounded with $\norm{T} \le
\norm{T}_{\gamma}$.

\begin{rem}[Real vs.~Complex]\label{gamma.r.rc1}
  In the case $\K = \C$ we can view the complex spaces $H,X$ as real
  spaces, and we shall indicate this by writing $H_r, X_r$.  Then
  $H_r$ is a real Hilbert space with respect to the scalar product
  $\sprod{f}{g}_r := \re\sprod{f}{g}$. For $\C$-linear $T: H \to X$ we
  now have two definitions of $\norm{T}_\gamma$, one using
  $\dprod{\cdot}{\cdot}_r$-orthonormal systems (called $\R$-ons's for
  short) and real Gaussians, and the other using $\C$-orthonormal
  systems and complex Gaussians. We claim that both definitions lead
  to the same quantity. In particular, one has
\[ 
   \gamma_\infty(H;X) = \gamma_\infty(H_r;X_r)\cap \Lin(H;X).
\]
In order to see this we note first that if $\{ e_1, \dots, e_d\}$ is a
$\C$-orthonormal system, then $\{ e_1,\dots, e_d, \ui e_1, \dots, \ui
e_d\}$ is an $\R$-ons. Hence, if $\tilde{\gamma}_j = \gamma_j +
\ui\gamma_j'$ are independent complex standard Gaussians,
\[ 
    \Exp\norm{ {\sum}_{j} \tilde{\gamma}_j Te_j}^2
=   \Exp\norm{ {\sum}_{j} \gamma_j T(e_j) + \gamma_j' T(\ui e_j)}^2
\le  \norm{T}_{\gamma, \R}^2
\]
with the obvious meaning of $\norm{T}_{\gamma,\R}$.  This yields
$\norm{T}_{\gamma,\C} \le \norm{T}_{\gamma,\R}$. On the other hand,
let $\{f_1, \dots, f_n\}$ be an $\R$-ons and let $\gamma_1, \dots,
\gamma_n$ be real standard Gaussians. Pick a $\C$-ons $\{e_1, \dots,
e_n\}$ such that $f_k \in \spann_\C\{ e_1, \dots, e_n\}$ for each $k$.
Then we can find $\lambda_{kj}= a_{kj} + \ui b_{kj}$ such that
\[ 
    f_k = \sum_{j} (a_{kj} + \ui b_{kj})e_j = \sum_j a_{kj} e_j +
b_{kj}(\ui e_j) \qquad (1\le k \le n).
\]
Define the real matrices $A := (a_{kj})_{k,j}$, 
$B := (b_{kj})_{k,j}$ and  $C := [A\, B]$, as well as
$g_j := e_j$ for $1\le j \le n$ and $g_j := \ui e_j$ for $n< j \le 2n$. 
Then, by the contraction
principle (Theorem~\ref{gamma.t.contr}),
\begin{align*}
  \Exp&\norm{ {\sum}_{k=1}^n \gamma_k Tf_k}^2 = \Exp\norm{
    {\sum}_{k=1}^n \gamma_k a_{kj} T(f_k) + b_{kj} T(\ui f_k)}^2 \\ &
  = \Exp\norm{ {\sum}_{k=1}^n {\sum}_{j=1}^{2n} \gamma_k c_{kj}
    Tg_j}^2 \le \norm{C}^2 \Exp\norm{ {\sum}_{j=1}^{2n} \gamma_j
    Tg_j}^2 \\ & = \norm{C}^2 \Exp\norm{ {\sum}_{j=1}^{n} (\gamma_j +
    \ui \gamma_{n+j}) Te_j}^2 \le \norm{C}^2 \norm{T}_{\gamma,\C}^2.
\end{align*}
But $c_{kj} = \dprod{f_k}{g_j}_r$ and hence $\norm{C}\le 1$. This
yields $\norm{T}_{\gamma,\R} \le \norm{T}_{\gamma,\C}$ and concludes
the proof of the claim.
\end{rem}

The following approximation property is   \cite[Prop.~3.18]{Nee2010}.

\begin{lemma}[$\gamma$-Fatou I]\label{gamma.l.fatou1}
  Let $(T_n)_{n\ge 1}$ be a bounded sequence in $\gamma_\infty(H;X)$
  such that $T_n \to T\in \Lin(H;X)$ in the weak operator
  topology. Then $T \in \gamma_\infty(H;X)$ and
\[
   \norm{T}_\gamma \le \liminf_{n \to \infty} \norm{T_n}_{\gamma}.
\]
\end{lemma}

It is easy to see that $\gamma_\infty(H;X)$ contains all finite rank
operators.  The closure in $\gamma_\infty(H;X)$ of the space of finite
rank operators is denoted by $\gamma(H;X)$, and its elements $T \in
\gamma(H;X)$ are called {\em $\gamma$-radonifying}.

The following property is one of the cornerstones of the theory.

\begin{thm}[Ideal Property]\label{gamma.t.ideal}
  Let $Y$ be another Banach space and $K$ another Hilbert space, let
  $L : X \to Y$ and $R: K \to H$ be bounded linear operators, and let
  $T \in \gamma_{\infty}(H;X)$. Then
\[  
LTR \in \gamma_{\infty}(K;Y)\quad \text{and}\quad \quad
\norm{LTR}_{\gamma} \le \norm{L}_{\Lin(X;Y)} \norm{T}_\gamma
\norm{R}_{\Lin(K;H)}.
\]
If $T \in \gamma(H;X)$, then $LTR \in \gamma(K;Y)$.
\end{thm}

\begin{proof}
  One can handle the left-hand and the right-hand side separately, the
  first being straightforward. For the latter, pick a finite orthonormal
  system $\{e_1, \dots, e_n\}$ within $K$. Then find an orthonormal system
   $\{f_1, \dots, f_m\}$ with 
\[ 
   \spann\{ Re_1, \dots, Re_n\}  = \spann\{ f_1, \dots, f_m\}.
\]
Consequently $Re_k = \sum_{j=1}^m a_{kj}f_j$ for some scalar 
$(n{\times}m)$-matrix $A = (a_{kj})_{k j}$. Then, by Theorem~\ref{gamma.t.contr} below,
\begin{align*}
  & \Exp \norm{ {\sum}_{k=1}^n \gamma_k T R e_k}^2 
   = \Exp\norm{ {\sum}_{k=1}^n \gamma_k T {\sum}_{j=1}^m a_{kj}f_j}^2 \\ 
  & = \Exp\norm{ {\sum}_{k=1}^n {\sum}_{j =1}^m \gamma_k a_{kj} Tf_j}^2
  \le \norm{A}^2 \Exp\norm{ {\sum}_{j =1}^m \gamma_j Tf_j}^2 
  \le \norm{A}^2 \norm{T}_\gamma^2.
\end{align*}
Since $\norm{A}_{\ell_2^m\to \ell_2^n} \le \norm{R}_{K \to H}$, the
claim is proved.
\end{proof}

See \cite[Theorem~6.2]{Nee2010} for a slightly different proof.  Based on
the ideal property, we can show that in the case $\K = \C$ a
difference between the complex and real approach to $\gamma(H;X)$ is
only virtual.

\begin{rem}[Real vs Complex, again]\label{gamma.r.rc2}
  Let $H,X$ be complex spaces. We claim that
\[ 
\gamma(H;X) = \{ T \in \gamma(H_r;X_r) \suchthat \text{$T$ is
  $\C$-linear}\} = \gamma(H_r;X_r) \cap \Lin(H;X).
\]
The inclusion ``$\subseteq$'' is trivial, so suppose that $T: H \to X$
is $\C$-linear and in $\gamma(H_r;X_r)$. Then there is a sequence
$T_n$ of real-linear finite rank operators such that $\norm{T_n -
  T}_\gamma \to 0$. Define $S_nx := \einhalb (T_nx - \ui T_n(ix))$. Then each
$S_n$ is a $\C$-linear finite rank operator $\norm{S_n - T}_\gamma 
\to 0$.  To prove this we note that the operator $M: x \mapsto \ui x$ is a
linear isometry on $H_r$ commuting with $T$, whence
\[  2 \norm{S_n - T}_\gamma \le 
\norm{T_n - T}_\gamma + \norm{ M^{-1}T_nM - T}_\gamma
\le \norm{T_n - T}_\gamma \to 0
\]
by the ideal property. It follows that $T \in \gamma(H;X)$, as
claimed.
\end{rem}

One might ask whether $\gamma_\infty(H;X)$ can differ from
$\gamma(H;X)$. An example from Linde and Pietsch, reproduced in
\cite[Exa.~4.4]{Nee2010}, shows that this indeed happens if $X= \co$.
On the other hand, by a theorem of Hoffman-J{\o}rgensen and Kwapie\'n,
if $X$ does not contain $\ce_0$ then $\gamma(H;X) =
\gamma_\infty(H;X)$, see \cite[Theorem~4.3]{Nee2010}.  Although this result
was obtained for real spaces only, Remark~\ref{gamma.r.rc2} shows that
it continues to hold in the complex case.

\medskip

For later reference, we quote the following approximation results from
\cite[Corollaries~6.4 and 6.5]{Nee2010}. Their proofs are straightforward from
the ideal property.

\begin{thm}[Approximation]\label{gamma.t.approx}
  Let $H,K$ be Hilbert and $X,Y$ be Banach spaces, and let $T \in
  \gamma_\infty(H;X)$.  Then the following assertions hold: 
\begin{aufzi}
\item\label{item:gamma.t.approx.a} If $(L_\alpha)_\alpha\subseteq
  \Lin(X;Y)$ is a uniformly bounded net that converges strongly to $L
  \in \Lin(X;Y)$, then $L_\alpha T \to LT$ in $\gamma_\infty(H;Y)$.
\item\label{item:gamma.t.approx.b} If $(R_\alpha^*)_\alpha \subseteq
  \Lin(H;K)$ is a uniformly bounded net that converges strongly to
  $R^* \in \Lin(H;K)$, then $TR_\alpha \to TR$ in
  $\gamma_\infty(K;X)$.
\end{aufzi}
\end{thm}

Note that if $T \in \gamma(H;X)$ the operators $LT$ and $TR$ are
again $\gamma$-radonifying, by the ideal property.

\subsection{Fourier series and nuclear operators}\label{ss.nuk-op}

For $g \in H$ we let $\konj{g} := \sprod{\cdot}{g} \in H'$, i.e.,
\[ 
   H \pfeil H', \qquad g \tpfeil \konj{g} = \sprod{\cdot}{g} 
\] 
is the canonical (conjugate-linear) bijection of $H$ onto its dual.
The definition
\begin{equation}\label{gamma.eq.dual-hil-sp1}
 \sprod{\konj{g}}{\konj{h}}_{H'} := \sprod{h}{g}_H \qquad (g,h\in H)
\end{equation}
turns $H'$ canonically into a Hilbert space, and a short computation
yields $\konj{\konj{g}} = g$ under the canonical identification $H =
H''$.  Moreover, \eqref{gamma.eq.dual-hil-sp1} becomes
\begin{equation}\label{gamma.eq.dual-hil-sp2}
  \sprod{\konj{x}}{\konj{y}}_H = \sprod{y}{x}_{H'}   \qquad (x,y\in H').
\end{equation}
If $H = \Ell{2}(\Omega) = \Ell{2}(\Omega;\K)$ for some measure space
$(\Omega,\Sigma,\mu)$, we can identify $H' = \Ell{2}(\Omega)$ via the
duality
\begin{equation}\label{gamma.eq.L2-dual}
  H \times H \pfeil \K,\qquad (h,g) \tpfeil \dprod{h}{g} := \int_\Omega
  h \cdot g \, \ud{\mu} \qquad (h,g \in \Ell{2}(\Omega)).
\end{equation}
Under this identification, the conjugate $\konj{g}$ of $g \in H$ as
defined above coincides with the usual complex conjugate of $g$ as a
function on $\Omega$.

\smallskip

Every finite rank operator $T: H \to X$ has the form
\begin{equation}\label{gamma.eq.finite-rank}
  T = \Sum[n]{j=1} \konj{g_j} \tensor x_j,
\end{equation}
and one can view $\gamma(H;X)$ as a completion of the algebraic 
tensor product $H' \tensor X$ with respect to the $\gamma$-norm.

Note that if $e_1, \dots, e_n$ is an orthonormal system in $H$, then 
$\konj{e_1}, \dots, \konj{e_n}$ is an orthonormal system in $H'$, dual to 
$\{e_1, \dots, e_n\}$ in the sense that
\[ 
\dprod{e_j}{\konj{e_k}} = \dprod{e_j}{\konj{e_k}}_{H,H'} = \delta_{jk}
\qquad (j,k = 1, \dots, n).
\]
The following shows that a ``Gaussian sum'' in a Banach space $X$ can
be regarded as a $\gamma$-norm of a finite rank operator.

\begin{lemma}\label{gamma.l.finran}
Let $g_1, \dots, g_m \in H$ be an orthonormal system in  $H$ and
  $x_1, \dots , x_m \in X$. Then
\[ 
  \norm{{\sum}_{j=1}^m \konj{g_j}\tensor x_j}_\gamma^2
= \Exp\norm{{\sum}_{j=1}^n \gamma_j x_j}_X^2.
\]
\end{lemma}

\begin{proof}
  Let $e_1, \dots, e_n$ be any finite orthonormal system in $H$ and let $T$ be
defined by \eqref{gamma.eq.finite-rank}.  Then
\[ 
    \Exp\norm{ {\sum}_{k=1}^n \gamma_k Te_k}^2
=   \Exp\norm{ {\sum}_{k=1}^n \gamma_k {\sum}_{j=1}^m \sprod{e_k}{g_j} x_j}^2
\le \Exp\norm{ {\sum}_{j=1}^m \gamma_j x_j}^2
\]
by Theorem~\ref{gamma.t.contr}, since the scalar matrix $A :=
(\sprod{e_k}{g_j})_{k, j}$ satisfies $\norm{A}\le 1$. On the other
hand, if we take $n= m$ and $e_k := g_k$, then we obtain equality.
\end{proof}

Let $(e_\alpha)_{\alpha \in A}$ be an orthonormal {\em basis} of $H$.
For a finite set $F \subseteq A$, let
\[ P_F := {\sum}_{\alpha\in F} \konj{e_\alpha} \tensor e_\alpha
\]
be the orthogonal projection onto $\spann\{ e_\alpha \suchthat
\alpha\in F\}$.  The net $(P_F)_F$ is uniformly bounded and converges
strongly to the identity on $H$. Hence, the following is a consequence
of Theorem~\ref{gamma.t.approx}, part~\ref{item:gamma.t.approx.b}.

\begin{cor}[Fourier Series]\label{gamma.c.fourier-series}
If $T \in \gamma(H;X)$ and
  $(e_\alpha)_\alpha$ is any orthonormal basis of $H$, then
\[ 
  {\sum}_{\alpha} \konj{e_\alpha} \tensor  Te_\alpha = T
\]
in the norm of $\gamma(H;X)$.
\end{cor}

It follows from Lemma~\ref{gamma.l.finran} that
\[ 
  \norm{\konj{g}\tensor x}_\gamma = \norm{g}_H \norm{x}_X
= \norm{\konj{g}}_{\konj{H}} \norm{x}_X
\]
for every $g\in H$, $x\in X$, i.e., the $\gamma$-norm is a
cross-norm. The following is an immediate consequence.
(Recall that $T$ is a nuclear operator if $T =
\sum_{n\ge 0} \konj{g_n} \tensor x_n$ for some $g_n \in H, x_n \in X$
with $\sum_{n\ge 0} \norm{g_n}_H \norm{x_n}_X < \infty$.)

\begin{cor}\label{gamma.c.nuclear}
A nuclear operator $T:H \to X$ is
$\gamma$-radonifying and $\norm{T}_\gamma \le
\norm{T}_{\nuc}$. 
\end{cor}

 The following application turns out to be quite useful.

\begin{lemma}\label{gamma.l.nuc}
  Let $H,X$ as before, and let $(\Omega,\Sigma,\mu)$ be a measure
  space.  Suppose that $f: \Omega \to H$ and $g : \Omega \to X$ are
  (strongly) $\mu$-measurable and
\[ 
   \int_\Omega \norm{f(t)}_H \norm{g(t)}_X \, \mu(\ud{t}) < \infty.
\]
Then $\konj{f} \tensor g \in \Ell{1}(\Omega; \gamma(H;X))$, and $T :=
\int_\Omega \konj{f} \tensor g \,\ud{\mu} \in \gamma(H;X)$ satisfies
\[ 
   Th = \int_\Omega \sprod{h}{f(t)} g(t)\, \mu(\ud{t}) \qquad (h \in H)
\]
and
\[ 
  \norm{T}_\gamma \le \int_\Omega \norm{f(t)}_H \norm{g(t)}_X \, \mu(\ud{t}).
\]
\end{lemma}

\subsection{Trace duality}\label{gamma.ss.dual}

We follow \cite{BlaTarVid2000,KalWei2004} and identify the dual of
$\gamma(H;X)$ with a subspace of $\Lin(H';X')$ via {\em trace
  duality}.  For a finite rank operator $U: H \to H$ given by
\[
 U := {\sum}_{j=1}^n g_j' \tensor h_j
\]
for certain $g_1',\dots, g_n' \in H'$ and $h_1, \dots, h_n\in H$,
its {\em trace} is
\[
  \tr(U) = {\sum}_{j=1}^n  \dprod{h_j}{g_j'}. 
\] 
This is independent of the representation of $U$, see
\cite[p.~125]{DieJarTon1995}.  Now, for $V\in \Lin(H';X')$ we define
\[ 
\norm{V}_{\gamma'} 
:= \sup \Big\{ \abs{\tr(V'U)} \suchthat U \in \Lin(H;X),\, \norm{U}_\gamma \le 1,\,
\dim \ran(U) < \infty\Big\},
\]
where we regard $V'U : H \to X \subseteq X'' \to H'' = H$, and let
\[ 
\gamma'(H';X') := \{ V\in \Lin(H';X') \suchthat 
\norm{V}_{\gamma'} < \infty\}.
\]
By a short computation, if $U \in \Lin(H;X)$ has the representation $U
= \sum_{j=1}^n g_j'\tensor x_j$ and $V \in \Lin(H';X')$, then
\begin{equation}\label{gamma.eq.trace}
 \tr(V'U) = \sum_{j=1}^n \dprod{x_j}{ Vg_j'}.
\end{equation}

\begin{lemma}[$\gamma'$-Fatou]
  Let $(V_n)_n$ be a bounded sequence in $\gamma'(H';X')$ and let $V :
  H'\to X'$ be such that $\dprod{x}{V_nh'} \to \dprod{x}{Vh'}$ for all
  $x\in X$ and $h'\in H'$. Then $V \in \gamma'(H';X')$ and
\[ 
  \norm{V}_{\gamma'} \le \liminf_{n \to \infty} \norm{V_n}_{\gamma'}
\]
\end{lemma}

\begin{proof}
  It follows from \eqref{gamma.eq.trace} that $\tr(V_n'U) \to
  \tr(V'U)$ for every $U: H \to X$ of finite rank. The claim follows.
\end{proof}

We now turn to an alternative description of the $\gamma'$-norm.  To
this end we note the following auxiliary result. We let  $\Nuc(H)$
denote the class of {\em nuclear operators} on $H$, also called operators
of {\em trace class}, with its natural norm $\norm{\cdot}_\nuc$.

\begin{lemma}\label{gamma.l.dual-nuc}
  If $T \in \Nuc(H)$ then $\tr(TA) = \norm{T}_\nuc$ for some $A\in
  \Lin(H)$, $\norm{A}\le 1$.
\end{lemma}

\begin{proof}
  By a standard result of Hilbert space operator theory, $T$ has the
  representation
\[ 
  T = {\sum}_{j\in J} s_j \konj{e_j} \tensor f_j
\]
where $J$ is either finite or $J =\N$, the $e_j$ as well as the $f_j$
form orthonormal systems, and the numbers $s_j > 0$ are the singular
values of $T$. Define $A := \sum_{j \in J} \konj{f_j} \tensor e_j$,
where in case $J = \N$ the series converges strongly. Then $\norm{A}
\le 1$ and $TA = {\sum}_{j\in J} s_j \konj{A^*e_j} \tensor f_j$.
Hence
\[ 
   \tr(TA) = \sum_{j\in J} s_j \sprod{f_j}{A^*e_j} = \sum_{j\in J} s_j =
\norm{T}_\nuc.
\]
\end{proof}

As a consequence we arrive at the following characterisation of the
$\gamma'$-norm.

\begin{cor}\label{gamma.c.dual-nuc}
Let $V \in \Lin(H';X')$. Then
\[ 
  \norm{V}_{\gamma'} = 
   \sup \Big\{ \norm{V'U}_{\nuc} \suchthat U \in \Lin(H;X),\,
   \norm{U}_\gamma \le 1,\, \dim \ran(U) < \infty\Big\}.
\]
\end{cor}

\begin{proof}
  Let $U:H \to X$ be of finite rank with $\norm{U}_\gamma \le
  1$. Then $\abs{\tr(V'U)} \le \norm{V'U}_\nuc$. On the other hand, by
  applying Lemma~\ref{gamma.l.dual-nuc} to $T := V'U$ we find $A\in
  \Lin(H)$ with $\norm{A}\le 1$ and
\[ 
    \norm{V'U}_\nuc = \tr(V'UA)  \le  \norm{V}_{\gamma'} \norm{UA}_\gamma
\le \norm{V}_{\gamma'} \norm{U}_\gamma \norm{A} 
\le \norm{V}_{\gamma'}
\]
by the ideal property.
\end{proof}

As a consequence of Corollary~\ref{gamma.c.dual-nuc} we obtain the
ideal property of $\gamma'(H';X')$.

\begin{cor}[Ideal Property]\label{gamma.c.dual-ideal}
  Let $R:H \to K$ and $L: Y \to X$ be bounded operators, and $V \in
  \gamma'(H';X')$. Then $L'VR' \in \gamma'(K';Y')$ with
\[ 
  \norm{L'VR'}_{\gamma'} \le \norm{L} \norm{V}_{\gamma'} \norm{R}.
\]
\end{cor}

\begin{proof}
  Let $U: K \to Y$ be of finite rank. Then
\begin{align*}
  \norm{(L'VR')'U}_\nuc & = \norm{R V'(L''U)}_\nuc \le \norm{R}
  \norm{V'(LU)}_\nuc \le \norm{R} \norm{V'}_{\gamma'} \norm{LU}_\gamma
  \\ & \le \norm{R} \norm{V'}_{\gamma'} \norm{L} \norm{U}_\gamma
\end{align*}
by the ideal property of  $\Nuc(K)$ and $\gamma(K;Y)$.
\end{proof}

With the following results we extend \cite[Prop.~5.1 and
5.2]{KalWei2004}.

\begin{thm}\label{gamma.t.dual}
\begin{aufzi}
\item\label{item:gamma.t.dual.a} If $U \in \gamma(H;X)$ and $V \in
  \gamma'(H';X')$, then $V'U \in \Nuc(H)$ with $\norm{V'U}_\nuc \le
  \norm{V}_{\gamma'}\norm{U}_\gamma$.  Moreover, the mapping
\[ \gamma'(H';X') \pfeil \Lin\big(\gamma(H;X); \Nuc(H)\big), 
\qquad V \tpfeil ( U \tpfeil V'U)
\]
is isometric.

\item\label{item:gamma.t.dual.b} The bilinear mapping (``trace
  duality'')
  \[ \gamma(H;X) \times \gamma'(H';X') \pfeil \C,\qquad (U,V) \tpfeil
  \dprod{U}{V} := \tr(V'U)
\]
establishes an isometric isomorphism $\gamma(H;X)' \cong \gamma'(H';X')$.

\item\label{item:gamma.t.dual.d} Let $(e_\alpha)_\alpha$ be an orthonormal basis of $H$. Then 
\[   \dprod{U}{V} = \tr(V'U) = \sum_{\alpha}  \dprod{Ue_\alpha}{V\konj{e_\alpha}}_{X,X'}
\]
for every $U \in \gamma(H;X)$ and $V \in \gamma'(H';X')$.

\item\label{item:gamma.t.dual.c} If $V \in \gamma(H';X')$ then $V \in
  \gamma'(H';X')$, with $\norm{V}_{\gamma'} \le \norm{V}_\gamma$.
\end{aufzi}
\end{thm}

\begin{proof}
 ~\ref{item:gamma.t.dual.a}~follows from
  Corollary~\ref{gamma.c.dual-nuc} and approximation of a general $U
  \in \gamma(H;X)$ by finite rank operators.

\noindent
 ~\ref{item:gamma.t.dual.b} By~\ref{item:gamma.t.dual.a} the trace
  duality is well defined, and it reproduces the norm on
  $\gamma'(H';X')$ by construction. For surjectivity, let $\Lambda :
  \gamma(H;X) \to \C$ be a bounded functional and define
\[ 
  V : H'\pfeil X',  \qquad 
(Vh')(x) := \Lambda(h'\tensor x).
\]
A short computation reveals that $\tr(V'U) = \Lambda(U)$ for every
rank-one operator $U = h'\tensor x$. Hence $\tr(V'U) = \Lambda(U)$
even for every finite rank-operator $U: H \to X$.  But this implies
that $V \in \gamma'(H';X')$ and that $V$ induces $\Lambda$.

\noindent
\ref{item:gamma.t.dual.d} By Corollary~\ref{gamma.c.fourier-series},
$U = \sum_\alpha \konj{e_\alpha} \tensor Ue_\alpha$ and the convergence
is in $\norm{\cdot}_\gamma$. Hence
\[ \dprod{U}{V} = \sum_\alpha \dprod{\konj{e_\alpha} \tensor Ue_\alpha}{V}
= \sum_\alpha \dprod{Ue_\alpha}{V\konj{e_\alpha}}_{X,X'}
\]
by \eqref{gamma.eq.trace}.

\noindent
\ref{item:gamma.t.dual.c}~is proved as in \cite[Theorem~10.9]{Nee2010}.
\end{proof}

\begin{rem}
It is shown in \cite[Sec.~10]{Nee2010} that equality $\gamma(H';X') =
\gamma'(H';X')$ holds if $X$ is $K$-convex.  By a result of Pisier,
a space $X$ is $K$-convex if and only if it has nontrivial type. See
\cite[Sec.~10]{Nee2010} for more about $K$-convexity in this context.
\end{rem}

\subsection{Spaces of finite cotype} 
\label{ss:finite-cotype}

A {\em Rademacher} variable is a $\pm 1$-valued Bernoulli-$(\einhalb, \einhalb)$
random variable. A {\em complex Rademacher} variable is a random
variable of the form
\[ 
   r = r_1 + \ui r_2
\]
where $r_1, r_2$ are independent real Rademachers on the same
probability space.  Unless otherwise stated, our Rademacher variables
are understood to be complex.

By \cite[Prop.~2.6]{Nee2010} (see also \cite[Lemma~12.11]{DieJarTon1995}) 
\begin{equation}\label{gamma.eq.r-to-gamma}
  \Exp \Big\|{\sum}_{j=1}^n r_j x_j \Big\|_X^q \le \, (\pihalbe)^{\sfrac{q}2} \,  
\Exp \Big\|{\sum}_{j=1}^n \gamma_j x_j \Big\|_X^q, 
\end{equation}
whenever $ 1\le q < \infty$, $n \in \N$, $x_1, \dots, x_n \in X$,
$r_1, \dots, r_n$ are complex Rademachers and $\gamma_1, \dots,
\gamma_n$ are complex Gaussians.  (Our reference uses real random
variables, but the complex case follows by a straightforward argument,
yielding the same constant.)
 
A converse estimate does not hold in general unless the Banach space
has finite cotype. Recall that a Banach space $X$ has \emph{type
  $p\in[1,2]$} if there exists a constant $\tp_p(X)\geq0$ such that
for all finite sequences $(x_n)_{n=1}^m$ in $X$,
\[
   \Bignorm{ {\sum}_n r_n x_n }_{\Ell{2}(\Omega; X)}   \le \tp_p(X) \,\,
 \norm{ (x_n)_n}_{\ell_p(X)},
\]
and $X$ has \emph{cotype $q\in[2,\infty]$} if for some constant
$\ct_q(X)\geq0$,
\[
  \norm{ (x_n)_n }_{\ell_q(E)} \le \ct_q(X)\,\,  \Bignorm{ {\sum}_n r_n x_n }_{\Ell{2}(\Omega, E)},
\]
We refer to \cite[Chapter 11]{DieJarTon1995} for definitions, properties
and references on the notions of type and cotype of a Banach space.
(Using real in place of complex Rademachers may alter the values of
$\tp_p(X)$ and $\ct_q(X)$ by universal factors, but does not make a
qualitative difference.)

Each Banach spaces has cotype $\infty$ and type $1$; therefore, $X$ is
said to have \emph{nontrivial type} if it has type $p$ for some $p>1$,
and it said to have \emph{finite cotype} if it has cotype $q$ for some
$q < \infty$. Each Banach space of nontrivial type has finite cotype,
but the converse is false.

\smallskip

It is important for us that if $X$ has finite cotype, then a converse
to \eqref{gamma.eq.r-to-gamma} holds.  Namely, we have the following
deep result from \cite[Theorem~12.27]{DieJarTon1995}.

\begin{thm}\label{gamma.t.gam-vs-rad}
Let $2 \le q < \infty$. Then there is a universal constant $m_q>0$ such that
\[ 
   \Exp\Bignorm{ {\sum}_{j=1}^n \gamma_j x_j }^2 
\le m_q^2 \,\ct_q(X)^2 \, \Exp\Bignorm{ {\sum}_{j=1}^n r_j x_j }^2
\]
whenever $X$ is a Banach space of cotype $q$ and $x_1, \dots, x_n \in
X$.
\end{thm}

The following observation is needed in the proof of Theorem~\ref{acsf.t.nsfo}
below.

\begin{lemma}\label{gamma.l.cotyp}
A Banach space $X$ has the same  type and cotype as $\gamma(H; X)$.
\end{lemma}

\begin{proof}
  We show the result only for the case of cotype.  For the type case
  the arguments are similar.  Suppose first that $X$ has cotype
  $q<\infty$, and let $(U_k)_k$ be a finite sequence in $\gamma(H;
  X)$. Fix an orthonormal basis $(e_\alpha)_\alpha$ of $H$.  Then $U_k
  = \sum_{\alpha} \konj{e_\alpha} \tensor U_ke_\alpha$ for each $k$ by
  Corollary~\ref{gamma.c.fourier-series}. Hence, with $F$ denoting
  finite subsets of the index set of the orthonormal basis,
\begin{align*}
{\sum}_{k} \norm{U_k}_{\gamma}^q 
& = {\sum}_k \lim_{F} \Bignorm{{\sum}_{\alpha\in F} \konj{e_\alpha}
  \tensor U_ke_\alpha }_\gamma^q
  =  \lim_{F} {\sum}_k \Bignorm{{\sum}_{\alpha \in F} \konj{e_\alpha}
    \tensor U_ke_\alpha }_\gamma^q\\ 
& \lesssim \sup_{F} {\sum}_k   \Exp' \Bignorm{{\sum}_{\alpha \in F}
  \gamma'_\alpha U_ke_\alpha }_X^q
  = \sup_{F} \Exp' {\sum}_k \Bignorm{{\sum}_{\alpha \in F}
    \gamma'_\alpha U_ke_\alpha }_X^q\\ 
& \lesssim \sup_{F}\,  \ct_q(X)^q\,  \Exp'\, \Exp \Bignorm{  {\sum}_k
  r_k  \Big({\sum}_{\alpha \in F} \gamma'_\alpha
  U_ke_\alpha\Big)}_X^q\\ 
& \lesssim \sup_{F}\,  \ct_q(X)^{2q}\, m^q \, m_q^q \,  \Exp'\,\Exp
\Bignorm{  {\sum}_k \gamma_k  \Big({\sum}_{\alpha \in F}
  \gamma'_\alpha U_ke_\alpha\Big) }_X^q\\ 
& = \sup_{F}\,  \ct_q(X)^{2q}\, m^q\, m_q^q\,  \Exp\,\Exp'
\Bignorm{{\sum}_{\alpha \in F}  \gamma_\alpha'  \Big({\sum}_k \gamma_k
  U_k\Big) e_\alpha }_X^q\\ 
& \lesssim \ct_q(X)^{2q}\, m^q \,m_q^q\,   \Exp\, \Bignorm{{\sum}_k
  \gamma_k U_k}_\gamma^q \\ 
& \lesssim \ct_q(X)^{2q}\, m^{2q} \,m_q^q\,   \Exp\,\Bignorm{{\sum}_k
  r_k U_k}_\gamma^q, 
\end{align*}
where $m := \sqrt{\pihalbe}$ and the non-mentioned constants come from
the Khinchine--Kahane inequalities.  It follows that
\[ 
\norm{ (U_k)_k }_{\ell_q(\gamma(H;X))} \lesssim \ct_q(X)^2\, m^{2}
\,m_q\, \Bignorm{{\sum}_k r_k U_k}_{\Ell{2}(\Omega;\gamma(H;X))}
\]
and this shows that $\ct_q(\gamma(H;X)) \lesssim \ct_q(X)^2\, m^{2} \,m_q$.

\smallskip
\noindent
For the converse  suppose
that $\gamma(H; X)$ has cotype $q<\infty$. Let $(x_k)_k$ be a finite
sequence in $X$ and let $e\in H$ be a unit vector. Abbreviate $E := \gamma(H;X)$
and $U_k := \konj{e}\tensor x_k$.  Then 
\begin{align*}
\Big({\sum}_k \norm{x_k}_X^q\Big)^{1/q} = \Big({\sum}_k \norm{U_k}_E^q\Big)^{1/q}
\le \ct_q(E) \Bignorm{{\sum}_k r_k U_k}_{\Ell{2}(\Omega;E)}. 
\end{align*}
Moreover,
\begin{align*}
\Bignorm{{\sum}_k r_k U_k}_{\Ell{2}(\Omega;E)}^2
= \Exp \Bignorm{{\sum}_k r_k \konj{e}\tensor x_k}_E^2
= \Exp \Bignorm{ \konj{e}\tensor \Big({\sum}_k r_k  x_k\Big)}_E^2
= \Exp \Bignorm{{\sum}_k r_k  x_k}_X^2,
\end{align*}
whence it follows that $\ct_q(X) \le \ct_q(E)$.
\end{proof}

The next result shows the significance of spaces of finite cotype 
for the theory of $\gamma$-radonifying operators.

\begin{thm}\label{gamma.t.kw-ess}
  Let $X$ be a Banach space of finite cotype $q < \infty$.  There is a
  constant $c = c(q, \ct_q(X))$ such that the following holds:
  Whenever $K$ is a compact Hausdorff space, $H$ is a Hilbert space
  and $T \in \Lin(H;X)$ is an operator that factorises as $T = UV$
  over $\Ce(K)$, i.e.,
\[
\xymatrix{
H \ar[dr]_V\ar[rr]^T &  & X\\
& \Ce(K)\ar[ur]_U
},
\]
then $T \in \gamma(H;K)$ and $\norm{T}_{\gamma(H;X)} \le c \norm{U} \norm{V}$. 
\end{thm}

\begin{proof}
  Let $X$ be of cotype $2\le q <\infty$ and fix $q < p < \infty$. By
  \cite[Theorem~11.14]{DieJarTon1995} the operator $U$ is $p$-absolutely
  summing, and one has $\pi_p(U) \le c \cdot \norm{U}$, where $c$
  depends on $p$ and $\ct_q(X)$.  By the ideal property for
  $p$-absolutely summing operators, $T$ is $p$-absolutely summing with
  $\pi_p(T) \le\pi_p(U) \norm{V}$.  Now, a theorem of Linde and
  Pietsch \cite[12.1]{Nee2010} \cite{LinPie1974} yields that $T \in
  \gamma(H;X)$ with $\norm{T}_\gamma \le \max\{K_{2,p}^\gamma,
  K_{p,2}^\gamma\} \pi_p(T)$. Here $K_{p,2}^\gamma$ and
  $K_{2,p}^\gamma$ are the constants in the Khinchine--Kahane
  inequalities for Gaussians, see \cite[Prop.2.7]{Nee2010}.  By taking
  the infimum over $p$ we remove the dependence of the constant on
  $p$.
\end{proof}

\subsection{The space  $\gamma(\Omega;X)$}

We now consider the case that $H = \Ell{2}(\Omega)$ for some
measure space $(\Omega,\Sigma,\mu)$.
For a $\mu$-measurable function $f: \Omega \to X$ we define 
\begin{align*}
  \Sigma_f & := \{ A\in \Sigma \suchthat \car_A f\in
  \Ell{2}(\Omega;X)\},\qquad \De_f := \{ \car_A g \suchthat A\in
  \Sigma_f, g\in \Ell{2}(\Omega)\}
\end{align*}
and the operator
\[ 
  \Jei_f : \De_f \pfeil X,\qquad \Jei_f(h) := \int_\Omega hf\, \ud{\mu}.
\]
We have collected some general facts about this construction in
Appendix~\ref{app.P2}.  There it is shown that $\De_f$ is dense in
$\Ell{2}(\Omega)$, and that $\Jei_f$ extends to a bounded operator
(denoted also by $\extend{\Jei}_f$) on the whole of $\Ell{2}(\Omega)$ if and
only if $f \in \Pe_2(\Omega;X)$, the space of weakly
$\Ell{2}$-functions. In this case, for any $h \in \Ell{2}(\Omega)$
the value $\Jei_f(h)$ is the {\em Pettis integral} of $hf$, i.e., it satisfies
\[ \dprod{\Jei_f(h)}{x'} = \int_\Omega h  \, (x'\nach f) \, \ud{\mu}.
\]
Now we let
\[ \gamma(\Omega;X) := \{ f\in \Pe_2(\Omega) \suchthat \extend{\Jei}_f \in 
\gamma(\Ell{2}(\Omega);X)\}
\]
and define $\gamma_\infty(\Omega;X)$ similarly. We abbreviate
$\norm{f}_\gamma := \norm{\Jei_f}_{\gamma}$ for $f\in \gamma_\infty(\Omega; X)$.

\smallskip

There is a $\gamma$-analogue of Lemma~\ref{P2.l.fatou}.  In fact, it
follows directly from that result and Lemma~\ref{gamma.l.fatou1}.

\begin{lemma}[$\gamma$-Fatou II]\label{gamma.l.fatou2}
  If $(f_n)_{n \in \N}$ is a bounded sequence in
  $\gamma_\infty(\Omega;X)$ with $f_n \to f$ almost everywhere, then
  $f\in \gamma_\infty(\Omega;X)$, $\extend{\Jei}_{f_n} \to \extend{\Jei}_f$
  strongly, and
\[ 
   \norm{f}_{\gamma} \le \liminf_{n\to \infty} \norm{f_n}_{\gamma}.
\]
\end{lemma}

The spaces $\gamma(\Omega;X)$ and $\Pe_2(\Omega)$ are not
complete in general. This is different with
\[ 
  \gamma_2(\Omega;X) := \Ell{2}(\Omega;X) \cap \gamma(\Omega;X),
\]
which is a Banach space with respect to the norm
$\norm{f}_{\gamma_2}:= \norm{f}_{\Ell{2}} + \norm{\Jei_f}_\gamma$. Recall
that
\[ 
  \spann\{ \car_A \tensor x \suchthat A\in \Sigma,\,\,\mu(A) < \infty,\,\,
x\in X\}
\]
is called the space of ($X$-valued) {\em step functions}.

\begin{lemma}\label{gamma.l.gamma2}
  The space of step functions is dense in $\gamma_2(\Omega;X)$, i.e.,
  whenever $f\in \Ell{2}(\Omega;X)$ such that $\Jei_f\in
  \gamma(\Ell{2}(\Omega);X)$, then there is a sequence $(f_n)_n$ of
  $X$-valued step functions such that $\norm{f_n -f }_2 \to 0$ and
  $\norm{\Jei_{f_n} - \Jei_f}_\gamma \to 0$.
\end{lemma}

\begin{proof}
  Approximate $f$ in $\Ell{2}$ by $f_n := \Exp(f|\Sigma_n)$ where
  $\Sigma_n$ is a finite sub-$\sigma$-algebra of $\Sigma$. (Note that
  $f$ is essentially measurable with respect to a countably generated
  sub-$\sigma$-algebra of $\Sigma$.)  It follows from
  Theorem~\ref{gamma.t.ideal} that $\norm{\Jei_{f_n} - \Jei_f}_\gamma \to
  0$.
\end{proof}

For a general $f\in \gamma(\Omega;X)$ we still have the following
approximation method. 

\begin{lemma}\label{gamma.l.P2-approx}
  Let $f \in \gamma( \Omega;X)$ and let $(A_n)_n \subseteq \Sigma_f$
  with $\car_{A_n} \nearrow \car$ almost everywhere on $\{f \neq 0\}$.  Then $f
  \car_{A_n} \in \gamma_2(\Omega;X)$, $\norm{f\car_{A_n}}_\gamma \le
  \norm{f}_\gamma$ and $\norm{f - f\car_{A_n}}_\gamma \to 0$.
\end{lemma}

Note that a sequence $(A_n)_n$ as considered in
the lemma exists by Lemma~\ref{P2.l.locbd}.

\begin{thm}\label{gamma.t.P2-approx}
For a $\mu$-measurable function  $f: \Omega \to X$ the following assertions
are equivalent:
\begin{aufzii}
\item $f\in \gamma(\Omega;X)$.
\item There is a $\norm{\cdot}_\gamma$-Cauchy 
sequence $(f_n)_n$ of $X$-valued step functions
with $f_n \to f$ almost everywhere.
\end{aufzii}
Moreover, $\norm{f_n - f}_\gamma \to 0$ for each such sequence as in
{\upshape (ii)}.
\end{thm}

\begin{proof}
(ii)$\dann$(i):\ By the $\gamma$-Fatou Lemma~\ref{gamma.l.fatou2}, 
$f \in \gamma_\infty(\Omega;X)$ and $\Jei_{f_n} \to \Jei_f$ strongly. 
Since $\gamma(\Ell{2}(\Omega);X)$ is complete, there is
$T \in  \gamma(\Ell{2}(\Omega);X)$ such that $\norm{\Jei_{f_n} - T}_\gamma\to 0$.
This implies that $\Jei_f = T$, whence 
$f\in \gamma(\Omega;X)$ and $\norm{f_n -f}_\gamma\to 0$. 

\smallskip
\noindent
(i)$\dann$(ii):\ By Lemma~\ref{P2.l.locbd} and 
Lemma~\ref{gamma.l.P2-approx} we can find
$A_n \in \Sigma_f$, $\mu(A_n) < \infty$, $\norm{f - f\car_{A_n}}_\gamma\le 1/n$ 
and $\car_{A_n} \nearrow \car_{\{f\neq 0\}}$ outside a null set $M$, say. 
Now let $n \in \N$ be fixed. Then by Lemma~\ref{gamma.l.gamma2}
we can approximate $f\car_{A_n}$ in the norm $\norm{\cdot}_{\Ell{2}}
+ \norm{\cdot}_\gamma$ by a sequence of step functions. Without loss
of generality we may suppose that these step functions vanish on $A_n^c$. 
Passing to a subsequence we may suppose in addition that the convergence
is even pointwise almost everywhere. By a variant of Egoroff's theorem,
the convergence is almost uniform, i.e., there is a step function
$f_n$ such that $\{ f_n \neq 0\} \subseteq A_n$ and
$\norm{f\car_{A_n} - f_n}_\gamma < 1/n$, and there is  
a set $B_n \subseteq A_n$ with $\mu(A_n \ohne B_n)\le 2^{-n}$ and 
$\norm{f_n(\omega) - f(\omega)}_X\le 1/n$ for $\omega \in B_n$.

By construction $f_n$ is a step function and $\norm{f_n - f}_\gamma \le 2/n 
\to 0$. To show that $f_n\to f$ almost everywhere, we form the set
$N := \bigcap_{k \in \N} \bigcup_{n\ge k} A_n \ohne B_n$, which is a null set. 
Let  $x\notin N\cup M$. 
The there is $k \in \N$ such that $x \in B_n \cup A_n^c$ for all
$n \ge k$. But for large $n$ either $f(x) = f_n(x) = 0$ or we have $x \in A_n$, and hence
$x\in B_n$. But that means that $\norm{f_n(x) - f(x)}_X \le 1/n$ for large
$n \in \N$. 
\end{proof}

\subsection{The space $\gamma'(\Omega;X')$}

Again, let  $H = \Ell{2}(\Omega)$, $(\Omega, \Sigma, \mu)$ 
any measure space. We identify $H = H'$ via the duality
\eqref{gamma.eq.L2-dual}. We let
\[ 
\Pe_2'(\Omega;X') := \{ g : \Omega \to X' \suchthat 
 \dprod{x}{g(\cdot)} \in \Ell{2}(\Omega)\,\,\,\text{for every $x\in X$}\}.
\]
The closed graph theorem shows that if $g\in \Pe_2'(\Omega;X')$ then
there is $C\ge 0$ such that
\[ 
\left( \int_\Omega \abs{\dprod{x}{g(\omega)}}^2 \, \mu(\ud{\omega}) \right)^{\einhalb}
= \norm{  \dprod{x}{g(\cdot)} }_{\Ell{2}(\Omega)} 
\le C \norm{x} \qquad (x\in X).
\]
Hence, the mapping
\[ 
V_g: \Ell{2}(\Omega) \pfeil X',\qquad (V_gh)(x) := 
\int_\Omega h(\omega) \dprod{x}{g(\omega)} \, \mu(\ud{\omega})
\]
is a well defined bounded operator with norm
\[ 
\norm{V_g} = \sup_{\norm{h}_2 \le 1} \norm{V_gh}_{X'} = 
\sup_{\norm{x}\le 1} \norm{\dprod{x}{g(\cdot)}}_{\Ell{2}}.
=: \norm{g}_{\Pe'_2}
\] 
The following is a dual analogue of Lemma~\ref{P2.l.fatou}.

\begin{lemma}[$\Pe_2'$-Fatou]\label{gamma.l.P2'-fatou}
Let $(g_n)_{n\in \N}$ be a bounded sequence in $\Pe_2'(\Omega;X')$ with $\dprod{x}{g_n(\cdot)}  \to \dprod{x}{g(\cdot)}$ 
almost everywhere for every $x\in X$, then $f \in \Pe_2'(\Omega;X')$, 
$\norm{f_n}_{\Pe_2'} \le \liminf_{n \to \infty} \norm{f_n}_{\Pe_2'}$ and 
$V_{g_n} \to V_g$ in the weak$^*$ operator topology. 
\end{lemma}

\begin{proof}
For $x\in X$ and $h \in \Ell{2}(\Omega)$ the usual Fatou lemma states that
\begin{align*}
 \int_\Omega & \abs{h(t) \dprod{x}{g(t)}}\, \mu(\ud{t}) \le 
\liminf_{n \to \infty} \int_\Omega \abs{h(t) \dprod{x}{g_n(t)}}\, \mu(\ud{t})
\\ & \le   \liminf_{n \to \infty} \norm{h}_{\Ell{2}} \norm{x} \norm{g_n}_{\Pe_2'}.
\end{align*}
Hence $\dprod{x}{g(\cdot)} \in \Ell{2}(\Omega)$ for every $x\in X$, i.e.,  $g \in \Pe_2'(\Omega;X')$.
Similar to the proof of Lemma~\ref{P2.l.fatou} it follows that 
$\dprod{x}{g_n (\cdot)} \to \dprod{x}{g(\cdot)}$ weakly in $\Ell{2}$ for each $x\in X$. But this is just the same
as to say that $V_{g_n} \to V_{g}$ in the weak$^*$ operator topology.
\end{proof}

We define
\[
\gamma'(\Omega;X') := \{ g\in \Pe_2'(\Omega;X') \suchthat
V_g \in  \gamma'(\Ell{2}(\Omega);X')\}
\]
and write $\norm{g}_{\gamma'} := \norm{V_g}_{\gamma'}$.  The following
result, based on \cite[Corollary~5.5]{KalWei2004}, yields a convenient way to
use the trace duality.

\begin{thm}\label{gamma.t.gamma-hoelder}
  Let $f\in \gamma(\Omega;X)$ and $g\in \gamma'(\Omega;X')$. Then
$\dprod{f(\cdot)}{g(\cdot)} \in \Ell{1}(\Omega)$ and  
\[
\int_\Omega \big|\dprod{f(\cdot)}{g(\cdot)}_{X,X'}\big| \, \ud{\mu}
\le \norm{f}_{\gamma} \norm{g}_{\gamma'}.
\]
Moreover, 
\[
\dprod{\extend{\Jei}_f}{V_g} = \tr({V_g}'\extend{\Jei}_f)
= \int_\Omega \dprod{f(\cdot)}{g(\cdot)}_{X,X'} \, \ud{\mu}.
\]
\end{thm}

\begin{proof}
By Theorem~\ref{gamma.t.P2-approx}
it suffices to prove the claim for
$f \in \Ell{2}(\Omega) \tensor X$, say $f =
  \sum_{j=1}^n f_j \tensor x_j$.  Then, by \eqref{gamma.eq.trace},
\[
\tr({V_g}'\Jei_f) =  \sum_{j=1}^n \dprod{x_j}{V_gf_j}
= \sum_{j=1}^n \int_\Omega   \dprod{x_j}{g(\cdot)} f_j \, \ud{\mu} 
= \int_\Omega  \dprod{f(\cdot)}{g(\cdot)} \,  \ud{\mu}. 
\]
For the remaining statement, find $m \in \Ell{\infty}(\Omega)$ with
$\abs{m} \le 1$ and
\begin{align*}
 \int_\Omega \abs{\dprod{f}{g}} & \, \ud{\mu} 
= \int_\Omega m \dprod{f}{g} \, \ud{\mu}
 = \tr({V_g}' \Jei_{mf}) = \abs{\tr({V_g}' \Jei_{mf})} 
\\ & \le  \norm{g}_{\gamma'} \norm{mf}_\gamma 
\le \norm{g}_{\gamma'} \norm{f}_\gamma
\end{align*}
by what has been already shown and the ideal property. 
\end{proof}

\subsection{Banach lattices}\label{ss:banach-lattices}

In this section we derive an alternative description
of the $\gamma$-norms on Banach lattices. This will  
make the name ``square function'' plausible, and will help us
relating our abstract square functions to classical ones, see 
the Introduction and Section~\ref{s.exas} below.

Let $E$ be a {\em complex Banach lattice} (we refer to
\cite{LindenstraussTzafriri2,Meyer-Nieberg,Schaefer:banach-lattices}
for background, but actually we shall not need so much of it). If one adapts
the theory developed in \cite[pp.326-329]{DieJarTon1995} to the setting
of complex Banach lattices, one obtains the following.
Whenever $u_1, \dots, u_n \in E$ then
\begin{equation} \label{eq:defi-via-sup} 
\Big( {\sum}_{j=1}^n  \abs{u_j}^2 \Big)^\einhalb 
:= \sup \Big\{ \Big| {\sum}_{j=1}^n   \alpha_j u_j \Big|\suchthat
\alpha \in \ell_2^n, \norm{\alpha}_2\le   1\Big\}
\end{equation}
exists in $E$. The notation is inspired by the formula for scalars, and 
is coherent with usual pointwise notation in Banach function spaces such as
spaces $\Ell{p}(\Omega)$. That is to say, if $E = \Ell{p}(\Omega)$ for some
measure  space $(\Omega, \Sigma,\mu)$, $1\le p \le \infty$ and 
$u_1, \ldots,u_n \in E$, then 
\[ 
\Big( {\sum}_{j=1}^n \abs{u_j}^2 \Big)^\einhalb (\omega) = 
\Big( {\sum}_{j=1}^n \abs{u_j(\omega) }^2 \Big)^\einhalb 
\]
for $\mu$-almost every $\omega \in \Omega$. (This follows since in computing
the supremum in (\ref{eq:defi-via-sup}) one can restrict to a
countable subset.)

\medskip Now let $(\Omega, \Sigma, \mu)$ be any 
measure space, and let $f \in \Pe_2(\Omega; E)$. In complete analogy
to \eqref{eq:defi-via-sup} we shall write
\[ 
\Big( \int_\Omega \abs{f(\omega)}^2 \, \mu(\ud{\omega}) \Big)^\einhalb
:= \sup \Big\{ \Big|\int_\Omega g f \, \ud{\mu} \Big|\suchthat g \in
\Ell{2}(\Omega), \norm{g}_2\le 1 \Big\}
\] 
if this supremum exists in $E$. Our intention is to prove the
following.

\begin{thm}\label{acsf.t.blfc}
  Let $E$ be a Banach lattice of finite cotype, let $(\Omega, \Sigma,
  \mu)$ be a measure space, and let $f\in
  \Pe_2(\Omega;E)$.  Then the following assertions are equivalent:
\begin{aufzii}
\item $f\in \gamma(\Omega;E)$.
\item $\displaystyle \Big( \int_\Omega \abs{f(\omega)}^2 \,
  \mu(\ud{\omega}) \Big)^\einhalb$ exists in $E$.
\end{aufzii} 
In this case
\[ 
  \Bignorm{ \Big( \int_\Omega \abs{f(\omega)}^2 \, \mu(\ud{\omega})
    \Big)^\einhalb}_X \approx \norm{f}_{\gamma(\Omega;E)}.
\] 
\end{thm}

\noindent The proof requires several steps, and is based on the
following deep theorem.

\begin{thm}\label{acsf.t.blfc-discrete}
  If $E$ is a Banach lattice of finite cotype, then
\[ 
   \Big( \Exp \Bignorm{ {\sum}_j \gamma_j \, u_j}_E^2     \Big)^\einhalb
\approx  \Bignorm{ \Big({\sum}_j \abs{u_j}^2\Big)^\einhalb }
\]
for all finite sequences $u_1, \ldots,u_n \in E$.
\end{thm}

\begin{proof}
  In \cite[16.18]{DieJarTon1995} one can find the analogous statement
  for (real) Rademachers and real Banach spaces. The extension to
  complex spaces is straightforward. The equivalence with Gaussians in
  place of Rademachers follows from Theorem~\ref{gamma.t.gam-vs-rad}.
\end{proof}

 We remark that in the case $E= \Ell{p}(\Omega)$ for $1 \le p < \infty$
 the proof of Theorem~\ref{acsf.t.blfc-discrete} is a straightforward
 application of the Khinchine--Kahane inequalities and Parseval's
 identity.
 
 \medskip

\noindent
{\em Proof of Theorem~\ref{acsf.t.blfc}}: 1) We restrict  to the case that
$H := \Ell{2}(\Omega)$ is separable, the proof in the  general case
being a straightforward adaptation.  Fix an orthonormal basis
$(e_n)_n$ of $H$ and  let $u_n := \int_\Omega f e_n$. 
If $g\in \Ell{2}(\Omega)$ then 
\[ g = {\sum}_n  \sprod{g}{e_n} e_n
\]
in $\Ell{2}(\Omega)$, whence
\[ \int_\Omega fg = {\sum}_n  \sprod{g}{e_n} \int_\Omega f e_n
= {\sum}_n  \sprod{g}{e_n} u_n
\]
in $E$. It follows that 
\begin{equation}\label{eq:acsf.e.blfc.mondieu}
\sup_{\norm{g}_2\le 1} \abs{\int_\Omega fg} 
= \sup_{\norm{\alpha}_2\le 1} \abs{ {\sum}_{j} \alpha_j u_j}
= \sup_{n\in \N }  \Big( {\sum}_{j=1}^n \abs{u_j}^2\Big)^\einhalb
\end{equation}
in the sense that one exists if and only if the other does.  

\smallskip
\noindent
2) By recourse on the definition it follows that
\[ \Big({\sum}_{j=1}^n \abs{v_j + w_j}^2\Big)^\einhalb 
\le \Big({\sum}_{j=1}^n \abs{v_j}^2\Big)^\einhalb 
+ \Big({\sum}_{j=1}^n \abs{w_j}^2\Big)^\einhalb 
\]
for any $v_1,\ldots,v_n, w_1, \ldots, w_n \in E$. From this it follows that
\[ \Big| \Big({\sum}_{j=1}^m \abs{u_j}^2\Big)^\einhalb 
-
\Big({\sum}_{j=1}^n \abs{u_j}^2\Big)^\einhalb \Big|
\le
\Big({\sum}_{j=n+1}^m \abs{u_j}^2\Big)^\einhalb 
\]
if $n < m$. Writing $v_n := \Big({\sum}_{j=1}^n \abs{u_j}^2\Big)^\einhalb$ we 
hence obtain
\begin{align*}
 \norm{v_m - v_n} & \le \Bignorm{
\Big({\sum}_{j=n+1}^m \abs{u_j}^2\Big)^\einhalb}
\approx \Big( \Exp \Bignorm{  {\sum}_{j=n+1}^m \gamma_j u_j}_E^2 \Big)^\einhalb
\\ & = \Bignorm{{\sum}_{j=n+1}^m \konj{e_j}\tensor u_j}_{\gamma(H;E)}.
\end{align*}
3) Now suppose that (i) holds, i.e., $\Jei_f \in \gamma(H;E)$.
Then $\Jei_f = \sum_j \konj{e_j}\tensor u_j$ in the norm of $\gamma(H;E)$.
By our considerations in 2) we conclude that $(v_n)_n$ is a Cauchy sequence and 
hence has a limit $v := \lim_{n\to \infty} v_n$ in $E$. It is
clear that $(v_n)_n$ is increasing, which implies that $v = \sup_n v_n$.  Then
by (\ref{eq:acsf.e.blfc.mondieu}) (ii) follows.

\smallskip
\noindent
4) Conversely, suppose that (ii) holds and let $F \subseteq \N$ be
any finite subset. Then 
\[ \Big({\sum}_{j\in F} \abs{u_j}^2\Big)^\einhalb
\le v := \Big({\sum}_{j=1}^\infty \abs{u_j}^2\Big)^\einhalb,
\]
which exists by hypothesis and 1). But then
\[ \Big( \Exp \Bignorm{  {\sum}_{j\in F} \gamma_j u_j}_E \Big)^\einhalb
\approx 
\Bignorm{\Big({\sum}_{j\in F} \abs{u_j}^2\Big)^\einhalb}
\le \norm{v}.
\]
It follows that $\Jei_f \in \gamma_\infty(H;E) = \gamma(H;E)$ since
$E$ has finite cotype.\qed

\medskip

Let us specialise $E = \Ell{p}(\Omega')$, $1 \le p < \infty$  for some 
measure space $(\Omega',\Sigma', \mu')$.

\begin{cor}
Let $(\Omega, \Sigma, \mu)$ and $(\Omega',\Sigma', \mu')$ be 
measure spaces, $p \in [1, \infty)$ and let $f: \Omega \times \Omega' \to \C$ be measurable. 
Then the following assertions are equivalent.
\begin{aufzii}
\item $(\omega \mapsto f(\omega,\cdot)) \in \gamma(\Omega; \Ell{p}(\Omega'))$
\item $\displaystyle
\Bignorm{\Big( \int_\Omega \abs{f(\omega,x)}^2 \, 
\mu(\ud{\omega})\Big)^\frac{1}{2}}_{\Ell{p}(\Omega',\mu'(\ud{x}))} < \infty.$  
\end{aufzii}
\end{cor}

If $1 < p <\infty$ then the dual space $\Ell{p'}(\Omega)$ has nontrivial type, whence dual square functions
on $\Ell{p}$ coincide with square functions on $\Ell{p'}$.

\section{Abstract  square function  estimates}
\label{s.acsf}

Building on the theory of $\gamma$-radonifying operators developed in
the previous chapter, we now come to a central definition.

\begin{defi}\label{acsf.d.XYsf}
Let $X,Y$ be Banach spaces. Then an (abstract)  {\em  $(X,Y)$-square function}
is any operator 
\[ Q : \dom(Q) \to \gamma(H;Y), \quad \dom(Q) \subseteq X
\]
for some Hilbert space $H$. 
A {\em dual $(X,Y)$-square function} is any operator
\[ Q^d: \dom(Q^d) \to \gamma(H;Y)'\cong \gamma'(H';Y'),\quad \dom(Q^d) \subseteq X'
\]
for some Hilbert space $H$. 
\end{defi}

A {\em square function estimate} or a {\em quadratic estimate} for the 
 $(X,Y)$-square function $Q$ is any inequality of the form
\begin{equation}\label{acsf.eq.sfe-abstract}
 \norm{Qx}_\gamma \le C \norm{x}\qquad \text{for all $x\in \dom(Q)$}
\end{equation}
for some constant $C\ge 0$. If $Q$ is densely defined, such a square function
estimate holds true  if and only if $Q$ extends to a bounded operator
$Q: X\to \gamma(H;Y)$.  Note that a 
closed and densely defined square function 
satisfies a square function estimate if and only if it is fully defined.

Similarly, an estimate of the form
\[ \norm{Q^dx'}_{\gamma'}  \le C \norm{x'}\qquad (x'\in \dom(Q^d))
\] 
is called a {\em dual square function (quadratic) estimate}. 
The usual examples of dual square functions are not
densely, but only weakly$^*$-densely defined, and hence in general 
a dual square function estimate does not lead 
to a bounded operator $X'\to \gamma'(H';Y')$.

\medskip
The following is a standard way to arrive at $(X,Y)$-square functions.
Suppose that $A: \dom(A) \to \Lin(H;Y)$ is an  operator with 
$\dom(A) \subseteq X$. Then we can take its {\em part} in $\gamma(H;X)$
\[  A_\gamma: \dom(A_\gamma) \to \gamma(H;Y)
\]
with $\dom(A_\gamma) = \{ x\in \dom(A) \suchthat Ax \in \gamma(H;Y)\}$
and $A_\gamma x :=Ax$. It is easy to see that $A_\gamma$ is a closed
square function if $A$ is closed.  
(Obviously, a similar construction is possible to obtain dual square functions.)

\medskip
The square function
$Q: \dom(Q) \to \gamma(H;Y)$ is called {\emdf subordinate} to the square function 
$R: \dom(R) \to \gamma(K;Y)$, in symbols: $Q \subord R$, if 
$\dom(Q)\subseteq  \dom(R)$ and there is a bounded operator $T: H \to K$ such that
\[ Qx = Rx \nach T\quad \text{for all $x\in \dom(R)$}.
\]
The square functions are called {\emdf strongly equivalent}, in symbols
$Q\strongeq R$,  if $Q \subord R$ and  
$R \subord Q$. Note that if $Q \subord R$ then, by the ideal property, 
there is a constant $c\ge 0$ such that
\[ \norm{Qx}_\gamma \leq c \norm{Rx}_\gamma \quad \text{for all $x\in \dom(R)$}.
\]
Analogously, a dual square function $Q^d: \dom(Q^d) \to \gamma'(H';Y')$ is 
{\em subordinate} to a dual square function 
$Q^d: \dom(R^d) \to \gamma'(H';Y')$ if $\dom(Q^d)\subseteq  \dom(R^d)$ 
and there is a bounded operator $T: H'  \to K' $ such that
\[ Q^dx' = R^dx' \nach T\quad \text{for all $x' \in \dom(R^d)$}.
\]
It is evident that any (dual) square function subordinate to a bounded (dual) square function
is itself bounded. Subordination is a (trivial) way to generate new square function estimates from
known ones.

\medskip

In the following we shall describe how one can associate square
functions with a functional calculus in a natural way. To this end
we first have to review some basic functional calculus theory.

\subsection{A functional calculus round-up}

Let $\gebiet$ be a nonempty set, $\calF$ a unital algebra of
scalar-valued functions on $\gebiet$, $\calE \subseteq \calF$ a
subalgebra of $\calF$ and $\Phi: \calE \to \Lin(X)$ an algebra
homomorphism, where $X$ is a Banach space.  Then the triple
$(\calE,\calF,\Phi)$ is an {\em abstract functional calculus} in the
sense of \cite[Chapter 1]{HaaseFC}. The mapping $\Phi: \calE \to
\Lin(X)$ is called the {\em elementary} calculus. A function $f\in
\calF$ is {\em regularisable}, i.e., there is $e\in \calE$ (called a
{\em regulariser}) such that $ef \in \calE$ and $\Phi(e)$ is
injective.  In this case one can define
\[ \Phi(f) := \Phi(e)^{-1} \Phi(ef)
\]
with natural domain. This definition is independent of the regulariser
and consistent with the elementary calculus.  One can show
\cite[Section 1.2.1]{HaaseFC} that the set $\calF = \calF_r$ of
regularisable elements is a unital subalgebra of $\calF$, so we may
suppose without loss of generality that $\calF = \calF_r$ in the
following.

\medskip

\noindent
In our context the most interesting case is $\calF = \Ha^\infty(\gebiet)$, the
algebra of bounded holomorphic functions on  
an open set $\gebiet \subseteq \C$. In this case, if there is  $C\ge0$ such that
$\Phi(f) \in \Lin(X)$ and
\[ \norm{\Phi(f)} \le C \norm{f}_{\Ha^\infty} \qquad\text{for all $f\in \Ha^\infty(\gebiet)$},
\] 
then we speak of $\Phi$ as a {\em bounded $\Ha^\infty$-calculus} on $\gebiet$.

\begin{rem}\label{fcru.r.generator}
  If $\gebiet \subseteq \C$ is open, then by Liouville's theorem the
  algebra $\Ha^\infty(\Omega)$ is only interesting if $\leer \not=
  \gebiet \not= \C$. {\em We shall tacitly assume this when talking
    about $\Ha^\infty(\gebiet)$-functional calculus.}

  Now suppose that a functional calculus $\Phi: \calE \to \Lin(X)$ is
  given with $\calE \subseteq \Ha^\infty(\gebiet)$, and   suppose that 
  $\C\ohne \gebiet$ has nonempty interior $U$, say.  For
  each $\lambda \in U$ the function $r_\lambda(z):= (\lambda -
  z)^{-1}$ is holomorphic and bounded on $\gebiet$. If we suppose in
  addition that $R_\lambda := \Phi(r_\lambda) \in \Lin(X)$, then this
  yields a pseudo-resolvent on $U$. Hence by
  \cite[Prop.~A.2.4]{HaaseFC} there is unique operator $A$ with
  $R(\lambda,A) = R_\lambda$ for all $\lambda \in U$. (This operator
  is single-valued if and only if one/each $R_\lambda$ is injective.)
  It is common to call $\Phi$ a functional calculus {\em for} $A$ and
  write $f(A) := \Phi_A(f) := \Phi(f)$ for $f\in \calF$.
\end{rem}

We suppose that the reader is familiar with the functional calculus
for sectorial/strip type operators as developed in \cite{HaaseFC}.
For the convenience of the reader, we have included a brief
description of the construction in Appendix~\ref{app:fc}.

\subsection{Square functions associated with a  functional calculus}
\label{acsf.s.sffc}

In this section we shall associate square functions with a given
functional calculus.  As a motivating example we use the sectorial
calculus (see section~\ref{exas.s.sectorial} below).

\medskip Given a sectorial operator $A$ of angle $\omega_0$ on a
Banach space $X$ and a function $\psi \in
\Ha^\infty_0(\sector{\omega})$ with $\omega \in (\omega_0, \pi)$ one
considers --- for fixed $x\in X$ --- the vector-valued function
\[ (0, \infty) \pfeil X , \qquad t   \mapsto \psi(tA)x.
\] 
Following Kalton and Weis \cite{KalWei2004} one should interpret this function
as an operator
\[ T_\psi x: \Ell{2}^*(0, \infty)  \pfeil X
\]
via (Pettis) integration, cf.~Appendix~\ref{app.P2} and Section~\ref{exas.s.sectorial}.
Abbreviating $H := \Ell{2}^*(0, \infty)$ 
one  looks at estimates of the form 
\[  \norm{\psi(tA)x}_{\gamma( (0, \infty);X)} = \norm{T_\psi x}_{\gamma(H;X)}
\lesssim \norm{x},
\]
then called a square function estimate. For $x\in \dom(A)\cap \ran(A)$ one can
employ the definition of the functional calculus by Cauchy integrals
to obtain
\begin{align*}
 T_\psi x & = \int_0^\infty h(t)  \psi(tA)x \, \tfrac{\ud{t}}{t} = 
\int_0^\infty h(t) \frac{1}{2\pi \ui} \int_\Gamma \psi(tz)R(z,A)x\, \ud{z} \, \tfrac{\ud{t}}{t}
\\ & =
 \frac{1}{2\pi \ui} \int_\Gamma  \Big( 
\int_0^\infty h(t)  \psi(tz)  \, \tfrac{\ud{t}}{t} \Big)\,  R(z,A)x\, \ud{z}
\\ & =  \Big( 
\int_0^\infty h(t)  \psi(tz)  \, \tfrac{\ud{t}}{t} \Big)(A)x.
\end{align*}
The last step indicates an important {\em change in perspective}.  The
function of two variables $(t,z) \mapsto \psi(tz)$  may as well be
viewed as an $\Ell{2}^*(0, \infty)$-valued $\Ha^\infty$-function
\[ \Psi: \sector{\omega} \pfeil H, \qquad \Psi(z)(t) := \psi(tz).
\] 
Then 
\[  
z \mapsto \int_0^\infty h(t)  \psi(tz)  \, \tfrac{\ud{t}}{t} = 
\dprod{h}{\Psi(z)}
\]
is a scalar $\Ha^\infty$-function, into which $A$ can be inserted by
the functional calculus.  Finally, this operator can be applied to
$x\in \dom(A) \cap \ran(A)$. But then for fixed such $x$ this yields
an operator $H \to X$, and one can ask whether this operator is
$\gamma$-radonifying.  (In this special case it is, see
Section~\ref{exas.s.sectorial} below.)

\medskip
\noindent
Let us pass from concrete example to the general situation. 
We fix  a functional calculus $(\calE, \calF,\Phi)$ over a set 
$\gebiet$ as 
discussed in the previous section. Again we suppose $\calF = \calF_r$, i.e.,
every function in $\calF$ is regularisable.

For a Hilbert space
$H$ and a function $f: \gebiet \to H'$ we abbreviate
\[ 
\eprod{h}{f} : \gebiet \to \C,\qquad  (\eprod{h}{f})(z) 
:= \dprod{h}{f(z)}_{H, H'} \quad (z\in \gebiet,\, h \in H).
\]
Then we define
\[ 
   \calF(\gebiet;H')  := \{ f : \gebiet \to H' \suchthat \eprod{h}{f}\in \calF
\,\,\forall\, h\in H\}.
\]
We now extend the functional calculus $\Phi$ to $\calF(\gebiet;H')$ by setting
\begin{align*}
   \Phi(f) & : \dom(\Phi(f)) \to \Lin(H;X), \\
 \dom(\Phi(f)) & := \{ x\in X \suchthat x \in \dom(\Phi(\eprod{h}{f})) \text{ for all } h \in H \}\\ 
  [\Phi(f)x]\, h & := \Phi(\eprod{h}{f})x 
\end{align*}
This definition/notation is consistent with the original notation
under the identification $\calF(\gebiet;H') = \calF$ in the case
that  $H = \C$ is one-dimensional.

\medskip
In the next step we take the part of $\Phi(f)$ in $\gamma(H;X)$ to arrive
at the square function $\Phi_\gamma(f) : \dom(\Phi_\gamma(f))\to \gamma(H;X)$, 
\[ \Phi_\gamma(f)x := \Phi(f)x,\qquad \dom(\Phi_\gamma(f))
 = \{ x\in \dom(\Phi(f) \suchthat 
\Phi(f)x \in \gamma(H;X)\}.
\]
We call the square function $\Phi_\gamma(f)$ {\em bounded} if 
$\dom(\Phi_\gamma(f)) = X$ and 
\[ \Phi_\gamma(f): X \to \gamma(H;X)
\]
is a bounded operator. 
If $X$ does not  contain a copy of $\co$,
then $\gamma(H;X) = \gamma_\infty(H;X)$. Hence, for  
$f \in \calF(\gebiet; H')$ the associated square function
$\Phi_\gamma(f)$ 
is bounded  
if and only if
$\Phi(\eprod{h}{f}) \in \Lin(X)$ for all $h \in H$ and
there is a constant $c\ge 0$ such that
\[ \Exp\big\| {\sum}_{\alpha \in F} \gamma_\alpha \Phi( \eprod{e_\alpha}{f})x \big\|^2
\le c \norm{x}^2
\]
for all $x\in X$, a fixed orthonormal basis $(e_\alpha)_{\alpha\in I}$ of $H$
and all finite subsets $F \subseteq I$.

\medskip

In  the following lemma we collect some properties of 
the so-obtained square functions. Note that 
$\calF(\gebiet;H')$ is an $\calF$-module with respect
to  pointwise multiplication.

\begin{lemma}\label{acsf.l.properties}
In the situation just described, the following assertions hold
for each $f \in \calF(\gebiet;H')$:
\begin{aufzi}
\item The operators $\Phi(f)$ and $\Phi_\gamma(f)$ are closed.
\item If $g \in \calF(\gebiet;H')$ then
\[ \Phi_\gamma(f) + \Phi_\gamma(g) \subseteq \Phi_\gamma(f + g).
\]
\item If $g \in \calF$ then
\[ \Phi_\gamma(f)\Phi(g) \subseteq \Phi_\gamma( f\cdot g)
\]
with $\dom(\Phi_\gamma(f)\Phi(g)) = \dom(\Phi(g)) \cap 
\dom(\Phi_\gamma( f\cdot g))$.
\item If $g\in \calF$ then
\[ \Phi(g) \nach \Phi_\gamma(f) \subseteq \Phi_\gamma( f\cdot g)
\]
\item If $g\in \calF$ such that $\Phi(g) \in \Lin(X)$, then 
\[ \Phi(g) \nach \Phi_\gamma(f)   \subseteq 
\Phi_\gamma( f\cdot g) = \Phi_\gamma(f)\Phi(g) 
\]
In particular, $\dom(\Phi_\gamma(f))$ is invariant under $\Phi(g)$.
\end{aufzi}
\end{lemma}

The assertion d) means: if $x\in \dom(\Phi_\gamma(f))$ 
and $\Phi(g) [\Phi_\gamma(f)x]
\in \gamma(H;X)$, then $x\in \dom(\Phi_\gamma(f\cdot g))$ and
$\Phi(g) [\Phi_\gamma(f)x] = \Phi_\gamma(f\cdot g)x$.

\begin{proof}
The proof is left to the reader. The assertions in b) and c) follow
more or less directly from the corresponding statements about
the functional calculus $(\calE,\calF,\Phi)$
\cite[Prop.~1.2.2]{HaaseFC}. Assertion d) is straightforward,
and e) is a consequence of c) and d). (Note that by the ideal property of
$\gamma(H;X)$, 
$\dom( \Phi(g) \nach \Phi_\gamma(f)) = \dom(\Phi_\gamma(f))$.
\end{proof}

From Lemma~\ref{acsf.l.properties}  we see  
that the mapping $f \mapsto \Phi_\gamma(f)$ behaves like a 
functional calculus, so we call it the {\em vectorial} 
$\calF$-calculus. In particular,   in the case that $\calF = \Ha^\infty(\gebiet)$
for some open subset $\gebiet \subseteq \C$, the map
\[ \Phi_\gamma: \Ha^\infty(\gebiet;H') \to \{ \text{$H$-square functions on $X$}\} 
\] 
is called a  {\em vectorial} $\Ha^\infty$-calculus on $\gebiet$. 
The vectorial $\Ha^\infty$-calculus is {\em bounded}
if $\Phi_\gamma(f)$ is a bounded square function for each $f\in \Ha^\infty(\gebiet;H')$
and there is a constant $C\ge 0$ such that
\[ \norm{\Phi_\gamma(f)x}_\gamma \le C \, \norm{f}_{\Ha^\infty(\gebiet)} \norm{x}
\quad (x\in X,\, f\in \Ha^\infty(\gebiet;H')).
\]
Clearly, if the vectorial $\Ha^\infty$-calculus is bounded, then
the underlying  scalar $\Ha^\infty$-calculus is bounded. 
We shall prove that, essentially, the converse holds for 
sectorial/strip type operators (Theorem~\ref{exas.t:bdd-vectorial-calc}).

\medskip

Suppose again that $\calF = \Ha^\infty(\gebiet)$ for some open subset
$\gebiet \subseteq \C$. We say that the {\em scalar convergence lemma}
holds if the following is true: whenever $(f_n)_n$ is  a sequence
in $\Ha^\infty(\gebiet)$ with $\sup_{n \in \N} \norm{f_n}_{\infty} < \infty$
and $f_n \to f$ pointwise on $\gebiet$, 
$\Phi(f_n) \in \Lin(X)$ for all $n \in \N$ and  
$\sup_{n \in \N} \norm{\Phi(f_n)} < \infty$, 
then
$\Phi(f) \in \Lin(X)$ and $\Phi(f_n) \to \Phi(f)$ strongly as $n \to \infty$.

The scalar convergence lemma holds for the functional 
calculus of a sectorial operator with dense domain and range and
for a densely defined operator of strip type, see \cite[Section 5.1]{HaaseFC}.

\begin{lemma}[Convergence lemma]\label{acsf.l.vcl} Let 
$(\calE, \Ha^\infty(\gebiet), \Phi)$ be a functional calculus on a
Banach space $X$ such that the scalar convergence lemma holds.
Suppose that $X$ does not contain a copy of $\co$.
Then the {\em vectorial convergence lemma} holds, i.e.: 
Let $(f_n)_n$ be  a sequence
in $\Ha^\infty(\gebiet; H')$ satisfying
\begin{aufziii}
\item $\sup_{n \in \N} \norm{f_n}_{\infty} < \infty$,
\item $f_n(z) \to f(z)$ weakly for all $z \in \gebiet$,
\item  $\Phi_\gamma(f_n) \in \Lin(X; \gamma(H;X))$ for all $n \in \N$ and  
\item $\sup_{n \in \N} \norm{\Phi_\gamma(f_n)}_{\Lin(X;\gamma(H;X))} < \infty$.
\end{aufziii}
Then $\Phi_\gamma(f) \in \Lin(X;\gamma(H;X))$ and 
$\Phi_\gamma(f_n)x \to \Phi_\gamma(f)x$ strongly in $\Lin(H;X)$ 
as $n \to \infty$, for each $x\in X$.
\end{lemma}

\begin{proof}
Fix $h \in H$. Then $\sup_n \norm{\eprod{h}{f_n}}_\infty 
\le \norm{h} \sup_n\norm{f_m}_\infty < \infty$
and $\eprod{h}{f_n}\to \eprod{h}{f}$ pointwise on $\gebiet$.
Moreover, $\Phi(\eprod{h}{f_n}) \in \Lin(X)$ 
and 
\[ \norm{\Phi(\eprod{h}{f_n})x}_X = 
\norm{[\Phi_\gamma(f_n)x]h}_X
\le \norm{h} \norm{\Phi_\gamma(f_n)x}_{\Lin(H;X)}
\le  \norm{h} \norm{\Phi_\gamma(f_n)x}_{\gamma(H;X)}
\]
for all $n \in \N$. This yields 
$\sup_n \norm{\Phi(\eprod{h}{f_n})}_{\Lin} 
\le \norm{h} \sup_n \norm{\Phi_\gamma(f_n)}_{\Lin(X; \gamma(H;X)}$.  
By the scalar convergence lemma, $\Phi(\eprod{h}{f}) \in \Lin(X)$, and 
$\Phi(\eprod{h}{f_n}) \to \Phi(\eprod{h}{f})$ strongly on $X$.
That is, for every $x\in X$ is
$\Phi_\gamma(f_n)x \to \Phi(f)x$ strongly in  $\Lin(H;X)$. 
By the $\gamma$-Fatou Lemma ~\ref{gamma.l.fatou1}, 
$\Phi(f)x \in \gamma_\infty(H;X)$, and 
since $X$ does not contain a copy of $\co$, 
$\Phi(f)x \in \gamma(H;X)$ for each $x\in X$.
\end{proof}

\subsection{Dual square functions associated with a  functional calculus}
\label{acsf.s.dsffc}

Let again $(\calE, \calF, \Phi)$ be a proper  functional calculus
where $\calF$ is an algebra of functions defined on the set $\gebiet$.
As above, we suppose for simplicity that $\calF = \calF_r$. 

\medskip

\noindent
For a Hilbert
space $H$ and a function $f: \gebiet \to H$ we abbreviate
\[ \eprod{f}{h'} : \gebiet \to \C,\qquad  (\eprod{f}{h'})(z) 
:= \dprod{f(z)}{h'}_{H, H'} \quad (z\in \gebiet,\, h' \in H')
\]
and  define
\begin{equation}\label{acsf.F-Hwertig}
   \calF(\gebiet;H) 
:= \{ f : \gebiet \to H \suchthat \eprod{f}{h'}\in \calF
\,\,\forall\, h'\in H'\}.
\end{equation}
For  fixed $f\in \calF(\gebiet;H)$  we then define the operator 
\begin{align*}
   \Phi^d(f) & : \dom(\Phi^d(f)) \to \Lin(H';X')\\
 \dom(\Phi^d(f)) & := \{ x'\in X'\suchthat x'\in \dom(\Phi(\eprod{f}{h'})') \text{ for all } h'\in H' \}  \\                     
 [\Phi^d(f)x']\,h' & := \Phi(\eprod{f}{h'})'x'
\end{align*}
Then we pass to the associated  dual square function 
\begin{align*}
 \Phi_{\gamma'}(f) & : \dom(\Phi_{\gamma'}(f)) \to \gamma'(H';X'), \qquad
\Phi_{\gamma'}(f)x'  := \Phi^d(f)x'\\
 \dom(\Phi_{\gamma'}(f))
 & = \{ x'\in \dom(\Phi^d(f) \suchthat 
\Phi^d(f)x' \in \gamma(H';X')\} \subseteq X'.
\end{align*}
Of course, this is only meaningful if $\Phi(\eprod{f}{h'})'$ is single-valued,
i.e., if $\Phi(\eprod{f}{h'})$ is densely defined for each $h'\in H'$. 
We therefore make the following

\medskip
\noindent
{\bf Standing assumption:} {\em Whenever we speak of a dual square function
associated with a function $f\in \calF(\gebiet;H)$, we require 
that for each $h'\in H'$ the operator $\Phi(\eprod{f}{h'})$ is densely
defined.}

\medskip
\noindent
The following lemma is the analogue of Lemma~\ref{acsf.l.properties}.

\begin{lemma}\label{acsf.l.properties-dual}
In the situation just described, the following assertions hold
for $f \in \calF(\gebiet;H)$:
\begin{aufzi}
\item The operator $\Phi_{\gamma'}(f)$ is  weak$^*$-to-weak$^*$ closed.
\item If $g \in \calF$ such that $\Phi(g) \in \Lin(X)$ 
and if $x'\in \dom(\Phi_{\gamma'}(f\cdot g))$, 
then $\Phi(g)'x'\in \dom(\Phi_{\gamma'}(f))$
and 
\[ \Phi_{\gamma'}(f)\Phi(g)'x'= \Phi_{\gamma'}(f\cdot g)x'.
\] 
\end{aufzi}
\end{lemma}

\begin{proof}
a) is again left to the reader. For the proof of b) we fix $h'\in H'$ and 
note first that since  $\Phi(g)$ is bounded we have
\[ \Phi( \eprod{(f\cdot g)}{h'})' = \Phi( (\eprod{f}{h'}) g)'
\subseteq   \big( \Phi(g) \Phi(\eprod{f}{h})\big)'
= \Phi(\eprod{f}{h'})'\Phi(g)'
\]
by \cite[A.4.2 and 1.2.2]{HaaseFC}. The claim  now follows easily.
\end{proof}

The following theorem yields a useful  characterisation
of ``dual square function estimates''.

\begin{thm}\label{acsf.t.dsf}
Let $(e_\alpha)_{\alpha \in I}$ be a fixed orthonormal basis of $H$.
The following assertions are equivalent for $f \in \calF(\gebiet; H)$: 
\begin{aufzii}
\item
$\Phi_{\gamma'}(f)$ is a bounded operator 
$\Phi_{\gamma'}(f): X'\to \gamma'(H';X')$.
\item 
The assignment
\[ T( h'\tensor x) := \Phi(\eprod{f}{h'})x,\qquad 
h'\in H', \, x \in \dom(\Phi(\eprod{f}{h'}))
\]
extends to a bounded operator $T: \gamma(H;X) \to X$. 

\item There is a constant $c\ge 0$ such that
\[ \big\| {\sum}_{\alpha \in F} \Phi( \sprod{f}{e_\alpha} )x_\alpha\big\|_X^2
\le c\,\,  \Exp\big\| {\sum}_{\alpha \in F} \gamma_\alpha x_\alpha\big\|^2
\]
for all finite subsets $F \subseteq I$ and $x_\alpha \in 
\dom(\Phi( \sprod{f}{e_\alpha}))$ for $\alpha \in F$. 

\end{aufzii}
In this case $T = \Phi_{\gamma'}(f)'\restrict_{\gamma(H;X)}$ 
is the pre-adjoint of $\Phi_{\gamma'}(f)$ (under the identification
$ \gamma'(H';X') \cong \gamma(H;X)'$), and $c= \norm{T} = \norm{\Phi_{\gamma'}(f)}$ can be chosen in
{\upshape (iii)}.

\medskip
\noindent
Furthermore, if $g \in \calF$ is such that 
$\Phi(g) \in \Lin(X)$, then
\begin{equation}\label{acsf.e.dsf-inv}
 \Phi_{\gamma'}(f)'( \Phi(g)\nach S) = \Phi(g) \big( \Phi_{\gamma'}(f)'S \big)
\quad \text{for all $S\in \gamma(H;X)$}.
 \end{equation} 
\end{thm}

\begin{proof}
(i)$\dann$(ii): 
By hypothesis, $\Phi_{\gamma'}(f)': \gamma(H;X)'' \to X''$ is  bounded.  
Fix $x'\in X'$, $h'\in H'$ and 
$x \in \dom(\Phi(\eprod{f}{h'}))$. Then 
\begin{align*}
\dprod{ \Phi_{\gamma'}(f)'( h' \tensor x)}{x'}_{X'',X'}
& = \dprod{h' \tensor x}{ \Phi_{\gamma'}(f) x'}
= \tr\big( (\Phi_{\gamma'}(f)x')' (h' \tensor x) \big)
\\ & = \dprod{ x}{ [\Phi_{\gamma'}(f)x']\,h'}
 = \dprod{ x}{ \Phi( \eprod{f}{h'})'x'}
= \dprod{\Phi( \eprod{f}{h'})x}{x'}.
\end{align*}
Consequently, 
$\Phi_{\gamma'}(f)'( h' \tensor x) = \Phi( \eprod{f}{h'})x = T(h'\tensor x) \in X$.
Since  $\dom( \Phi( \eprod{f}{h'}))$ is dense in $X$, 
the linear span of  such elements 
$h'\tensor x$ 
is dense in $\gamma(H;X)$. The claim follows.

\smallskip
\noindent
(ii)$\Iff$(iii): This follows since 
$T( \sum_{\alpha \in F} \konj{e_\alpha} \tensor x_\alpha) =
\sum_{\alpha \in F} \Phi( \sprod{f}{e_\alpha})x_\alpha$.

\smallskip
\noindent
(ii)$\dann$(i): It suffices to show that
$\Phi_{\gamma'}(f) = T': X'\to \gamma(H;X)'\cong \gamma'(H';X')$.
Fix $x'\in X'$. Then 
\[ \dprod{x}{(T'x')(h')}_{X, X'} = \dprod{h' \tensor x}{T' x'}_{\gamma, \gamma'}  = \dprod{T(h'\tensor x)}{x'}_{X, X'} = \dprod{\Phi( \eprod{f}{h'} )x}{x'}
\]
for all $h'\in H'$ and $x \in \dom( \Phi(\eprod{f}{h'}))$. Hence 
$x' \in \dom( \Phi(\eprod{f}{h'})')$  and
\[
[\Phi_{\gamma'}(f)x']\,h' = \Phi(\eprod{f}{h'})'x' = (T'x')h'\qquad \text{for all $h'\in H'$}.
\]
That is, $\Phi_{\gamma'}(f) = T'$.

\smallskip
\noindent
For the remaining statement let again $h'\in H'$ and $x \in 
\dom(\Phi(\eprod{f}{h'}))$. Then,  with $S := h'\tensor x$,
\[  \Phi(g)( T(S)) = \Phi(g)\Phi(\eprod{f}{h'})x
= \Phi(\eprod{f}{h'})\Phi(g)x = T( h'\tensor \Phi(g)x)
= T( \Phi(g)\nach S).
\]
Since the linear span of such operators $S$ is a dense subset 
of $\gamma(H;X)$, the claim
follows from the ideal property of $\gamma(H;X)$.
\end{proof}

\subsection{Square functions over $\Ell{2}$-spaces}
\label{acsf.s.sfL2}

Up to now we worked with a general Hilbert space $H$.  If one is in
the special situation $H = \Ell{2}(\Omega)= H'$ for some measure space
$(\Omega, \ud{t})$, it is natural to consider functions of {\em two
  variables} $f = f(t,z)$ in the construction of square functions.

To proceed further we shall suppose in addition that $\calF =
\Ha^\infty(\gebiet)$ for some nonempty open set $\gebiet \subseteq \C$
with $\gebiet \neq \C$, and that $\Phi= \Phi_A$ is a functional
calculus for the (possibly multivalued) operator $A$,
cf.~Remark~\ref{fcru.r.generator}.  (Note that for any Hilbert space
$H$, the space $\calF(\strip{\omega};H)$ derived from the space $\calF
= \Ha^{\infty}(\strip{\omega})$ by \eqref{acsf.F-Hwertig} above,
coincides with the space of $H$-valued bounded holomorphic functions.)

\begin{lemma}\label{acsf.l.f(t,z)}
Let $\gebiet \subseteq \C$ be an open subset of the complex plane, 
let $f: \Omega \times \gebiet \to \C$ be measurable and suppose in addition that 
\begin{aufziii}
\item $f(t,\cdot) \in \Ha^\infty(\gebiet)$ for almost all $t\in \Omega$ and
\item $\sup_{z \in \gebiet} \int_\Omega \abs{f(t,z)}^2 \, \ud{t} < \infty$.
\end{aufziii} 
Then $(z \mapsto f(\cdot,z)) \in \Ha^\infty(\gebiet; \Ell{2}(\Omega))$. 
\end{lemma}

\begin{proof}
Let $g\in \Ell{2}(\Omega)$. It remains to show that the function 
$F(z) := \int_\Omega g(t)f(t,z)\, \ud{t}$ is
holomorphic. To this end, let $B$ be any open ball such that $\cls{B} \subseteq \gebiet$.
Then $f(a,t) = \frac{1}{2\pi \ui} \int_{\rand B} \frac{f(t,z)\ud{z}}{z- a}$ 
for $a\in B$, for almost all $t\in \Omega$, by 
the Cauchy formula. Fubini's theorem yields
\[    F(a) := \int_\Omega g(t)f(t,a)\, \ud{t} = \frac{1}{2\pi \ui} \int_{\rand B} \frac{F(z)\ud{z}}{z- a}
\]
for all $a\in B$. By a standard result in complex function theory \cite[Theorem~10.7]{Rudin:RC},
$F$ is holomorphic.
\end{proof}

For $f$ as in the lemma we have
\[ [\Phi(f)x]\,h = \Big( \int_\Omega h(t)f(t,z)\, \ud{t}\Big)(A)x
\]
if $x\in \dom(\Phi(f))$ and $h \in H= \Ell{2}(\Omega)$. As in the example
of sectorial operators and ``dilation type'' square functions 
discussed at the beginning of this section, one  has 
\[ [\Phi(f)x]\,h = \Big( \int_\Omega h(t)f(t,z)\, \ud{t}\Big)(A)x 
= \int_\Omega h(t) f(t,A)x\, \ud{t}
\]
in many situations at least for vectors $x$ from a large subspace of $X$.
We therefore use the symbol $f(\cdot,A)x$ or $f(t,A)x$ as a 
convenient alternative notation --- as a {\em fa\c{c}on de parler} ---
for the operator $\Phi(f)x$. So, whenever expressions of the form
\[ \norm{f(t,A)x}_\gamma
\]
appear, it is {\em not}  implied that ``$f(t,A)x$'' has to make sense
literally (i.e.,
$x \in \dom(f(t,A))$ for almost all $t\in \Omega$ and 
$\Phi(f)x= \Jei_{f(t,A)x}$) but just as a suggestive notation. 
It is actually one of the advantages of our approach to square functions that
one does not have to worry about the vector-valued integration  too much.

\section{Square function estimates: New from old}\label{s.sfno}

In this chapter we discuss certain general principles how to 
generate new (dual) square function estimates from known ones.
A fairly trivial instance of such a principle is given by subordination.

\subsection{Subordination}\label{sfno.s.subord}

Subordination for abstract square functions has been defined in the beginning
of Chapter~\ref{s.acsf}.
Here we consider a special instance for the case of square functions
associated with a functional calculus $(\calE,\calF,\Phi)$ over a set $\gebiet$.

\begin{thm}\label{sfno.t:subord}
Let $K$ be another Hilbert space and  $T: K \to H$  a bounded linear operator.
\begin{aufzi}
\item If  $g \in \calF(\gebiet;H')$
then $T'\nach g \in \calF(\gebiet;K')$, 
$\dom(\Phi_\gamma(T' \nach g)) \subseteq \dom(\Phi_\gamma(g))$ and 
\[ \Phi_\gamma(T'\nach g)x = \Phi_\gamma(g)x \nach T
\qquad \text{for all $x\in \dom(\Phi_\gamma(f))$.}
\]
In particular, $\Phi_\gamma(T'\nach g)\subord \Phi_\gamma(g)$.\\

\item If  $f \in \calF(\gebiet;K)$
then $T\nach f \in \calF(\gebiet;H)$, 
$\dom(\Phi_{\gamma'}(T \nach f)) \subseteq \dom(\Phi_{\gamma'}(f))$ and 
\[ \Phi_{\gamma'}(T\nach f)x' = \Phi_{\gamma'}(f)x' \nach T'
\qquad \text{for all $x' \in \dom(\Phi_\gamma(f))$. }
\]
In particular, $\Phi_{\gamma'}(T\nach f)\subord \Phi_{\gamma'}(f)$.
\end{aufzi}
\end{thm}

\begin{proof}
This is an easy exercise.
\end{proof}

We shall abbreviate $\Phi_\gamma(f) \subord \Phi_\gamma(g)$ and
$\Phi_{\gamma}(f) \strongeq \Phi_{\gamma}(g)$ simply by
\[ f \subord g  \quad \text{and}\quad f \strongeq g,
\]
respectively, whenever it is convenient. The same  abbreviation is
used in the case of dual square functions.
For applications of the subordination principle see 
Chapter~\ref{s.exas} below.

\subsection{Tensor products (and property $(\alpha)$)}
\label{sfno.s.tensor}

Again we work with a functional calculus $(\calE,\calF, \Phi)$ on
a Banach space $X$, $\calF$ being an algebra of functions defined on 
a set $\gebiet$. Let  $H,K$ be Hilbert spaces and 
$f \in \calF(\gebiet;H')$ and $g \in \calF(\gebiet;K')$. Then one can
consider the function 
\[ f\tensor g: \gebiet \pfeil H'\tensor K'\,\,\, \subseteq\,\,\, (H\tensor K)'
\qquad (f\tensor g)(z) := f(z) \tensor g(z),
\]
and we suppose in addition that $(f \tensor g) \in \calF(\gebiet; (H\tensor K)')$.
(This is the case, e.g., if $\calF = \Ha^\infty(\gebiet)$, and $\gebiet$ some
open subset of $\C$.) Even more, suppose that the associated
square functions 
\[  \Phi_\gamma(f): X \pfeil \gamma(H;X)\quad \text{and}\quad
\Phi_\gamma(g): X \pfeil \gamma(K;X)
\]
are bounded. It is then
natural to ask 
whether or under which conditions the square function
$\Phi_\gamma(f\tensor g)$ is bounded as well. 
By the ideal property, composition with $\Phi_\gamma(g)$ yields
a bounded operator 
\[ \Phi_\gamma(g)^\tensor : \gamma(H;X) \to \gamma(H; \gamma(K;X)),\qquad
\Phi_\gamma(g)^\tensor T :=   \Phi_\gamma(g) \nach T
\]
(the ``tensor extension''). Hence
\[ \Phi_\gamma(g)^\tensor \nach  \Phi_\gamma(f): X \pfeil
\gamma(H; \gamma(K;X))
\]
is bounded.  With $x\in X$, $h\in H$ and $k \in K$ we can compute
\begin{align*}
\Big[ & \big[\Phi_\gamma(g)^\tensor( \Phi_\gamma(f)x )\big]h \Big]k = 
\Big[ \big[\Phi_\gamma(g)\nach  (\Phi_\gamma(f)x) \big]h \Big]k = 
\Big[ \Phi_\gamma(g)\big( [\Phi_\gamma(f)x]h \big)\Big]k 
\\ & =
\Big[ \Phi_\gamma(g) \big(\Phi(\eprod{h}{f})x \big) \Big]k 
= \Phi( \eprod{k}{g}) \Phi(\eprod{h}{f})x 
= \Phi((\eprod{h}{f}) (\eprod{k}{g}))x 
\\ & = \Phi( \eprod{(h\tensor k)}{(f\tensor g)})x
= \big[\Phi_\gamma(f\tensor g)x \big](h\tensor k).
\end{align*}
Thus our question can be answered positively if the natural mapping
\[ h' \tensor (k' \tensor x) \tpfeil (h' \tensor k') \tensor x
\]
induces a bounded operator $\gamma( H ; \gamma(K;X)) \to \gamma(H
\tensor K;X)$.  This is the case if and only if the Banach space $X$
has Pisier's ``property $(\alpha)$'', see \cite[Definition
2.1]{Pisier:lust} for the original definition employing Rademacher
sums, and \cite[Chapter 13]{Nee2010} for the stated equivalence.
Every Hilbert space has property $(\alpha)$ and each space
$\Ell{p}(\Omega;X)$ with $1\le p < \infty$ inherits this property from
$X$ \cite[Chapter 13]{Nee2010}.
Let us summarise our considerations in the following lemma.

\begin{lemma}\label{sfno.l.tensor}
  Let $H, K$ be two Hilbert spaces and $X$ be a Banach space with
  property $(\alpha)$. Suppose further that  the square functions
\[
  \Phi_{\gamma}(f): \; X \pfeil \gamma(H; X)
      \quad\text{ and }\quad
  \Phi_{\gamma}(g): \; X \pfeil \gamma(K; X)
\]
are bounded. Then the tensor square function
\[
   \Phi_{\gamma}(f \otimes g): \; X \pfeil \gamma(H{\otimes}K; X)
\]
is bounded, too.
\end{lemma}

\subsection{Lower square function estimates I}\label{acsf.s.lsfe}

A {\emdf lower square function estimate} is 
an estimate of the form
\[ \norm{x}_X \leq C\, \norm{\Phi_\gamma(g)x}_\gamma \qquad (x\in \dom(\Phi_\gamma(g))).
\]
In certain situations one can combine a lower square function
estimate, a usual square function estimate and a subordination
to  show the boundedness of an operator $\Phi(f)$.

\begin{lemma}\label{nsfo.l.pushing-through}
Let $H,K$ be Hilbert spaces and let $g \in \calF(\gebiet;K')$ and 
$\tilde{g} \in \calF(\gebiet;H')$.
Suppose that the square function $\Phi_\gamma(\tilde{g}): X \to \gamma(K;X)$ is bounded
and that one has a lower square function estimate 
\[ \norm{x}_X \leq C\, \norm{\Phi_\gamma(g)x}_\gamma \qquad (x\in \dom(\Phi_\gamma(g)))
\]
for $\Phi_\gamma(g)$.
Suppose further that the scalar-valued function
$f\in \calF$ is such that there is $T_f \in \Lin(H;K)$ with 
\[ f \cdot g = T_f' \nach \tilde{g}.
\]
Then 
\[ \norm{\Phi(f)x} \leq C \norm{T_f} \norm{\Phi_\gamma(\tilde{g})}\, \norm{x} 
\qquad(x\in \dom(\Phi(f))).
\] 
\end{lemma}

\begin{proof}
By Lemma~\ref{acsf.l.properties}.c), 
\[ \Phi_\gamma(g)\Phi(f) \subseteq \Phi_\gamma(f\cdot g) 
= \Phi_\gamma(T_f \nach \tilde{g}) = [\Phi_\gamma(\tilde{g}) \cdot]\nach T_f
\]
Since the rightmost operator is fully defined, if $x\in \dom(\Phi(f))$ then 
$\Phi(f)x\in \dom(\Phi_\gamma(g))$  (still by Lemma~\ref{acsf.l.properties}.c))
and hence
\begin{align*}
 \norm{\Phi(f)x}_X & \le C \norm{\Phi_\gamma(g)\Phi(f)x}_\gamma
= C \norm{\Phi_\gamma(\tilde{g})x \nach T_f}_\gamma
\le C \norm{T_f} \norm{\Phi_\gamma(\tilde{g})x}_\gamma
\\ & \le C \norm{T_f} \norm{\Phi_\gamma(\tilde{g})} \, \norm{x}
\end{align*}
as claimed.
\end{proof}

Lemma~\ref{nsfo.l.pushing-through} is an abstract version of the
``pushing the operator through the square function''-technique used by
Kalton and Weis in \cite{KalWei2004} (see also \cite[Theorem
10.9]{KunWei2004}) to show that a norm equivalence
\[ \norm{R(\pm \ui \omega + \cdot , A)x}_{\gamma(\Ell{2}(\R);X)}
\sim \norm{x}_X
\]
for a strip type operator $A$ implies the boundedness of the
$\Ha^\infty$-calculus on a strip, see Section~\ref{ss:sing-int-repres}
below for details.

\subsection{Lower square function estimates II}\label{acsf.s.lsfe2}

We now present some methods to {\em establish} lower square function
estimates.  These, however, require slightly stronger assumptions
about the underlying functional calculus.  Indeed, we shall work with
a functional calculus $(\calE, \Ha^\infty(\gebiet), \Phi)$ admitting a
function $e \in \calE$ with the following properties:
\begin{aufziii}\label{lsfe.requirements}
\item $ef \in \calE$ for all 
$f \in \Ha^\infty(\gebiet)$;
\item if $(f_n)_{n \in \N}$ is a sequence in $\Ha^\infty(\gebiet)$
with $\sup_{n \in \N} \norm{f_n}_{\Ha^\infty}< \infty$ and $f_n \to f$ pointwise,
then $\Phi(ef_n) \to \Phi(ef)$ weakly on $X$;
\item $\Phi(e)$ is injective. 
\end{aufziii}
The standard functional calculi for  strip type
and sectorial operators are of this kind, see  Lemma~\ref{acsf.l.Qf} 
below. Note that 1) and 3) just tell that the function $e$ is a
``universal regulariser'' for $\Ha^\infty(\gebiet)$.

\begin{rem}
The following considerations are motivated by McIntosh's approximation formula
\[ x =  \int_0^\infty \vphi(tA)\psi(tA)x \, \tfrac{\ud{t}}{t}
\]
for $x\in \cls{\dom}(A) \cap \cls{\ran}(A)$, sectorial operators $A$ and 
appropriate functions $\vphi, \psi$, see \cite{McI1986}
and \cite[Sec.~5.2]{HaaseFC}.
\end{rem}

Let
$f \in \Ha^\infty(\gebiet;H)$,  
$g \in \Ha^\infty(\gebiet;H')$ and 
\[ (\eprod{f}{g})(z) :=  \dprod{f(z)}{g(z)}_{H, H'} \qquad (z\in \gebiet).
\]
(This notation is consistent with the notation $\eprod{h'}{f}$ and 
$\eprod{g}{h}$ introduced in Sections~\ref{acsf.s.sffc} and~\ref{acsf.s.dsffc}.)
Then $\eprod{f}{g} \in \Ha^\infty(\gebiet)$ and we expect
the formula
\begin{equation}
 \dprod{\Phi(\eprod{f}{g})x}{x'}_{X,X'} = 
\dprod{\Phi_{\gamma}(g)x}{\Phi_{\gamma'}(f)x'}_{\gamma,\gamma'}
\end{equation}
to hold. The following result gives some conditions.

\begin{lemma}\label{acsf.l.lsfe-basic}
In the described situation, if $e\in \calE$ 
has the properties {\upshape 1)} and {\upshape 2)} above, and if 
$x = \Phi(e)y$ for some
$y \in \dom(\Phi_\gamma(g))$, then 
\[
 \dprod{\Phi(\eprod{f}{g})x}{x'} = \dprod{\Phi_{\gamma}(g)x}{\Phi_{\gamma'}(f)x'}
\]
for all $x'\in \dom(\Phi_{\gamma'}(f))$.
\end{lemma}

\begin{proof}
Let us first note that, under the given conditions, 
$x \in \dom(\Phi_\gamma(g))$.
Indeed, this follows directly from 
Lemma~\ref{acsf.l.properties}.d). 

\smallskip
\noindent
For the proof of the claim we let $(e_\alpha)_{\alpha \in I}$ be an 
orthonormal basis of $H$ and denote 
\[ g_\alpha(z) := \dprod{e_\alpha}{g(z)}= \big(e_\alpha | \konj{g(z)}\big),
\quad 
f_\alpha(z) := \dprod{f(z)}{\konj{e_\alpha}} = 
\sprod{f(z)}{e_\alpha}
\]
for $\alpha \in I$. Then by general Hilbert space theory
\[ (\eprod{f}{g})(z) = {\sum}_{\alpha} f_\alpha(z) \cdot g_\alpha(z)
\]
for each $z\in \gebiet$, and the partial sums are uniformly bounded.
(Note that
the sum is actually only over countably many $\alpha$ since 
$\{ \konj{g(z)}, f(z) \suchthat z\in \gebiet\}$ is separable.)
Hence, for  $x'\in \dom(\Phi_{\gamma'}(f))$  we can compute
\begin{align*}
& \dprod{\Phi_\gamma(g)x}{\Phi_{\gamma'}(f)x'}_{\gamma, \gamma'} 
= 
\sum_{\alpha} \dprod{\Phi(g_\alpha)x}{ \Phi(f_\alpha)'x'}_{X,X'} 
\\ & \qquad =
\sum_\alpha \dprod{ \Phi(e)\Phi(g_\alpha)y}{\Phi(f_\alpha)'x'}_{X,X'} 
=
\sum_\alpha \dprod{ \Phi(f_\alpha)\Phi(e)\Phi(g_\alpha)y}{x'}_{X,X'}
\\ & \qquad =  
\sum_\alpha \dprod{ \Phi(e f_\alpha g_\alpha)y}{x'}_{X'X'}
=
\dprod{ \Phi(e (\eprod{f}{g}))y}{x'}_{X,X'}  = \dprod{\Phi(\eprod{f}{g})x}{x'}_{X,X'}.
\end{align*}
Here we used c) of Theorem~\ref{gamma.t.dual} and property 
2) of the function $e$.
\end{proof}

\begin{thm}\label{acsf.t.lsfe}
Let $f \in \Ha^\infty(\gebiet;H)$ and  $g\in \Ha^\infty(\gebiet;H')$ and
suppose that there is $e\in \calE$ satisfying {\upshape 1)--3)} above.
If $\Phi_{\gamma'}(f)$ is a bounded operator, then
\[ \Phi_{\gamma'}(f)'\Phi_\gamma(g) \subseteq \Phi(\eprod{f}{g}).
\]
In other words,
$\dom(\Phi_\gamma(g)) \subseteq \dom( \Phi(\eprod{f}{g}))$ and 
\[ 
 \dprod{\Phi(\eprod{f}{g})x}{x'} = 
\dprod{\Phi_{\gamma}(g)x}{\Phi_{\gamma'}(f)x'} \qquad \text{for all 
$x\in \dom(\Phi_\gamma(g) )$ and all $x'\in X'$}.
\]
In particular, one has the lower estimate
\[ \norm{\Phi(\eprod{f}{g})x}_X\lesssim \norm{\Phi_\gamma(g)x}_{\gamma}
\quad \text{for all $x\in \dom(\Phi_\gamma(g))$}.
\] 
\end{thm}

\begin{proof}
We let $y := \Phi_{\gamma'}(f)'[\Phi_{\gamma}(g)x] \in X$ by 
Theorem~\ref{acsf.t.dsf}.
Take $e \in \calE$ satisfying 1), 2) and 3) above. Then by 
Lemma~\ref{acsf.l.lsfe-basic},
for each $x'\in X'$ we have
\begin{align*}
 \dprod{\Phi(e (\eprod{f}{g}))x}{x'} & = 
 \dprod{\Phi(\eprod{f}{g})\Phi(e)x}{x'}  = 
\dprod{\Phi_{\gamma}(g)\Phi(e)x}{\Phi_{\gamma'}(f)x'} 
\\ & = 
\dprod{ \Phi(e)\nach [\Phi_{\gamma}(g)x]}{\Phi_{\gamma'}(f)x'}
 = 
 \dprod{ \Phi_{\gamma'}(f)'\big( \Phi(e)\nach [\Phi_{\gamma}(g)x])}{x'}
\\ & =
\dprod{\Phi(e) (\Phi_{\gamma'}(f)'[\Phi_{\gamma}(g)x])}{x'}
= \dprod{\Phi(e)y}{x'}
\end{align*}
where we used Lemma~\ref{acsf.l.properties}.e) and 
\eqref{acsf.e.dsf-inv}. By construction of the functional calculus, 
$x\in \dom(\Phi(\eprod{f}{g}))$ and $\Phi(\eprod{f}{g})x = y$. 
The remaining assertions follow easily. 
\end{proof}

\begin{cor}\label{acsf.c.normeq}
Let $f \in \Ha^\infty(\gebiet;H)$ and  $g\in \Ha^\infty(\gebiet;H')$
such that $\Phi_\gamma(g)$ and $\Phi_{\gamma'}(f)$ are bounded operators.
Then $\Phi(\eprod{f}{g})$ is a bounded operator and
\[ 
 \dprod{\Phi(\eprod{f}{g})x}{x'} = 
\dprod{\Phi_{\gamma}(g)x}{\Phi_{\gamma'}(f)x'} \qquad \text{for all $x\in X$ and 
$x'\in X'$}.
\]
In particular, if $\eprod{f}{g} = \car$ then one has the  norm equivalence
\[ \norm{x}_X \simeq \norm{\Phi_\gamma(g)x}_{\gamma}
\quad \text{for all $x\in X$}.
\] 
\end{cor}

The problem whether to a given function $f \in
 \Ha^\infty(\gebiet; H)$ there exists a function $g$ with 
$\eprod{f}{g}= \car$  is known as the {\em Corona problem}.  
For separable Hilbert
spaces and bounded holomorphic functions on the disc such functions
$g$ always exist provided $\inf_{z\in \bbD} \norm{f(z)}_H >0$, see
Tolokonnikov \cite{Tol1981} and Uchiyama \cite{Uch1980} and also
\cite[Appendix 3]{Nikolski:shift}. By a conformal mapping this
result extends to strips or sectors immediately.

\begin{cor}
In addition to the standing assumptions of this section, suppose that
$\gebiet$ is a simply connected domain in $\C$ and  that
$\Phi_{\gamma'}(f)$ is a bounded operator for all $f \in
\Ha^\infty(\gebiet; H)$. Then there is a constant $C\ge 0$ with the following
property: whenever  $g \in \Ha^\infty(\gebiet; H')$ is such
that $\delta := \inf_{z\in \gebiet} \norm{g(z)}_{H'} >0$, one has 
\[
   \norm{x} \leq C \, c(\delta)
\norm{  \Phi_{\gamma}(g)\, x }_\gamma \qquad \text{ for all } x\in  \dom( \Phi_{\gamma}(g) )
\]
with $c(\delta)\le \delta^{-2}\ln(1+\tfrac1\delta)^{\sfrac32}$.
\end{cor}

\begin{proof}
The closed graph theorem yields a constant $C_1$ with $\norm{\Phi_{\gamma'}(f)} 
\le C_1 \norm{f}_\infty$ for all $f\in \Ha^\infty(\gebiet;H)$. 
And the Tolokonnikov--Uchiyama lemma yields for given $g$ a function
$f\in \Ha^\infty(\gebiet)$ with $\eprod{f}{g}= \car$ and 
$\norm{f}_\infty \le C_2  \delta^{-2}\ln(1+\tfrac1\delta)^{\sfrac32}$.
Now the claim follows from Theorem~\ref{acsf.t.lsfe} with $C = C_1 C_2$.
\end{proof}

\subsection{Integral representations}\label{acsf.s.nsfo}

In this section we describe a method of how to obtain new square function estimates
from known ones via integral representations. We build on the previous results and 
hence work   with a functional calculus $(\calE, \Ha^\infty(\gebiet), \Phi)$
on a Banach space $X$ under  the same hypotheses
as before in Section~\ref{acsf.s.lsfe}, i.e., we require the existence of 
a function $e \in \calE$ satisfying 1)--3) on page
\pageref{lsfe.requirements}. As always, $H$ is an arbitrary Hilbert space.
The following is the  main
result.

\begin{thm}\label{acsf.t.nsfo}
Let $(\Omega, \ud{t})$ be a measure space, define $K := \Ell{2}(\Omega)$, and let  
$f,g \in \Ha^\infty(\gebiet;K)$ and $m \in \Ell{\infty}(\Omega;H')$
such that 
\begin{aufziii}
\item $\Phi_\gamma(g) : X \to \gamma(K;X)$ is bounded and 
\item $\Phi_{\gamma'}(f): X'\to \gamma'(K;X')$ is bounded.
\end{aufziii}
Consider the function $u\in \Ha^\infty(\gebiet;H')$ defined by  
\begin{equation}\label{acsf.e.nsfo}
 u(z) := \int_\Omega m(t) \cdot f(t,z)\, g(t,z)\, \ud{t}\,\,  \in\,\,  H' \qquad (z\in \gebiet).
\end{equation}
If $H$ has finite dimension or $X$ has finite cotype,  then 
the operator $\Phi_\gamma(u): X \to \gamma(H;X)$ is bounded, too, with 
\begin{equation}
 \norm{\Phi_\gamma(u)}
\le c\, \norm{m}_{\Ell{\infty}(\Omega;H')} \, 
\norm{\Phi_{\gamma'}(f)}
\norm{\Phi_\gamma(g)},
\end{equation}
where $c$ depends on $\dim(H)$ or the cotype (constant) of $X$, respectively.
\end{thm}

\begin{proof}
We define the bounded linear mapping
\[ S: H \to \Ell{\infty}(\Omega) \hookrightarrow \Lin(K),\qquad 
Sh := \eprod[H]{h}{m}
\]
and form
\[ g_h:= (\eprod[H]{h}{m}) \, g = (Sh)g\in  \Ha^\infty(\gebiet;K)
\]
for $h \in H$. Then, by definition, 
\[ (\eprod[H]{h}{u})(z) = \int_\Omega f(z) (\eprod[H]{h}{m}) g(z)
= \int_\Omega f(z) g_h(z) = (\eprod[K]{f}{g_h})(z) 
\qquad (z\in \gebiet).
\]
For $k \in K = \Ell{2}(\Omega)$,
\[ \eprod[K]{g_h}{k} = \eprod[K]{ ((Sh)\, g)}{k}
= \eprod[K]{g}{((Sh)\,k)}.
\]
Hence 
\[ \Phi(g_h)x = (\Phi_\gamma(g)x) \nach (Sh) \in \gamma(K;X) \qquad \text{for each $x\in X$},
\]
by the ideal property. The operator
\[ T: H \to \gamma(K;X), \qquad h \mapsto Th := \Phi(g_h)x
\]
factors through $\Ell{\infty}(\Omega)$, and $\gamma(K;X)$ has finite
cotype (Lemma~\ref{gamma.l.cotyp}). 
Hence, by Theorem~\ref{gamma.t.kw-ess}, $T \in \gamma(H; \gamma(K;X))$. 
By Theorem~\ref{acsf.t.dsf} and hypothesis 2), 
$\Phi_{\gamma'}(f)': \gamma(K;X)\to X$
is bounded, and another application of the ideal property yields
that $h \mapsto \Phi_{\gamma'}(f)' Th = \Phi_{\gamma'}(f)' \Phi_\gamma(g_h)x$ is in $\gamma(H;X)$.
But 
\[ \Phi_{\gamma'}(f)' \Phi_\gamma(g_h) = \Phi(\eprod[K]{f}{g_h}) =
\Phi( \eprod[H]{h}{u}) 
\]
by Theorem~\ref{acsf.t.lsfe}. It follows that $\Phi_\gamma(u)$ is bounded. 
To prove the norm estimate we trace back all these steps:
\begin{align*}
\norm{\Phi_\gamma(u)x}_\gamma & = 
\norm{h \mapsto \Phi_{\gamma'}(f)' \Phi_\gamma(g_h)x}_\gamma
\\ & \le \norm{\Phi_{\gamma'}(f)} \norm{h \mapsto \Phi_{\gamma}(g_h)x}_{\gamma(H;\gamma(K;X))}
\end{align*}
and
\begin{align*}
\norm{h \mapsto\Phi_{\gamma}(g_h)x}_{\gamma(H;\gamma(K;X))} 
& = \norm{T}_{\gamma(H;\gamma(K;X))} = \norm{h \mapsto \Phi_\gamma(g)x\nach Sh}_{\gamma(H;\gamma(K;X))}
\\  \le c \norm{\Phi_\gamma(g)x}_{\gamma(K;X)} & \norm{S}_{H \to \Ell{\infty}(\Omega)}
\le c \norm{\Phi_\gamma(g)} \norm{x}  \norm{m}_{\Ell{\infty}(\Omega;H')}.
\end{align*}
Here $c = c(q, \ct_q(X))\ge 0$ is the constant coming from the
application of the factorisation Theorem~\ref{gamma.t.kw-ess}. (Note
that by the proof of Lemma~\ref{gamma.l.cotyp}, $\ct_q(\gamma(K;X))$
depends on $q, \ct_q(X)$ and some universal constants.)
\end{proof}

Suppose that $\Phi$ is a functional calculus for the operator $A$, 
$H= \Ell{2}(\Omega')$, and 
\[ u(s,z)  = \int_\Omega m(s,t) f(t,z) g(t,z)\, \ud{t}.
\]
Then  Theorem~\ref{acsf.t.nsfo} says the following:
if the square and dual square functions
associated with $g(\cdot ,A)$ and $f(\cdot,A)$, respectively, are
bounded, then also the square function associated with $u(\cdot,A)$ 
is bounded. (Note our convention from Section~\ref{acsf.s.sfL2}.)
For $H =\C$ this theorem is the main tool to infer bounded
$\Ha^\infty$-calculus from square and dual square function estimates.
Examples are given in Chapter~\ref{s:appl} below.

\begin{rem}
  We do not know of a proper dual analogue of
  Theorem~\ref{acsf.t.nsfo}. However, under certain conditions one can
  use it to obtain bounded square functions for the dual functional
  calculus and  by the inclusion $\gamma(H';X') \subseteq \gamma'(H';X')$
  this yields a bounded dual square function for the original
  calculus.
\end{rem}

\subsection{Square function estimates from \ELLONE-frame-boundedness}
\label{lob.s.l1-sfe}

Whereas the main result of the previous section, Theorem~\ref{acsf.t.nsfo},
can be used to infer a bounded $\Ha^\infty$-calculus from bounded
(dual) square functions, in the present section we study the converse.
Our results are based on a  certain boundedness 
concept for subsets of a Hilbert space,
the so called \ELLONE-frame-boundedness. 

\medskip
\noindent
Let $H$ be a complex Hilbert space. 
The {\em \ELLONE-frame-bound} of a subset $M
\subseteq H$ is defined as
\begin{equation} \label{eq:ell-eins-bdd-first}
 \abs{M}_1 := \inf \norm{L}\sup_{x \in M} {\sum}_{\alpha\in I} \abs{\dprod{Rx}{e_\alpha}},
\end{equation}
where the infimum is taken over all pairs $(L,R)$ 
of bounded linear operators
\[  R: H \to \ell_2(I), \quad L: \ell_{2}(I)\to H,\quad LR = \Id_H,
\]
with $I$ being any (sufficiently large) index set. 
And $M$ is called \emph{\ELLONE-frame-bounded} if $\abs{M}_1 < \infty$.

\medskip
This notion is, to the best of our knowledge, new and ---
presumably --- interesting in its own right. However, it plays only an auxiliary role here,
so we decided to postpone a thorough  treatment to 
 Appendix~\ref{app:l1-bdd}.

\begin{thm}\label{lob.t.l1-sfe}
Let $\gebiet \subseteq \C$ be an open set and let 
$\Phi: \Ha^\infty(\gebiet) \to \Lin(X)$ be a bounded algebra homomorphism,
where $X$ is a Banach space.  Furthermore, let $f\in \Ha^\infty(\gebiet;H)$ and 
$g\in \Ha^\infty(\gebiet;H')$.
Then the following assertions hold.
\begin{aufzi}
\item 
If $g$ has \ELLONE-frame-bounded image in $H'$
and $X$ has cotype $q < \infty$, then  $\Phi_\gamma(g) \in \Lin(X; \gamma(H;X))$ with 
\[   \norm{\Phi_\gamma(g)x}_\gamma \leq   
2 \ct_q(X) m_q\, \norm{\Phi} \, \abs{g(\gebiet)}_1 \cdot \norm{x}_X
\qquad (x\in X),
\]
where 
$\ct_q(X)$ is the cotype-$q$
constant of $X$ and $m_q$ is  the $q$-th absolute moment of the
normal distribution.

\item If $f$ 
has \ELLONE-frame-bounded image in $H$, then 
$\Phi_{\gamma'}(f) \in \Lin(X'; \gamma'(H';X'))$ with 
\[   \norm{\Phi_{\gamma'}(f) x'}_{\gamma'} \leq
\sqrt{\tfrac{\pi}{2}}\, \norm{\Phi} \, \abs{f(\gebiet)}_1 \cdot \norm{x'}_{X'}
\qquad (x' \in X').
\]
\end{aufzi}
\end{thm}

\begin{proof}
We fix an index set $I$ and bounded operators $R: H'\to \ell_2(I)$, $L: \ell_2(I) \to H'$ with 
$LR = \Id_{H'}$. Let $(r_\alpha)_{\alpha \in I}$ be
independent complex Rademacher variables. Then by Theorem~\ref{gamma.t.gam-vs-rad} 
for $F\subseteq I$ finite we have 
\begin{align*}
 \Exp \, \Bignorm{ {\sum}_{\alpha \in F} &  \gamma_\alpha
[\Phi(g)x]R'e_\alpha}_X^2  
 \le \; \ct_q(X)^2 m_q^2 \, \Exp\, \Bignorm{ {\sum}_{\alpha \in F} r_\alpha 
\Phi(\eprod{R'e_\alpha}{g})x}_X^2  
\\ = &\;    \ct_q(X)^2\,  m_q^2 \,  \Exp \, 
\Big\|  \Phi\Big( {\sum}_{\alpha \in  F} r_\alpha (\eprod{R' e_\alpha}{g})\Big)x  
\Big\|_X^2 \\
\le &\; \big( \ct_q(X) m_q \norm{\Phi} \,  \norm{x}\big)^2 \; 
\Exp\, \sup_{z\in \gebiet} \Bigl| {\sum}_{\alpha\in F} r_\alpha\dprod{R' e_\alpha}{g(z)} \Bigr|^2\\
\le &\; \big(2 \ct_q(X) m_q \norm{\Phi} \,  \norm{x}\big)^2\;  
\Big(\sup_{z\in \gebiet} {\sum}_{\alpha \in I} \bigl| \dprod{e_\alpha}{R g(z)} \bigr| \Big)^2.
\end{align*}
Consequently, by the ideal property,
\begin{align*}
\norm{\Phi(g)x}_\gamma
&  = \norm{\Phi(g)x\nach R'L'}_\gamma \le   
 \norm{L}\;  \norm{\Phi(g)x \nach R'}_\gamma   \\
&\le \;2 \ct_q(X) m_q \norm{\Phi} \,  \norm{x}\;  
\norm{L} \sup_{z\in \gebiet} {\sum}_{\alpha \in I} \bigl| \dprod{e_\alpha}{Rg(z)} 
\bigr|.
\end{align*}
Taking the infimum over all pairs $(L,R)$  by \eqref{eq:ell-eins-bdd-first} 
we obtain
$\Phi(g)x \in \gamma_\infty(H;X)$ and 
\[ \norm{\Phi(g)x}_\gamma
\le \;2 \ct_q(X) m_q \norm{\Phi} \, \abs{g(\gebiet)}_1 \,   \norm{x}.
\]
Finally, since $X$ has finite cotype it cannot contain a copy of $\co$ and 
hence by the Hoffmann-J\o{}rgensen--Kwapie\'n theorems,
$\Phi(g)x \in \gamma(H;X)$.

\medskip
\noindent
For the proof of b) let us abbreviate $V:= \Phi^d(f)x': H' \to X'$.
We fix again an index set $I$ and operators $R: H \to \ell_2(I)$, $L: \ell_2(I) \to H$
with $LR = \Id_H$, and a family  $(r_\alpha)_{\alpha \in I}$ of
independent (complex) Rademachers.
 
Let $F \subseteq I$ be finite, 
$(x_\alpha)_{\alpha\in F} \subseteq X$ and $U := 
\sum_{\alpha\in F} e_\alpha \tensor x_\alpha \in \ell_2(I)\tensor X$.  Then
\begin{align*}
\abs{\tr( (VR')'U)} & \stackrel{(1)}{=}
\abs{{\sum}_{\alpha\in F} \dprod{x_\alpha}{VR'e_\alpha}}
= 
\abs{{\sum}_{\alpha\in F} \dprod{x_\alpha}{ \Phi(\eprod{f}{R'e_\alpha})'x' }}
\\ & 
=
\abs{{\sum}_{\alpha\in F} \dprod{\Phi(\eprod{f}{R'e_\alpha})x_\alpha}{x'}}
\le \norm{x'} \norm{ {\sum}_{\alpha \in F} \Phi(\eprod{R f}{e_\alpha})x_\alpha }
\\ &  \stackrel{(2)}{\le}
\tfrac{1}{2} \norm{x'}
\norm{ \Exp\,  {\sum}_{\alpha\in F}
 \konj{r_\alpha} \Phi(\eprod{Rf}{e_\alpha})  {\sum}_{\beta \in F} r_\beta x_\beta }
\\ & \le
\tfrac{1}{2} \norm{x'}\,
\Exp\,  \norm{ \Phi \Big({\sum}_{\alpha\in F} \konj{r_\alpha} 
(\eprod{Rf}{e_\alpha})\Big)}_{\Lin(X)}
\norm{  {\sum}_{\beta\in F} r_\beta x_\beta }_X 
\\ & \stackrel{(3)}{\le}
\tfrac{1}{\sqrt{2}} \norm{\Phi}\,  \norm{x'}\,
 \Big(\sup_{z\in \gebiet} {\sum}_{\alpha\in F} \abs{\dprod{Rf(z)}{e_\alpha}}\Big) \, 
\Exp\,  \norm{  {\sum}_{\beta\in F} r_\beta x_\beta }_X 
\\ & \stackrel{(4)}{\le} 
\tfrac{\sqrt{\pi}}{2} \norm{\Phi}\,  \norm{x'}\,
 \Big(\sup_{z\in \gebiet} {\sum}_{\alpha\in F} \abs{\dprod{Rf(z)}{e_\alpha}}\Big) \, 
\Exp\,  \norm{  {\sum}_{\beta\in F} \gamma_\beta x_\beta }_X 
\\ &  \stackrel{(5)}{\le}
\sqrt{\tfrac{\pi}{4}} \norm{\Phi}\,  \norm{x'}\,
 \Big(\sup_{z\in \gebiet} {\sum}_{\alpha} \abs{\dprod{Rf(z)}{e_\alpha}}\Big) \, 
\norm{U}_{\gamma}. 
\end{align*}
[(1) holds by \eqref{gamma.eq.trace}; (2) holds since
$\Exp\;( \konj{r_\alpha}r_\beta) = 2 \delta_{\alpha \beta}$; (3)
holds since $\abs{r_\alpha} = \sqrt{2}$; (4) is an application of
\eqref{gamma.eq.r-to-gamma} with q=1; and (5) follows from 
the standard inequality $\Exp \abs{g} \le \big(\Exp \abs{g}^2\big)^{1/2}$.]
Consequently, we obtain
\[ \norm{VR'}_{\gamma'} \le \sqrt{\tfrac{\pi}{4}} \norm{\Phi}\,  \norm{x'}\,
 \Big(\sup_{z\in \gebiet} {\sum}_{\alpha} \abs{\dprod{Rf(z)}{e_\alpha}}\Big)
\]
and hence, by the ideal property (Corollary~\ref{gamma.c.dual-ideal}), 
\[ \norm{V}_{\gamma'} = \norm{VR'L'}_{\gamma'} \le 
\norm{L'} \norm{VR'}_{\gamma'} 
\le \sqrt{\tfrac{\pi}{4}} \norm{\Phi}\,  \norm{x'}\,
\norm{L} \,  \Big(\sup_{z\in \gebiet} {\sum}_{\alpha} \abs{\dprod{Rf(z)}{e_\alpha}}\Big).
\]
Taking the infimum over all pairs $(L,R)$ concludes the proof of  b).
\end{proof}

\begin{rem}\label{lob.r.l1-sfe}
  It is notable that part b) of the theorem holds without any
  geometric assumption on the Banach space. On the other hand, the
  appearance of finite cotype in the formulation of a) is natural as
  one needs to estimate a Gaussian sum in terms of a Rademacher sum.
  This raises the question why we do not work with Rademacher sums
  exclusively right from the start. The reason is that
  ``$R$-radonifying'' operators do not have as nice properties as the
  $\gamma$-radonifying ones, in particular the right ideal property
  fails.  Since this property was also needed in the proof of a),
  nothing would be gained by working with ``$R$-radonifying''
  operators.
\end{rem}

\section{Examples}\label{s.exas}

In this chapter we discuss several applications of the constructions and results of the 
previous chapters. Our main protagonist is the functional calculus
 for strip type operators as sketched in Appendix~\ref{app:fc}. Via the 
$\exp/\log$-correspondence these results all have sectorial versions.

\subsection{Strip type operators}\label{exas.s.sf-strip}

Suppose that $A$ is an operator of strip type $\omega_0\ge 0$ on a
Banach space $X$, and let $\omega > \omega_0$.  Then we can consider
the functional calculus $(\elementary(\strip{\omega}),
\Ha^\infty(\strip{\omega}), \Phi)$ for $A$ as defined in
Appendix~\ref{app:fc}.

\begin{lemma}\label{acsf.l.Qf}
  In the described situation the following assertions hold:
\begin{aufzi}
\item\label{acsf.l.item-Qf} Let $g \in \Ha^\infty(\strip{\omega};H')$
  and let $\Phi_\gamma(g)$ be the associated square function
\[ \Phi_\gamma(g): \dom(\Phi_\gamma(g)) \to \gamma(H;X), \quad 
\Phi_\gamma(g)x = \Big(h \mapsto (\eprod{h}{g})(A)x\Big).  
\]
Then $\dom(\Phi_\gamma(g))$ contains $\ran(e(A))$ for each $e\in
\elementary(\strip{\omega})$.

\item\label{acsf.l.item-Qdf} 
Suppose that $A$ is densely defined. 
Then each operator $f(A)$, $f\in \Ha^\infty(\strip{\omega})$, is densely defined
and dual square functions are well defined.

\smallskip
\noindent
Let $f \in \Ha^\infty(\strip{\omega};H)$ and let 
$\Phi_{\gamma'}(f)$ be the associated square function
\[ \Phi_{\gamma'}(f): \dom\Phi_{\gamma'}(f) \to \gamma(H;X), \quad \Phi_{\gamma'}(f)x' = \Big(h' \mapsto (\eprod{f}{h'})(A)'x'\Big)  
\]
Then $\dom(\Phi_{\gamma'}(f))$ 
contains $\ran(e(A)')$ for each $e\in \elementary(\strip{\omega})$.

\end{aufzi}
\end{lemma}

\begin{proof}
\ref{acsf.l.item-Qf}  
Let $e\in \elementary(\strip{\omega})$, let $x\in X$ and $h \in H$. Then
\begin{align*}
 [\Phi(f)e(A)x]\,h  & = (\eprod{h}{f})(A)e(A)x = ((\eprod{h}{f}) e)(A)x  
\\ &  = \frac{1}{2\pi \ui} \int_{\rand\strip{\omega'}} \dprod{h}{f(z)} e(z) R(z,A)x \, \ud{z}
\end{align*} 
with $\omega' \in (\omega_0, \omega)$.  
This shows that $e(A)x \in \dom(\Phi(f))$ and  that
\[  \Phi(f)e(A)x = 
\frac{1}{2\pi \ui} \int_{\rand\strip{\omega'}} f(z) \tensor e(z) R(z,A)x \, \ud{z}
\]
is nuclear, whence in $\gamma(H;X)$ (Lemma~\ref{gamma.l.nuc}). 
The proof of~\ref{acsf.l.item-Qdf} is similar.
\end{proof}

As a consequence we  obtain that this functional calculus satisfies
the requirements of  Section~\ref{acsf.s.lsfe}. Indeed, every 
$e\in \calE(\strip{\omega})$ satisfies 1) and 2) from page
\pageref{lsfe.requirements}. 

\medskip
\noindent
In the following, we shall discuss several
 instances of square functions for strip type operators.
We recall our standing assumption that whenever we speak of dual square functions
the dual calculus is supposed to be well defined, cf.~Section~\ref{acsf.s.dsffc}.
We begin with  the square functions ``of shift type''.

\begin{exa}[Shift type square functions]\label{acsf.exa.shift-sf}
Let $\omega' > \omega$ and $\psi \in \elementary(\strip{\omega'})$,
and define $g(t,z) := \psi(t + z)$ for $t\in \R$ and $z\in \strip{\omega}$.
Then $g: \R \times \strip{\omega} \to \C$ satisfies the hypotheses of Lemma~\ref{acsf.l.f(t,z)}. 
This gives rise to the
(dual) square function 
\begin{align*}
[\Phi_\gamma(g)x]\,h & = \Big(\int_\R h(t) \psi(t+z) \, \ud{t}\Big)(A)x &  (h \in \Ell{2}(\R),\, 
x\in \dom(\Phi_\gamma(g))\\
[\Phi_{\gamma'}(g)x']\,h & = \Big(\int_\R h(t) \psi(t+z) \, \ud{t}\Big)(A)'x' & (h \in \Ell{2}(\R),\,
x'\in \dom(\Phi_{\gamma'}(g)). 
\end{align*}
For $x\in \ran(e(A))$, $e\in \calE(\strip{\omega})$,  Lemma~\ref{acsf.l.Qf} and a simple Fubini argument show that $\Phi_\gamma(g)x$ is integration against the vector-valued function
$t\mapsto \psi(t{+}A)x$. 
If $X$ is itself a Hilbert space, one has 
\[ \norm{\Phi_\gamma(f)x}_\gamma = \Big(\int_\R \norm{\psi(t{+}A)x}^2\, \ud{t} \Big)^\einhalb,
\]
and hence a square function estimate takes the form 
$\int_\R \norm{\psi(t{+} A)x}^2\, \ud{t} \le C \norm{x}^2$. Similar remarks hold true for the dual square function.

\smallskip

\noindent
{\bf Claim:}\ {\em The function $g$ considered as a mapping 
$\strip{\omega} \to \Ell{2}(\R)$ has \ELLONE-frame-bounded range.} 

\begin{proof}
We first note that
\[ \sup_{z \in \strip{\omega}} \int_\R \abs{\psi(t+z)}\, \ud{t} < \infty
\]
by Lemma~\ref{fcru.l.elem-prop}.\ref{fcru.l.elem-prop-a}. Moreover,
$\psi', \psi'' \in \calE(\strip{\alpha})$ for each $\alpha \in
(\omega, \omega')$ by
Lemma~\ref{fcru.l.elem-prop}.\ref{fcru.l.elem-prop-d}. As before, this
implies that
\[ \sup_{z \in \strip{\omega}} \norm{\psi(\cdot + z)}_{\We_1^2(\R)} =
\sup_{z \in \strip{\omega}} \int_\R \abs{\psi(t+z)} + \abs{\psi'(t+z)}
+ \abs{\psi''(t+z)}\, \ud{t} < \infty.
\]
Now by Lemma~\ref{lob.l.W12} the claim is proved.
\end{proof}
\end{exa}

\begin{cor}\label{cor:shift-type-sqf}
Suppose that the $\Ha^\infty(\strip{\omega})$-calculus for $A$ is bounded
and $\psi \in \calE(\strip{\omega'})$ for some $\omega' > \omega$. 
Then the dual square function associated with $\psi(t{+}z)$ is bounded.
If $X$ has finite cotype, then also the 
square function associated with $\psi(t{+}z)$ is bounded.
\end{cor}

\begin{proof}
Simply combine Theorem~\ref{lob.t.l1-sfe} with Example~\ref{acsf.exa.shift-sf}.
\end{proof}

The Fourier transform is an isomorphism on $\Ell{2}(\R)$. 
Hence (dual) square functions related by the Fourier transform 
are strongly equivalent. 

\begin{exa}[Weighted group orbits]\label{acsf.exa.weighted-group-orbits}
Let as before $\omega' > \omega$ and $\psi \in \elementary(\strip{\omega'})$.
Taking the inverse Fourier transform with respect to the variable $t$ 
in the $\Ell{2}(\R)$-valued function $\psi(t {+} z)$ yields 
the function 
\[ {\psi}^\vee(s) \ue^{-\ui sz} = \Fourier_t^{-1}( \psi(t + z))(s).
\]
Hence, the (dual) square functions associated with $\psi(t + z)$ and 
$\psi^\vee(s)\ue^{-\ui s z}$ are strongly equivalent. In particular, 
\[ \text{if}\quad
\psi(z) = \frac{ \mfrac{\pi}{\omega}}{\cosh( (\mfrac{\pi}{2\omega}) z)}
\quad\text{then}\quad
\psi^\vee(s) = \frac{1}{\cosh{\omega s}}
\]
(see Remark~\ref{appl.r.poisson} below), whence
\[ \frac{ \mfrac{\pi}{\omega}}{\cosh( (\mfrac{\pi}{2\omega}) (t{+} z))} 
\quad \strongeq\quad
\frac{\ue^{-\ui s z}}{\cosh(\omega s)}.
\]
Hence , by the results of Example~\ref{acsf.exa.shift-sf}
the latter function also has \ELLONE-frame-bounded range in $\Ell{2}(\R)$.
\end{exa}

Employing the subordination principle repeatedly one obtains
\begin{align*}
\frac{ \mfrac{\pi}{\omega}}{\cosh( (\mfrac{\pi}{2\omega})( t{+}z))} &\strongeq
\frac{\ue^{-\ui s z}}{\cosh(\omega s)}  \strongeq
\ue^{-\omega \abs{s}} \ue^{-\ui sz}
\\ & \strongeq  \big(\car_{\R_+}(s) \ue^{-\omega s} \ue^{-isz}, \,
\car_{\R_+}(s) \ue^{-\omega s} \ue^{isz}\big)
\strongeq
\big(\pm \ui \omega + t - z \big)^{-1}. 
\end{align*}
Therefore we may, informally, write
\[ \| \cosh(\omega s)^{-1} \ue^{-\ui sA} x\|_\gamma \sim
\norm{R(\pm \ui \omega + t, A)x}_\gamma.
\]
Such square functions were considered in \cite[Theorem~6.2]{KalWei2004},
see Section~\ref{ss:sing-int-repres} below.

\subsection{Sectorial operators}\label{exas.s.sectorial}
The results of the previous section have their natural analogues
for sectorial operators via the $\exp/\log$-correspondence.
Of course, one has to use the Hilbert space $\Ell{2}^*(0, \infty)$
and the ``shift type'' square functions become ``dilation type''
square functions of the form $\psi(tz)$. The analogue of the Fourier transform
is the Mellin transform, and the ``weighted group orbits'' square functions
are of the form $\psi(s) z^{-\ui s}$, i.e., the group of imaginary powers
emerges here.

\subsection{Ritt operators}\label{exas.s.Ritt}
A bounded operator on a Banach space is a {\em Ritt} operator if 
\[
{\sum}_{k \ge 1} k \norm{T^{k-1}(\Id - T)} < \infty.
\] 
The semigroup $\{ T^n \suchthat n \ge 0\}$ is the discrete analogue of
an analytic semigroup, see \cite{LeMerdy:sqf-Ritt}. The spectrum of a
Ritt operator is contained in a {\em Stolz domain} and one has a
natural functional calculus there, see
\cite{ArhancetLeMerdy:Ritt,KaltonPortal:Ritt,LancienLeMerdy:Ritt,LeMerdy:sqf-Ritt}. In
the recent article \cite{LeMerdy:sqf-Ritt}, LeMerdy considers square
functions associated with the $\ell_{2}$-valued $\Ha^\infty$-mappings
\[ f_m(k,z) := k^{m - \einhalb} z^{k-1}(1-z)^m \qquad (k \in \N)
\]
which are the discrete analogues of the $\Ell{2}^*(0, \infty)$-valued mappings
\[ g_m(t,z) := (tz)^m\ue^{-tz}
\]
of dilation type. To some extent, the theory of bounded $\Ha^\infty$-calculus
and square function estimates on Stolz domains is equivalent to the strip or the sector case, by conformal equivalence of the underlying complex domains.

\section{Applications}\label{s:appl}

In this chapter we present several applications of the 
integral representation Theorem~\ref{acsf.t.nsfo}.  
In each case one starts from very specific bounded square and dual square functions
and concludes the boundedness of an $\Ha^\infty$-calculus or even, in the case
that the Banach space has finite cotype, the boundedness
of a vectorial $\Ha^\infty$-calculus. However, one usually has to pay a price in the form
that the domain set for the holomorphic functions represented by the integral formula
has to be larger than the domain set used for the square functions.

\subsection{Cauchy--Gau\ss\  representation}\label{ss:cauchy-gauss-repres}
Our first instance uses the variant of the usual Cauchy integral formula
with an additional Gaussian factor.

\medskip
\noindent
Let $0 <  \omega< \omega'$, and let $\Gamma := \rand \strip{\omega}$
with arc length (=Lebesgue) measure
Then it is simple complex analysis to show that
\[ u(z) = \frac{1}{2\pi \ui} \int_\Gamma u(w) \tfrac{ \ue^{-(w-z)^2}}{w-z}\, \ud{w} \qquad 
(\abs{\im z} < \omega)
\]
whenever $u \in \Ha^\infty(\strip{\omega'};H)$, cf. Formula~\eqref{eq:gauss-cauchy-formel}.
To interpret it in the light of Theorem~\ref{acsf.t.nsfo}
we define
\[ m(w) := u(w),\,\, f(w,z) := \tfrac{\ue^{-\frac{1}{2}(w-z)^2}}{w-z},\,\,
g(w,z) := \ue^{-\frac{1}{2}(w-z)^2}\qquad (w\in \Gamma, \, z\in \strip{\omega})
\]
and $K := \Ell{2}(\Gamma)$. Then 
$f,g \in \Ha^\infty(\strip{\alpha};K)$ for each $\alpha \in (0, \omega)$.
Consequently, if for an operator $A$ of strip type $\omega_0< \omega$ 
on a Banach space $X$
the square and dual square functions
associated with $f$ and $g$, respectively, are bounded, then 
$A$ has a bounded $\Ha^\infty(\strip{\omega'})$-calculus. 
And if $X$ has finite cotype, then 
$A$ has a bounded {\em vectorial} $\Ha^\infty(\strip{\omega'})$-calculus.

Actually, one can say more here. Theorem~\ref{acsf.t.nsfo} yields
a constant $C\ge 0$ such that
\[ \norm{f(A)}_{\gamma} \le C \norm{f}_{\Ha^\infty(\strip{\omega})}
\quad \text{for all}\quad f\in {\bigcup}_{\omega'> \omega} 
\Ha^\infty(\strip{\omega'})
\]
If the operator $A$ is densely defined, then by the scalar/vectorial
convergence lemma one obtains a bounded (vectorial) 
$\Ha^\infty(\strip{\omega})$-calculus.

\medskip
Combining these results with Theorem~\ref{lob.t.l1-sfe}, or rather
with Corollary~\ref{cor:shift-type-sqf}, we arrive at the following
central result.

\begin{thm}\label{exas.t:bdd-vectorial-calc}
Let $\alpha > 0$ and let $\Phi: \Ha^\infty(\strip{\alpha})\to \Lin(X)$
be a bounded $\Ha^\infty$-calculus over the strip $\strip{\alpha}$ 
on a Banach space $X$ of finite cotype. Further, let
$\beta > \alpha$ and $H$ be an arbitrary Hilbert space.
Then, for each $u\in \Ha^\infty(\strip{\beta};H')$ the square function
$\Phi_\gamma(u): X \to \gamma(H;X)$ is a bounded operator and 
there is a constant $C\ge 0$ such that
\[  \norm{\Phi_\gamma(u)x}_{\gamma}
\le C  \norm{u}_{\Ha^\infty(\strip{\beta})}\, \norm{x}_X \quad
\text{for all}\quad u\in \Ha^\infty(\strip{\beta};H'),\, x\in X.
\]
\end{thm}

\begin{proof}
Fix $\omega \in (\alpha, \beta)$. Then, as in Example~\ref{acsf.exa.shift-sf},
\[
\sup_{z\in \strip{\alpha}} \norm{g(z)}_{\We_1^2(\Gamma)} + \norm{f(z)}_{\We_1^2(\Gamma)} < \infty
\]
and hence $g,f: \strip{\alpha} \to K$ have \ELLONE-frame-bounded range.
By Theorem~\ref{lob.t.l1-sfe}, the associated square and dual square functions
are bounded. As explained above, the claim now follows from
 Theorem~\ref{acsf.t.nsfo}.
\end{proof}

Clearly, Theorem~\ref{exas.t:bdd-vectorial-calc} has a 
straightforward analogue for sectorial operators.
Note that the vectorial calculus in 
Theorem~\ref{exas.t:bdd-vectorial-calc} ``lives'' on a slightly larger strip. 
Consequently, in the  sectorial version one needs to enlarge the sector.

\begin{rem}\label{bdd-vect-calc.r.lemerdy}
  While we were working on the present manuscript, Christian \LeMerdy\
  independently found the equivalent
  result of Theorem~\ref{exas.t:bdd-vectorial-calc}
  for sectorial operators \cite[Theorem~6.3]{LeMerdy:sqf-Ritt}. (His
  ``quadratic'' $\Ha^\infty$-calculus is essentially what we call a ``bounded vectorial'' 
  $\Ha^\infty$-calculus.)
\LeMerdy's proof, which rests implicitly on an 
\ELLONE-frame-boundedness argument, 
 is  based on the Franks--McIntosh decomposition, to be treated below
in Section~\ref{ss:McIntosh--Franks-repres}.   
\end{rem}

\subsection{Poisson representation}\label{ss:poisson-repres}
Our next example uses a variant of the Poisson formula for the strip.

\begin{lemma}\label{appl.l.poisson}
Let $0 < \omega < \omega'$ and $u \in \Ha^\infty(\strip{\omega'};H)$. Then 
\begin{equation}\label{appl.eq.poisson}
 u(z) = \frac{1}{2\pi} \int_{\R} 
\frac{\sfrac{\pi}{2\omega}}{\cosh( \sfrac{\pi}{2\omega}
(z+s))} \big( u(\ui \omega -s) + u( -\ui \omega -s)\big) \, \ud{s}
\end{equation}
whenever $\abs{\im z} < \omega$. 
\end{lemma}

\begin{proof}
Fix $0 < \alpha \le  \sfrac{\pi}{2 \omega}$,  then for $\abs{\im z} < \omega$ 
the function  
\[ f(w) = \frac{\alpha(z-w)}{ \sinh( \alpha( w-z))}\, u(w)
\]
is analytic in a strip larger than $\strip{\omega}$. (Note that $w= z$ is
a removable singularity.) Hence, by Cauchy's integral formula. 
\[ u(z) =  f(z) = \frac{1}{2\pi \ui} \int_\Gamma  
\frac{\alpha }{ \sinh( \alpha( w-z))}\, u(w) \, \ud{w}
\]
where $\Gamma := \rand\strip{\omega}$ with the natural orientation.
Now write out the parametrisation, specialise $\alpha = \sfrac{\pi}{2\omega}$
 and use that 
$\sinh( a \pm \ui \sfrac{\pi}{2}) = \pm \ui \cosh(a)$.
\end{proof}

\begin{rem}\label{appl.r.poisson}
Specialising $u(z) = \ue^{\ui z x}$ and $z= 0$ 
in \eqref{appl.eq.poisson} 
one obtains again the formula
\[ \frac{1}{2\pi} \int_\R 
\frac{\sfrac{\pi}{2\omega}}{\cosh( \sfrac{\pi}{2\omega}
s)} \ue^{\ui s t} \, \ud{s} = \frac{1}{\cosh(\omega t)} \qquad (t\in \R)
\]
used in Example~\ref{acsf.exa.weighted-group-orbits}.
\end{rem}

In order to apply Theorem~\ref{acsf.t.nsfo}, we need to factorise the
integral kernel in \eqref{appl.eq.poisson}. A possibility is
\begin{equation}\label{poisson.eq.product}
 \frac{\sfrac{\pi}{2\omega}}{\cosh( \sfrac{\pi}{2\omega} z)} = 
  \Big[ \frac{\alpha}{\omega} \frac{\cosh(\sfrac{\pi}{2\alpha} z) }{\cosh( \sfrac{\pi}{2\omega} z)} \Big] 
  \cdot \frac{\sfrac{\pi}{2\omega}}{\cosh( \sfrac{\pi}{2\alpha} z)} = f(z) \cdot g(z),
 \end{equation} 
for $\alpha > \omega$. With $m_u(s) := \big( u(\ui \omega -s) + u( -\ui \omega -s)\big)$,
Formula \eqref{appl.eq.poisson} then becomes
\begin{equation}\label{poisson.eq.rep}
 u(z) = \frac{1}{2\pi} \int_\R m_u(s) f(z-s) g(z-s)\, \ud{s}. 
\end{equation}
This is an instance of \eqref{acsf.e.nsfo}, whence Theorem~\ref{acsf.t.nsfo} can be applied. 
The function $f$ still looks a little unwieldy, but turns out to be strongly equivalent
to $g$, since
\begin{align*}
   &   \frac{\alpha}{\omega} \frac{\cosh(\sfrac{\pi}{2\alpha} (z{+}s)) }{\cosh( \sfrac{\pi}{2\omega} (z{+}s))}  
\strongeq \frac{2\alpha}{\pi} \cos\Big(\frac{\pi \omega}{2\alpha}\Big)   
\frac{\cosh(\omega t)}{\cos(\sfrac{\pi\omega}{\alpha}) + \cosh(2\omega t)}\,   
\ue^{-\ui t z}  \strongeq \frac{\ue^{-\ui t z}}{ \cosh(\omega t)}.
\end{align*}
Here, the first equivalence comes from taking the inverse Fourier transform, and the second holds
by multiplying by $\Ell{\infty}$-functions.

\begin{thm}\label{appl.t.poisson}
Let $A$ be a densely defined operator of strip type $\omega_0\ge 0$ 
on a Banach space $X$ (of finite cotype).
Let $\omega > \omega_0$ and suppose that the 
square and dual square functions associated with
the weighted group orbit $\ue^{-it z}/\cosh(\omega t)$ are bounded, i.e.,
\[  \Big\|\frac{\ue^{-\ui t A}x}{\cosh(\omega t)} \Big\|_\gamma \lesssim \norm{x}\quad \text{and}\quad
 \Big\|\frac{\ue^{-\ui t A'}x'}{\cosh(\omega t)}\Big\|_{\gamma'} \lesssim \norm{x'}.
\]
Then $A$ has a bounded (vectorial) $\Ha^\infty$-calculus on $\strip{\omega}$.
\end{thm}

\begin{proof}
We apply the preceding remarks to obtain 
\[ \norm{\Phi_\gamma(f)}_{\gamma} \lesssim \norm{f}_{\Ha^\infty(\strip{\omega})}
\]
for $f\in \bigcup_{\omega' > \omega} \Ha^\infty(\strip{\omega'})$. 
The remaining step to a full vectorial $\Ha^\infty(\strip{\omega})$-calculus
is made via the vectorial convergence lemma (Lemma~\ref{acsf.l.vcl}).
\end{proof}

\begin{rem}\label{poisson.r.transference}
The factorisation~\ref{poisson.eq.product} 
has been used in \cite{Haa2009} to prove the
transference principle for groups. A close inspection reveals that Formula \eqref{poisson.eq.rep}
is --- after taking a Fourier transform --- just the transference identity in disguise.  
Using the arguments in the proof of \cite[Theorem~3.2]{Haa2009} leads to 
an alternative proof of Lemma~\ref{appl.l.poisson}, see the following section.
\end{rem}

\subsection{CDMcY-representation}\label{appl.s.cdmcy}

A variant of the Poisson type representation in the previous section was used
by Cowling, Doust, McIntosh and Yagi in their influential paper \cite{CowDouMcIYag1996}. To motivate
it we first sketch an

\medskip
\noindent
{\em Alternative proof of Lemma~\ref{appl.l.poisson}:}\ 
Suppose first that $f = \fourier{g}$
is the Fourier transform of a function $g$ on $\R$ with 
$\int_\R \cosh(\omega t)\abs{g(t)}\, \ud{t} < \infty$.  We abbreviate $g_\omega(t) := \cosh(\omega t) g(t)$. 
Then
\begin{align*}
  f(z) & = \int_\R \ue^{-\ui t z} g(t)\, \ud{t}
= \int_\R \frac{\ue^{-\ui t z}}{\cosh(\omega t)}  \cosh(\omega t)g(t)\, \ud{t}
= \int_\R \frac{\ue^{-\ui t z}}{\cosh(\omega t)}  g_\omega(t)\, \ud{t}
\\ & = \int_\R \frac{\sfrac{\pi}{\omega}}{ \cosh(\sfrac{\pi}{2\omega} (z + s))}
   \Fourier^{-1}(g_\omega)(s)\, \ud{s}
\end{align*} 
and
\begin{align*}
\Fourier^{-1}(g_\omega)(s) & = \frac{1}{2\pi} \int_\R g(t) \cosh(\omega t) \ue^{\ui t s}\, \ud{t}
= \frac{1}{4\pi} \int_\R g(t) \big( \ue^{-\ui( (\ui \omega -s)  t} + 
\ue^{ -\ui( -\ui \omega - s)t}  \big) \, \ud{s}
\\ & = \frac{1}{4\pi} \big( f(\ui \omega - s) + f( -\ui \omega -s) \big). 
\end{align*} 
Hence, \eqref{appl.eq.poisson} is valid for such functions $f$, and the general
case is proved by approximation.\qed

\medskip
\noindent
The idea behind the CDMcY-representation is to sneak in an additional factor in the previous argument and compute
formally\footnote{In order to keep our own presentation consistent, we deviate inessentially from   
\cite{CowDouMcIYag1996} in that we use inverse Fourier transforms in place of Fourier transforms,
and work on strips in place of sectors.}
\begin{align*}
  f(z) & = \int_\R \ue^{-\ui t z} g(t)\, \ud{t}
= \int_\R \psi^\vee(t) \ue^{-\ui t z} \frac{g_\omega(t)}{\psi^\vee(t)\cosh(\omega t)}  \, \ud{t}
\\ & = \int_\R \psi(z {+} s)\, \Fourier^{-1}\Big[\frac{g_\omega(t)}{\psi^\vee(t)\cosh(\omega t)}\Big](s)\, \ud{s} 
\\ & =   
\int_\R \psi(z {+} s)\, \Big[\Fourier^{-1}\Big(\frac{1}{\psi^\vee(t)\cosh(\omega t)}\Big) 
\ast  \Fourier^{-1}(g_\omega)\Big](s)\, \ud{s}.
\end{align*} 
To make this work, the authors require that 
\begin{equation}\label{cdmcy.eq.req1}
 \frac{1}{\psi^\vee(t)\cosh(\nu t)} \in \Ell{\infty}(\R) \quad \text{for some $\nu <  \omega$.} 
\end{equation}
In order to  obtain an $\Ell{\infty}$-bound on 
\[ m_f(t) := \Fourier^{-1}\Big(\frac{1}{\psi^\vee(t)\cosh(\omega t)}\Big) 
\ast  \Fourier^{-1}(g_\omega)
\]
in terms of the $\Ha^\infty$-norm of $f$ it remains to ensure that the first factor in the convolution
is in $\Ell{1}(\R)$. Hence, by the well known  Carlson--Bernstein criterion and under the hypothesis \eqref{cdmcy.eq.req1},
it suffices to have
\[   \frac{ (\psi^\vee)'}{\psi^\vee} \frac{\cosh(\nu t)}{\cosh(\omega t)} \in \Ell{2}(\R).
\]
Under the additional assumption (made in \cite{CowDouMcIYag1996})  that 
$\psi(z) = \vphi(\ue^z)$, and $\vphi \in \Ha^\infty_0$ on a sector, this is the case, see \cite[p.~67]{CowDouMcIYag1996}.

\begin{rem}
The authors of  \cite{CowDouMcIYag1996} used this representation to infer bounded $\Ha^\infty$-calculus
from ``weak quadratic estimates'' of the form
\[  \int_\R \abs{\dprod{\psi(t + A)x}{x'} } \ud{t} \lesssim \norm{x} \norm{x'}. 
\]
This notion is not covered so far in  our approach (which avoids computing with $X$-valued functions). 
However, when it comes to square function estimates, it is not clear whether there is really
a surplus compared with Theorem~\ref{appl.t.poisson}. The reason is that requirement \eqref{cdmcy.eq.req1}
implies that 
\[ \frac{\ue^{-it z}}{\cosh(\nu t)}\,\, \subord \,\, \ue^{-\ui tz}\psi^\vee(t) \,\,\strongeq \,\,\psi(z {+} s)
\]
and hence the boundedness of the shift-type square function associated with $\psi$ 
implies the boundedness of the ``weighted group orbit''-square functions
considered in Theorem~\ref{appl.t.poisson}. (Even more, 
the CDMcY-choice of $\psi$ implies
also that $\psi^\vee/\cosh(\omega' \cdot) \in \Ell{\infty}(\R)$ 
for some $\omega'$ and 
hence square function estimates for $\psi$ are basically equivalent with
square function estimates for weighted group orbits.)
\end{rem}

\subsection{Laplace (transform) representation}\label{ss:Laplace-repres}

In this section we work with a sectorial operator $A$ of angle $\theta < \sfrac{\pi}{2}$, i.e.,
 $-A$ generates a (sectorially) bounded holomorphic semigroup $(e^{-tA})_{t> 0}$. 
Ubiquitous  square functions in this context are dilation type square functions $\psi(tz)$ with $H= \Ell{2}^*(0, \infty)$,
in particular for the choice $\psi= \psi_\alpha$, where
\[ \psi_\alpha(z) = z^\alpha \ue^{-z} \qquad (\alpha > 0)
\]
and $z$ is from a sufficiently large sector. Aiming at an application of Theorem~\ref{acsf.t.nsfo} we look for
a representation
\begin{align*}
 u(z) & = \int_0^\infty m_u(t) \psi_\alpha(tz)\psi_\beta(tz) \, \tfrac{\ud{t}}{t} 
= z^{\alpha + \beta} \int_0^\infty m_u(t)t^{\alpha + \beta - 1}  \ue^{-2tz} \,\ud{t}
\\ & =
\frac{1}{2^{\alpha + \beta}}
z^{\alpha + \beta} \int_0^\infty m_u(t/2)t^{\alpha + \beta - 1}  \ue^{-tz} \,\ud{t}
\end{align*}
with $m_u \in \Ell{\infty}(0, \infty)$. This means that 
$\tfrac{1}{2^{\alpha + \beta}}m_u(t/2)t^{\alpha + \beta - 1}$ is the inverse Laplace 

\noindent \begin{minipage}{0.6\linewidth}
transform of $u(z)/z^{\alpha + \beta}$. 
Now let us suppose that $u \in \Ha^\infty(\sector{\omega'})$ for some $\omega'> \sfrac{\pi}{2}$. 
Then one can use the complex inversion formula to compute
\[ \frac{m_u(t/2)}{2^{\alpha + \beta}} t^{\alpha + \beta - 1} =
\frac{1}{2\upi \ui} \int_{\Gamma_{\omega, t}} \frac{u(z)}{z^{\alpha + \beta}} \ue^{tz} \, \ud{z}
\]
 Here, $\sfrac{\pi}{2} < \omega <\omega'$ and   
the contour $\Gamma_{\omega,t}$ is the boundary of the region
$\sector{\omega} \ohne \{ \abs{z}\le t\}$.  
Hence, with a change of variable, 
\begin{align*}
 m(t/2) = & \; \frac{2^{\alpha + \beta}}{2\upi \ui} 
\int_{\Gamma_{\omega, t}} \frac{u(z) t^{\alpha + \beta-1}}{z^{\alpha + \beta}} \ue^{tz} \,\ud{z} \\
= & \; \frac{2^{\alpha + \beta}}{2\upi \ui} 
\int_{\Gamma_{\omega, 1}}  \frac{u(z/t)}{z^{\alpha + \beta}} \ue^{z} \,\ud{z},
\end{align*}
\end{minipage}
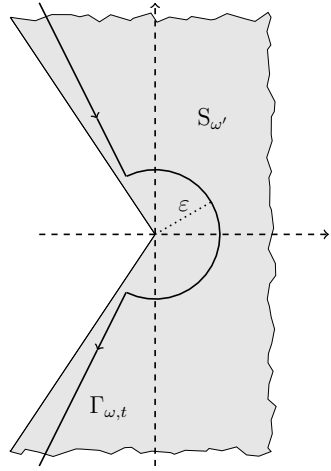
\begin{minipage}{0.6\linewidth}
\noindent \resizebox{.57\linewidth}{!}{
\begin{tikzpicture}
% sektor omega'
\draw[fill=gray!20]    decorate [decoration={random steps,segment length=2mm}] { (-2.5,3.75)-- (2,3.75)-- (2, -3.75)-- (-2.5,-3.75) }
                          [rounded corners=0.1ex]  -- (-2.5,-3.75)-- (0,0) -- cycle ;
\path(1, 2) node {{\Large $\sector{\omega'}$}} ;
% sektorrand
\draw (0,0)--(-2.5, 3.75); 
\draw (0,0)--(-2.5,-3.75);
% axen
\draw[thick,dashed, ->] (-2,0)--(3,0);
\draw[thick,dashed, ->] (0,-4)--(0,4);
% integrationsweg
\draw[ thick,->] (-2, 4) -- (-1,2);
\draw[ thick,-] (-1,2)--(-.5, 1); 
\draw[ thick,->] (-.5,-1)--(-1,-2);
\draw[ thick,-] (-1,-2)--(-2,-4); 
\draw[ thick,domain=-116:116] plot ({sqrt(1.25)*cos(\x)}, {sqrt(1.25)*sin(\x)});
\path(-0.8,-3) node{{\Large $\Gamma_{\omega,t}$}} ;
% eps
\draw[dotted, thick,domain=0:sqrt(1.25)] plot ({0.86603*\x}, {.5*\x});
\path(.5, .5) node {{\Large $\varepsilon$}} ;
\end{tikzpicture}}
\end{minipage}

\noindent and this yields an estimate
\begin{align*}
  \norm{m_u}_{\Ell{\infty}(0, \infty)}  \lesssim \Big(\int_{\Gamma_{\omega,t}} \frac{e^{\re z}}{\abs{z}^{\alpha + \beta}} \abs{\ud{z}} \Big)
\,\,\norm{u}_{\Ha^\infty(\sector{\omega})}.
  \end{align*}
Combining these consideration with 
Theorem~\ref{acsf.t.nsfo} we obtain the following result.

\begin{thm}\label{appl.t.poisson-HG-sqf}
Let $A$ be a sectorial operator, with dense domain and range, 
of angle $\theta < \sfrac{\pi}{2}$ on a Banach space $X$ (of finite cotype).
Let $\alpha ,\, \beta> 0$ and suppose that the square function associated with 
$ \vphi_\alpha(tz)=  (zt)^\alpha \ue^{-tz}$ and the dual square function associated with 
$ \vphi_\beta(tz) = (zt)^\beta \ue^{-tz}$ are bounded operators. 
Then $A$ has a bounded (vectorial) $\Ha^\infty$-calculus on each sector 
$\sector{\omega'}$ with $\omega' > \sfrac{\pi}{2}$.
\end{thm}

\begin{rem}
If $\alpha + \beta > 1$, then one can choose $\omega  = \sfrac{\pi}{2}$ in the complex inversion formula. 
Hence one obtains an estimate $\norm{m_u}_{\Ell{\infty}} \lesssim \norm{u}_{\Ha^\infty(\sector{\sfrac{\pi}{2}})}$ and then, by the convergence lemma, a
bounded $\Ha^\infty(\sector{\sfrac{\pi}{2}})$-calculus.
\end{rem}

It is an intriguing question under which conditions one can actually push the ``$\Ha^\infty$-angle'' 
(that is, the angle $\omega$ such that $A$ has a bounded (vectorial) $\Ha(\strip{\omega})$-calculus)
down below $\sfrac{\pi}{2}$.
To the best of
our knowledge, this requires using the concept of $R$-boundedness and
the multiplier theorem for $\gamma$-spaces. Recently
\cite{LeMerdy:sharp-equivalence}, Christian \LeMerdy\ has shown that if
$X$ has Pisier's property $(\alpha)$, then boundedness of the (dual)
square function associated with $\vphi_{\sfrac{1}{2}}(tz)=
(tz)^{\sfrac{1}{2}} \ue^{-tz}$ already suffices. Apart from a result
by Kalton and Weis involving $R$-boundedness, \LeMerdy\ needed to
``improve the exponent'', i.e., to pass from $\vphi_{\sfrac{1}{2}}$ to
$\vphi_{1}$ and even to $\vphi_{\sfrac{3}{2}}$. His clever argument,
carried out for $X$ being an $\Ell{p}$-space, can be covered by our
abstract theory.

\begin{lemma}[\LeMerdy] 
Suppose that $X$ is a Banach space with  Pisier's property $(\alpha)$, and let 
$A$ be a sectorial operator of angle $\theta < \sfrac{\pi}{2}$, with sectorial functional calculus $\Phi$.
Suppose that for given $\alpha, \beta > 0$ the square functions $\Phi_\gamma(\vphi_\alpha)$ and 
$\Phi_{\gamma}(\vphi_\beta)$ are bounded operators. Then $\Phi_\gamma(\vphi_{\alpha + \beta})$ is bounded, too.
\end{lemma}

\begin{proof}
  The proof relies on the tensor product square function and
  subordination.  We abbreviate $H = \Ell{2}(\R_+)$.  Since $X$ has
  property $(\alpha)$, Lemma~\ref{sfno.l.tensor} shows that the
  function
\[ 
   (\vphi_\alpha \tensor \vphi_\beta)(s,t,z) = s^{\alpha} t^{\beta} z^{\alpha + \beta} \ue^{-(t+s)z}
\]
yields a bounded square function on $\Ell{2}^*(0, \infty) \otimes
\Ell{2}^*(0, \infty)$.
Equivalently, the function
\[
  (f_\alpha \tensor f_\beta)(s,t,z) = s^{\alpha-\einhalb} t^{\beta-\einhalb} z^{\alpha + \beta} \ue^{-(t+s)z}
\]
yields a bounded square function  on $H \otimes H$,
 where we have put $f_\alpha(t, z) := t^{\alpha-\einhalb}
z^\alpha \ue^{-tz}$.  

\noindent Next, observe that $T: H \to H \otimes H$
defined by $(T f) (s, t) = (t{+}s)^{-\einhalb} f(t{+}s)$ is
isometric. Indeed,
\[
\int_0^\infty \int_0^\infty \tfrac{ |f(t{+}s)|^2 }{t+s}\,\ud{t}\ud{s}
= 
\int_0^\infty \int_s^\infty \tfrac{ |f(t)|^2 }{t}\,\ud{t}\ud{s}
=
\int_0^\infty  |f(t)|^2  \Bigl( \tfrac{1}{t} \!\int_0^t \,\ud{s}\Bigr)\ud{t}.
\]
Therefore, $T^* T = \text{Id}_{H}$.  As a consequence, $\Phi_\gamma(
f) \in \Lin(X; \gamma(H; X))$ if and only if $\Phi_\gamma( T\circ f)
\in \Lin(X; \gamma(H{\otimes}H; X))$. Now,
\begin{align*}
  T^*(f_\alpha \otimes f_\beta )(t,s, z) 
& =\; \tfrac1{\sqrt{t}} \int_0^t  f_\alpha(t-s, z) f_\beta(s, z)\,ds \\
& =\;  c_{\alpha, \beta} \; t^{\alpha+\beta-\einhalb} z^{\alpha + \beta}  \ue^{-tz},
\end{align*}
and this  concludes the proof.
\end{proof}

\begin{rem}
Passing from $\Ell{2}^*(0, \infty)$ to $\Ell{2}(\R_+)$ and then to 
$\Ell{2}(\R)$ via the  Fourier transform, one has
\[ 
\vphi_{\sfrac{1}{2}}(tz) = (tz)^{\sfrac{1}{2}}\ue^{-tz} \quad \text{on 
$\Ell{2}^*(0, \infty)$} \quad
\strongeq \quad 
\frac{z^{\sfrac{1}{2}}}{z + \ui s}
\quad \text{on  $\Ell{2}(\R)$}
\]
These  square functions --- in the form $A^\einhalb R(\ui s, A)x$ and 
$(A')^\einhalb R(\ui s, A')x'$ ---
 were considered by Kalton and Weis in \cite[Theorem~7.2]{KalWei2004}.
\end{rem}

\subsection{Franks--McIntosh representation}\label{ss:McIntosh--Franks-repres}

In \cite{FraMcI1998} Franks and McIntosh prove the following result: 
Given  $\theta\in (0,
\pi)$ there exist sequences $(f_n)_n, (g_n)_n$ in
$\Ha^\infty(\sector{\theta})$ such that
\begin{aufzi}
\item\label{item:FrMcI:a} $\sup_{z\in \sector{\theta}} \sum_n
  |f_n(z)|+|g_n(z)| \le C$, 
  \item\label{item:FrMcI:b} Any $\phi \in  \Ha^\infty(\sector{\theta}; X)$
    decomposes as $\phi(z) = \sum_n a_n f_n(z) g_n(z)$ with
    coefficients $a_n \in X$ satisfying $\norm{a_n}\lesssim \norm{\phi}_{\infty}$.
\end{aufzi}
The decomposition~\ref{item:FrMcI:b} is an instance of our
representation formula \eqref{acsf.e.nsfo} for $K = \ell_2$. 
Condition~\ref{item:FrMcI:a} tells --- in our terminology ---
that the $\ell_2$-valued $\Ha^\infty$-functions 
$F(z) =(f_n(z))_n$ and $G(z) = (g_n(z))_n$ have \ELLONE-frame-bounded range.

\medskip
In \cite{LeMerdy:sqf-Ritt} \LeMerdy\ employs this representation to prove that
on a space $X$ of finite cotype each 
sectorial operator with a bounded $\Ha^\infty$-calculus
on a sector has bounded vectorial (``quadratic'') $\Ha^\infty$-calculus
on each larger sector, i.e., the sectorial equivalent to 
our Theorem~\ref{exas.t:bdd-vectorial-calc}, cf.~Remark~\ref{bdd-vect-calc.r.lemerdy}.

\subsection{Singular Cauchy representation}\label{ss:sing-int-repres}

All the results in this chapter so far were applications of
Theorem~\ref{acsf.t.nsfo}, that is, they infer a bounded (vectorial)
$\Ha^\infty$-calculus from bounded square and dual square
functions. In the present section, however, we shall treat an
application of Lemma~\ref{nsfo.l.pushing-through}. That is, we want to
infer bounded $\Ha^\infty$-calculus from upper and lower square
function estimates. We discuss an example due to Kalton and Weis
\cite{KalWei2004}, see also \cite[Theorem~10.9]{KunWei2004}.

\medskip
\noindent
Let $A$ be a densely defined operator of strip type $\omega_0 \ge 0$ on a Banach space $X$.
We fix $\omega > \omega_0$ and let 
$\Gamma_\omega = \rand\strip{\omega} = (\ui \omega + \R) \cup (-\ui \omega + \R)$ 
with arc length measure,  let $H := \Ell{2}(\Gamma_\omega)$ and consider the $H$-valued function
$g(\lambda, z) := \frac{1}{\lambda - z}$. Under the canonical isomorphism $H \cong \Ell{2}(\R) \oplus \Ell{2}(\R)$,
the function $g$ is strongly equivalent with the {\em pair} $(\pm\ui \omega +s - z)^{-1}$ of shift-type
square functions, which --- as demonstrated in Section~\ref{acsf.exa.shift-sf} --- is again strongly equivalent
with the weighted group orbit square function associated with 
$\ue^{-\ui sz}/\cosh(\omega s)$. Our aim is to prove the following remarkable result
of Kalton and Weis \cite[Theorem~6.2]{KalWei2004}.

\begin{thm}\label{cor:von-Resolventen-sqf-zu-Hinfty-calc}
Let $A$ be a densely defined 
operator of  strip type $\omega_0$ and let $\omega>\omega_0$.
Suppose that
  \[
\norm{R(\lambda,A)x}_{\gamma(\Ell{2}(\Gamma_\omega);X)} 
\quad\sim\quad
  \norm{x} \qquad (x \in X).
\]
Then $A$ has a bounded
$\Ha^\infty(\strip{\omega})$-calculus.
\end{thm}

\begin{proof}
Let $0\le \omega_0 < \alpha < \omega < \omega'$ and let $f \in
\Ha^\infty(\strip{\omega'})$ and $\lambda \in \Gamma_\omega$.  Let
$\Gamma_{\varepsilon, \lambda}  = 
\partial (\strip{\omega}\cup B(\lambda, \varepsilon))$, oriented positively.  

\noindent \resizebox{\linewidth}{!}
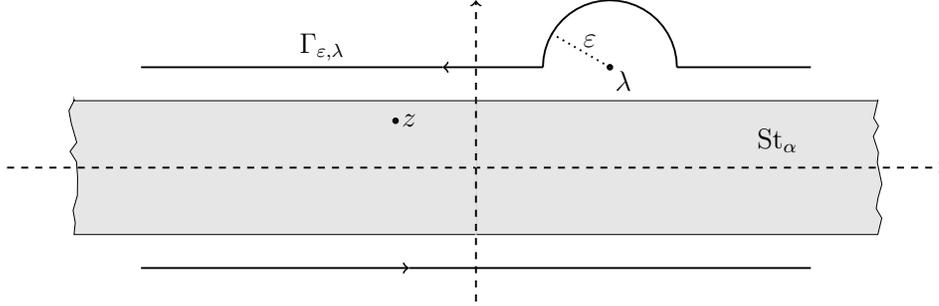
{
\begin{tikzpicture} 
% streifen alpha
\draw[fill=gray!20] decorate [decoration={random steps,segment length=2mm}] { (-6,1)--(-6,-1) }
                             [rounded corners=0.1ex]  -- (0,-1) -- (0,1) -- (-6,1) --cycle ;
\draw[fill=gray!20] decorate [decoration={random steps,segment length=2mm}] { (6,1)--(6,-1) }
                             [rounded corners=0.1ex]  -- (0,-1) -- (0,1) -- (6,1) --cycle ;
\path(4.5, .4) node {{\Large $\strip{\alpha} $}} ;
% axen
\draw[thick,dashed,->] (-7,0)--(7,0);
\draw[gray!20, thick, -] (0,-1)--(0,1); %% rand aus fill=gray loeschen :)
\draw[thick,dashed, ->] (0,-2)--(0,2.5);
% integrationsweg
\draw[thick, ->] (-5,-1.5)--(-1, -1.5); % unten
\draw[thick, -] (-1,-1.5)--(5, -1.5);   % unten
\draw[thick, -] (-5, 1.5)--(-0.5,1.5);  % oben l
\draw[thick, <-] (-0.5,1.5)--(1,1.5);   % oben m
\draw[thick, -] (3, 1.5)--(5, 1.5);   % oben r
\draw[thick, domain=0:180] plot ({2+cos(\x)}, {1.5 + sin(\x)});
\path(-2.3, 1.8) node {{\Large $\Gamma_{\varepsilon, \lambda}$}} ;
% lambda, epsilon, z
\draw [ultra thick] (2,1.5) circle [radius=0.02];
\draw [ultra thick] (2,1.5) circle [radius=0.01];
\path(2.2, 1.3) node {{\Large $\lambda$}} ;
\draw[dotted, thick,domain=0:1] plot ({2-0.86603*\x}, {1.5+.5*\x});
\path(1.7, 1.9) node {{\Large $\varepsilon$}} ;
\draw [ultra thick] (-1.2, .7) circle [radius=0.02];
\draw [ultra thick] (-1.2, .7) circle [radius=0.01];
\path(-1, .7) node {{\Large $z$}} ;
\end{tikzpicture}
}

Then for $z\in \strip{\alpha}$
\begin{align*}
\frac{f(w)}{(w-z) (\lambda - z)} = \frac{f(w)}{(w- \lambda) (\lambda - z)}
- \frac{f(w)}{(w - \lambda ) (w - z)}.
\end{align*}
Integrating this with respect to $w$ over $\{ w\in \Gamma_{\veps, \lambda}, \abs{w}
\le r\}$ and letting $r \to \infty$ yields
\[ \frac{ f(z)}{ \lambda - z} = \frac{f(\lambda)}{\lambda - z}
-  \frac{1}{2\upi \ui}
\int_{\Gamma_{\veps, \lambda}}  \frac{f(w)}{(w - \lambda )(w - z)}\, \ud{w}.
\]
By the fractional Cauchy theorem,
the limit as $\veps \to 0$ of the integral over the half circle 
avoiding $\lambda \in \Gamma_\omega$ at distance $\veps$ is 
\[  \frac{1}{2\upi\ui} \cdot (\ui \pi) \frac{f(\lambda)}{\lambda - z}
= \frac{1}{2} \frac{f(\lambda)}{\lambda - z}.
\]
Hence, as $\veps \to 0$ we obtain
\begin{align*}
\frac{ f(z)}{ \lambda - z}
& = \frac{f(\lambda)}{\lambda - z} - 
 \frac{1}{2\upi \ui}\text{p.v.}\!\!\int_{\Gamma_\omega}   
\frac{f(z)}{(w - \lambda )(w - z)}\, \ud{w} -  
\frac{ f(\lambda)}{ 2(\lambda - z) }
\\ & = \frac{f(\lambda)}{2(\lambda - z)} + 
\frac{1}{2\upi \ui}  \text{p.v.}\!\!\int_{\Gamma_\omega}   
\frac{f(w)}{(\lambda - w)(w - z)}\, \ud{w}
\end{align*}
Let $T_f: \Ell{2}(\Gamma_\omega) \to \Ell{2}(\Gamma_\omega)$ 
be defined by 
\[ (T_fh)(\lambda) := \frac{f(\lambda)}{2} h(\lambda) 
+ \frac{1}{2\upi \ui}  \text{p.v.}\!\!\int_{\Gamma_\omega}   
\frac{f(w)}{(\lambda - w)}\,h(w) \ud{w} \qquad (\lambda \in \Gamma_\omega).
\]
Note that $T_f$ is bounded by a constant times 
$\norm{f}_{\Ha^\infty(\strip{\omega})}$:
the first summand is simply multiplication by $\tfrac{1}{2} f$
and the second is multiplication with $f$ composed with 
convolution with $1/w$.  

\noindent
Now, by the computations above, we have  
\[   f(z) g(\lambda, z) =  \big(T_f( g(\cdot, z)\big))(\lambda) \qquad (z\in \strip{\alpha},
\lambda \in \Gamma_\omega).
\]
Viewing $g$ as a function  in $\Ha^\infty(\strip{\alpha}; 
\Ell{2}(\Gamma_\omega))$ we hence have
\[ f \cdot g = T_f \nach g
\]
as in the hypotheses of Lemma~\ref{nsfo.l.pushing-through}.
(Note that we as usual identify $\Ell{2}(\Gamma_\omega) = \Ell{2}(\Gamma_\omega)'$ here.) We hence obtain a constant $C\ge 0$ independent of $\omega' > \omega$
such that
\[ \norm{f(A)} \le C \norm{f}_{\Ha^\infty(\strip{\omega})} 
\quad \text{for all}\, f\in  \Ha^\infty(\strip{\omega'}).
\]
The claim now follows from the scalar convergence lemma
\cite[Section 5.1]{HaaseFC}. 
\end{proof}

\appendix

\section{The contraction principle for Gaussian sums}
\label{app.contr}

The aim of this section is to give a complete and concise proof of the
following fundamental result. We work over the scalar field $\K \in
\{\R, \C\}$.

\begin{thm}[Contraction Principle]\label{gamma.t.contr}
  Let $\gamma_1, \gamma_2, \dots$ be independent scalar standard
  Gaussians on some probability space, let $X$ be a Banach space,
  $x_1, \dots, x_m \in X$ and let $A = (a_{kj})_{kj}$ be a scalar
  $n\times m$-matrix. Then
\[
\Exp\norm{ {\sum}_{k=1}^n {\sum}_{j=1}^m \gamma_k a_{kj}x_j}_X^2
\le \norm{A}^2 \Exp\norm{ {\sum}_{j=1}^m \gamma_j x_j}_X^2,
\]
where the matrix $A$ is considered as an operator $A: \ell_2^m \to \ell_2^n$.
\end{thm}

The proof proceeds in three steps. In the first step one reduces the
problem to the case that $n = m$. If $m > n$ one just extends $A$ to
an $m \times m$-matrix by adding $0$-rows. If $m < n$ one extends $A$
to an $n \times n$-matrix by adding $0$-columns, and defines $x_j :=
0$ for $m < j \le n$.

Now, if $m=n$ we may suppose by scaling that $A$ is a
contraction. Then the following lemma reduces the claim to $A$ being
an isometry.

\begin{lemma}
  Every contraction on the Euclidean space $\K^d$ is a convex combination of
  at most $d$ isometries.
\end{lemma}

\begin{proof}
  This is well known, see but the proof is given here 
  for the convenience of the reader.  We may suppose that $\norm{A} = 1$.  By
  polar decomposition $A = U \abs{A}$ where $\abs{A} = (A^* A)^{\einhalb}$,
  and $U$ is isometric. Hence we may assume that $A=A^*$ is positive semi definite.
  By the spectral theorem we may even further reduce the problem to
  $A$ being a diagonal matrix with entries $1= \lambda_d \ge \dots \ge
  \lambda_1 \ge 0$. (Note that $1$ has to be an eigenvalue since
  $\norm{A} =1$.)  Now we set $\lambda _0 = 0$ and write
\[
\diag(\lambda_1, \dots ,\lambda_d)
= \Sum[d]{j=1} (\lambda_j - \lambda_{j-1}) P_j
\]
where $P_j(x_1, \dots, x_d) := (x_1, \dots, x_j, 0\dots, 0)$ is
projection onto the first $j$ coordinates. (So $P_d = I$.)  This is
convex combination of projections.  But for any orthogonal projection
$P$ on a Hilbert space,
\[
P = \frac{1}{2} \Id + \frac{1}{2} (2P- \Id)
\]
is a representation as a convex combination of unitaries, since
$(2P-I)^*(2P-I) = (2P-I)^2 = 4P^2 - 4P + I = I$. Since in the
representation above always the identity $\Id$ is used, we can collect
terms and arrive at a convex combination of at most $d$ terms.
\end{proof}

Finally, we are reduced to the case that $n= m$ and $A$ is an
orthogonal/unitary matrix. Then by the rotation invariance of the
$n$-dimensional, resp. $2n$-dimensional, standard Gaussian measure
\cite[p.239]{DieJarTon1995},
\begin{equation*}
\Exp\norm{ {\sum}_{k=1}^n {\sum}_{j=1}^n \gamma_k a_{kj} x_j }^2  
= \Exp\norm{  {\sum}_{k=j}^n  \Big({\sum}_{k=1}^n a_{kj} \gamma_k\Big) x_j }^2
= \Exp\norm{  {\sum}_{j=1}^n  \gamma_j x_j }^2.
\end{equation*}
This concludes the proof of Theorem~\ref{gamma.t.contr}.

\section{Weakly square integrable functions and Pettis integrals}\label{app.P2}

Let $(\Omega,\Sigma,\mu)$ be a measure space.  Recall that a function
$f: \Omega \to X$ is {\em $\mu$-measurable} if there is a sequence
$(f_n)_n$ of {\em step functions} (finite linear combinations of
functions $\car_A \tensor x$ with $\mu(A) < \infty$ and $x\in X$) such
that $f_n \to f$ pointwise almost everywhere. Each $\mu$-measurable
function is essentially separably valued and vanishes outside a
$\sigma$-finite set.
 
\smallskip

For a $\mu$-measurable function $f: \Omega \to X$ we let
\[
\Sigma_f := \{ A\in \Sigma \suchthat \car_A f \in \Ell{2}(\Omega;X)\}.
\]
Then $\Sigma_f$ is closed under taking finite unions and measurable
subsets.  The following result shows that $\Sigma_f$ is quite rich.

\begin{lemma}\label{P2.l.locbd}
  Let $f: \Omega \to X$ be $\mu$-measurable. Then the following
  assertions hold:
\begin{aufzi}
\item\label{item:P2.l.locbd.a} Let $A\in \Sigma$ with $\mu(A) <
  \infty$ and $\epsilon > 0$. Then there is a subset $A_\epsilon
  \subseteq A$ with $\mu(A_\epsilon)\le \epsilon$ and such that $f$ is
  bounded on $A \ohne A_\epsilon$.  In particular, $A\ohne A_\epsilon
  \in \Sigma_f$.
\item\label{item:P2.l.locbd.b} There is a sequence $A_n \in \Sigma_f$ 
 such that $\mu(A_n) < \infty$  and 
 $\car_{A_n} \nearrow \car_{\{f\neq 0\}}$ almost everywhere.
\item\label{item:P2.l.locbd.c} $\spann\{ \car_A \suchthat \mu(A) <
  \infty,\,\, A \in \Sigma_f\}$ is dense in $\Ell{2}(\Omega)$.
\end{aufzi}
\end{lemma}

\begin{proof}
\ref{item:P2.l.locbd.a} 
\ Simply note that $A \cap \{ \omega
  \suchthat \norm{f(\omega)} \le n\} \nearrow A$ as $n \to \infty$.

\noindent
\ref{item:P2.l.locbd.b}\ This follows from a) and the fact that
$\{f\not= 0\}$ is $\sigma$-finite.

\noindent
~\ref{item:P2.l.locbd.c}\ This follows from a) and the fact that the
step functions are dense in $\Ell{2}$.
\end{proof}

We let 
\[
   \De_f := \{ \car_A g \suchthat g \in \Ell{2}(\Omega), A\in \Sigma_f\}.
\]
Then $\De_f$ is an ideal of $\Ell{2}(\Omega)$, i.e., a linear subspace
of $\Ell{2}(\Omega)$ with $h\in \De_f$ whenever $\abs{h}\le \abs{k}$
and $k \in \De_f$.  Moreover, $\car_A \in \De_f$ for each $A\in
\Sigma_f$ with finite measure and hence $\De_f$ is dense in
$\Ell{2}(\Omega)$.  Now we define the operator
\[
\Jei_f : \De_f \pfeil X \qquad \Jei_f(h) := \int_\Omega hf\, \ud{\mu}.
\]
The following lemma shows that the mapping $f \mapsto \Jei_f$ is
essentially one-to-one.

\begin{lemma}\label{P2.l.injective}
  Let $f: \Omega \to X$ be $\mu$-measurable such that $\Jei_f = 0$. Then
  $f = 0$ $\mu$-al\-most everywhere.
\end{lemma}

\begin{proof}
  Let $x' \in X'$ and $A \in \Sigma_f$.  Then $\int_A (x'\nach f) g =
  0$ for all $g\in \Ell{2}(\Omega;X)$, i.e., $x'\nach \car_A f = 0$
  almost everywhere. Since $f$ is essentially separably-valued, it
  follows that $\car_A f= 0$ almost everywhere. But every set of
  finite measure differs from a set from $\Sigma_f$ by as little as we
  like, and hence $f = 0$ almost everywhere on each set of finite
  measure. Since the set $\{f \neq 0\}$ is $\sigma$-finite, 
the claim follows.
\end{proof}

In certain cases the operator $\Jei_f$ extends to a bounded operator
$\Ell{2}(\Omega) \to X$, which we denote also by $\extend{\Jei_f}$.

\begin{rem}\label{P2.r.L2}
  If $f\in \Ell{2}(\Omega;X)$ then $\De_f = \Ell{2}(\Omega)$, $\Jei_f$ is
  bounded, and $\norm{\Jei_f} \le \norm{f}_2$.  
  However, we stress that for a general $\mu$-measurable $f: \Omega \to X$ 
the space $\De_f$ need not be equal to all of $\Ell{2}(\Omega)$
even if $\Jei_f$ is bounded. 
\end{rem}

In the following we characterise those functions $f$ such that
$\Jei_f$ is bounded.

\begin{thm}\label{P2.t.P2}
  Let $f: \Omega \to X$ be $\mu$-measurable, and let $(A_n)_{n \in \N}
  \subseteq \Sigma_f$ with $\car_{A_n} \nearrow \car$ a.e. on the set
$\{ f\neq 0\}$.  Then the
  following assertions are equivalent:
\begin{aufzii}
\item\label{item:P2.t.P2.i} $\Jei_f$ is bounded;
\item\label{item:P2.t.P2.ii}  $f$ is 
``weakly $\Ell{2}$'', i.e.,  
$x'\nach f \in \Ell{2}(\Omega)$ for each $x' \in X'$;
\item\label{item:P2.t.P2.iii} $\sup_{n \in \N} \norm{\Jei_{f\car_{A_n}}} < \infty$.
\end{aufzii}  
If  {\upshape~\ref{item:P2.t.P2.i}--\ref{item:P2.t.P2.iii}} hold, then
\begin{equation}\label{P2.eq.P2-norm} 
\sup_{n \in \N} \norm{\Jei_{f\car_{A_n}}} =
\norm{\Jei_f} = \sup_{\norm{x'}\le 1} \norm{x'\nach f}_2 =: \norm{f}_{\Pe_2}
\end{equation}
and $\Jei_{f\car_{A_n}} \to \extend{\Jei}_f$ strongly.
\end{thm}

\begin{proof}
 ~\ref{item:P2.t.P2.i}$\dann$\ref{item:P2.t.P2.iii}: For $g \in
  \Ell{2}(\Omega)$ and $n \in \N$ we have
\[
\norm{\Jei_{f\car_{A_n}}(g)}
= \norm{\Jei_f(\car_{A_n} g)} \le \norm{\Jei_f} \norm{\car_{A_n} g}_2
\le \norm{\Jei_f} \norm{g}_2
\]
which implies that $\norm{\Jei_{\car_{A_n}} f} \le \norm{\Jei_f}$.

\smallskip 
\noindent
\ref{item:P2.t.P2.iii}$\dann$\ref{item:P2.t.P2.ii}: For each $n \in
\N$,
\begin{align*}
\int_{A_n} & \abs{x' \nach f}^2 = \int \abs{x'\nach (f\car_{A_n})}^2
= \sup_{\norm{h}_2 \le 1} \abs{ \dprod{ \int (f\car_{A_n}) h }{x'} }^2
\\ & \le  \sup_{\norm{h}_2 \le 1} \norm{x'}^2 \norm{\Jei_{f\car_{A_n}} h}^2
= \norm{x'}^2 \norm{\Jei_{f\car_{A_n}}}^2.
\end{align*} 
Since $A_n \nearrow \{f\neq 0\}$ up to a null set,~\ref{item:P2.t.P2.ii}
follows from~\ref{item:P2.t.P2.iii}.

\smallskip 
\noindent
\ref{item:P2.t.P2.ii}$\dann$\ref{item:P2.t.P2.i}: If $f$ is weakly
$\Ell{2}$, then by the closed graph theorem there must be $c\ge 0$
such that $\norm{x' \nach f}_2 \le c \norm{x'}$ for all $x'\in
X'$. Fix $h = \car_A g \in \De_f$, with $A\in \Sigma_f$ and $g\in
\Ell{2}$.  Then $ \dprod{\Jei_f(h)}{x'} = \int (x'\nach f) h$, whence by
Cauchy--Schwarz
\[
\abs{\dprod{\Jei_f(h)}{x'} }\le  \norm{x'\nach f}_2 \norm{h}_2
\le c \norm{x'} \norm{h}_2.
\]
This yields $\norm{\Jei_f(h)} \le c \norm{h}_2$, and~\ref{item:P2.t.P2.i}
follows.

\smallskip\noindent Finally, suppose that
\ref{item:P2.t.P2.i}--~\ref{item:P2.t.P2.iii} hold. Then
\eqref{P2.eq.P2-norm} has already been shown. For the strong
convergence note that $\Jei_{f\car_{A_n}}(h) \to \extend{\Jei}_f(h)$ for each
$h\in \De_f$, and this is dense in $\Ell{2}(\Omega)$.
\end{proof}

We let $\Pe_2(\Omega;X)$ be the space of strongly measurable functions
$f: \Omega \to X$ such that $\Jei_f$ is bounded, viz.~$f$ is weakly
$\Ell{2}$.  If $f\in \Pe_2(\Omega;X)$ then the extension $\extend{\Jei_f}$
of $\Jei_f$ to all of $\Ell{2}$ can be described by a {\em weak (= Pettis)
  integral}.  Namely for each $x' \in X'$ the function $x' \nach f$ is
$\Ell{2}$ and hence
\[ \dprod{\Jei_f(h)}{x'} = \int_\Omega \dprod{h(\omega) f(\omega)}{x'} \, 
\mu(\ud{\omega}) \quad \text{for all $h\in \Ell{2}(\Omega)$}.
\]

It is sometimes convenient to decide the boundedness of $\Jei_f$ on a
subset of its natural domain.

\begin{lemma}\label{P2.l.core}
  Let $f: \Omega \to X$ be $\mu$-measurable, and let $D$ be a subspace
  of $\De_f$, dense in $\Ell{2}(\Omega)$ and invariant under the
  multiplication by characteristic functions.  If $\Jei_f$ is bounded on
  $D$ then it is bounded on $\De_f$, with the same norm.
\end{lemma}

\begin{proof}
  By Lemma~\ref{P2.l.locbd} we find a sequence $(A_n)_n\subseteq
  \Sigma_f$ such that $\car_{A_n} \nearrow \car$ almost
  everywhere on $\{f\neq 0\}$. Let $g\in D$ and $n \in \N$. Then
\[
\norm{\Jei_{f\car_{A_n}}g}= \norm{\Jei_f(\car_{A_n} g)} \le 
c \norm{\car_{A_n} g}_2 \le c \norm{g}_2,
\] 
where $c$ is the bound of $\Jei_f$ on $D$. Since $\Jei_{f\car_{A_n}}$ is
bounded and $D$ is dense, it follows that $\norm{\Jei_{f\car_{A_n}}} \le
c$, independent of $n \in \N$.  Hence, the assertion follows from
Theorem~\ref{P2.t.P2}.
\end{proof}

We have a Fatou-type property.

\begin{lemma}[$\Pe_2$-Fatou]\label{P2.l.fatou}
  If $(f_n)_{n \in \N}$ is a bounded sequence in $\Pe_2(\Omega;X)$
  with $f_n \to f$ almost everywhere, then $f\in \Pe_2(\Omega;X)$,
  $\extend{\Jei}_{f_n} \to \extend{\Jei}_f$ strongly, and
\[
\norm{f}_{\Pe_2} \le \liminf_{n\to \infty} \norm{f_n}_{\Pe_2}.
\]
\end{lemma}

\begin{proof}
We form the set
\[
D := \{ \car_A g\suchthat g\in \Ell{2}(\Omega), A\in \Sigma_f,\,
\norm{f - f_n}_{\Ell{2}(A)} \to 0\}.
\]
By Egoroff's theorem and Lemma~\ref{P2.l.locbd}, $D$ is dense in
$\Ell{2}(\Omega)$.  Moreover, $D \subseteq \De_f$ by construction, and
$D$ is obviously invariant under multiplication by characteristic
functions.  If $h := \car_A g \in D$ then $f_n h \to fh$ in $\Ell{1}$
and hence $\Jei_{f_n}(h) =\int h f_n \to \int h f = \Jei_f(h)$.  We conclude
that
\[
\norm{\Jei_{f}(h)} = \lim_{n\to \infty}\norm{\Jei_{f_n}(h)}
\le \liminf_{n\to \infty}\norm{\Jei_{f_n}} \norm{h}_2.
\]
By Lemma~\ref{P2.l.core} it follows that $\Jei_f$ is bounded with norm
$\norm{\Jei_f} \le \liminf_n \norm{\Jei_{f_n}}$.  The rest follows by a
standard approximation argument.
\end{proof}

\section{Holomorphic functional calculus on sectors and strips}\label{app:fc}

For the reader's convenience we briefly
develop the two calculi most relevant in our context, namely the
calculus for sectorial and strip type operators. Although this has been
done at several places in the literature, e.g. in the second author's 
book \cite[Chapters 2 and 4]{HaaseFC}, the construction here is a little different in order
to achieve perfect correspondence between the sector and the strip case.

To begin with, we fix some notation. Let $\strip{0} := \R, \sector{0} := (0,\infty)$;
further,  for $\omega > 0$ we let
\[ \strip{\omega}  := \{ z\in \C \suchthat \abs{\im z} < \omega\}
\quad \text{and}\quad
\sector{\omega} := \{ z \in \C\ohne\{0\} \suchthat \abs{\arg z} < \omega\},
\]
where the latter is only meaningful if $\omega \le \pi$. In that case
the transformations $w = \log z$ and $z = \exp(w)$ form  a pair of mutually inverse
holomorphic mappings from $\sector{\omega}$ to $\strip{\omega}$ and vice 
versa. In particular, the functional calculus theories for 
$\Ha^\infty(\strip{\omega})$ and $\Ha^\infty(\sector{\omega})$ are 
equivalent. We shall concentrate on the strip case and 
only briefly touch upon the sector case.

\medskip
\noindent
For $\omega > 0$ the algebra of {\em elementary functions} on $\strip{\omega}$
is
\[ \elementary(\strip{\omega}) := \Big\{ f\in \Ha^\infty(\strip{\omega}) \suchthat
\int_{-\infty}^\infty \abs{f(r + \ui \alpha)} \ud{r} < \infty\, \,
\text{for all}\, \abs{\alpha} < \omega\Big\}.
\]
If $0 < \omega \le \pi$ we correspondingly define 
\[  \elementary(\sector{\omega}) := \Big\{ f \in \Ha^\infty(\sector{\omega}) \suchthat
\int_{0}^\infty \abs{f(r\ue^{\ui \alpha})} \tfrac{\ud{r}}{r} < \infty\, \,
\text{for all}\,\abs{\alpha} < \omega\Big\}.
\]
Then $\elementary(\sector{\omega}) = \{ f\nach\log \suchthat f\in 
\elementary(\strip{\omega})\}$.

\begin{rem}
It is common in the literature to use a class of 
elementary functions defined
via explicit growth conditions instead of integrability. In this approach, 
the class
\[ \Ha^\infty_0(\sector{\omega}) = \{ f\in \Ha^\infty(\sector{\omega}) 
\suchthat \exists s, C> 0 : \abs{f(z)} \le C \min(\abs{z}^s, \abs{z}^{-s})\}
\]
features prominently. However,
such  growth conditions are not compatible
with the $\exp/\log$-correspondence, whereas our definition is.
\end{rem}

It is clear that $f\in \elementary(\strip{\omega})$ if and only if
$f(\cdot + r) \in \elementary(\strip{\omega})$ for some/each $r\in \R$. 
Moreover, by Cauchy's theorem, the following formulae hold
for any elementary function $f\in \elementary(\strip{\omega})$:
\begin{align}
   f(z) = &\; \tfrac1{2\pi\ui} \int_{\partial\strip{\omega'} } \tfrac{f(\zeta)}{\zeta{-}z}\ud\zeta \label{eq:cauchy-formel}\\
        = &\; \tfrac1{2\pi\ui} \int_{\partial\strip{\omega'} } f(\zeta) \tfrac{\ue^{-(\zeta-z)^2} }{\zeta{-}z}\ud\zeta \qquad 
(z\in \strip{\omega'},\, 0 < \omega' < \omega).
\label{eq:gauss-cauchy-formel}
\end{align}
Note that for $z\in \C$ and $\zeta \in \strip{\omega}$
\[
\big|\ue^{-(\zeta - z)^2}\big| = \ue^{-\re(\zeta - z)^2} 
= \ue^{- (\re \zeta - \re z)^2} \cdot \ue^{(\im \zeta - \im z )^2}
\le \ue^{- (\re \zeta - \re z)^2} \ue^{(\omega + \abs{\im z})^2}.
\] 
Consequently, for fixed $z\in \C$ the function
$\zeta \to \ue^{-(\zeta - z)^2}$ is an elementary function
on $\strip{\omega}$. It follows that the
representation formula \eqref{eq:gauss-cauchy-formel} actually
holds for  {\em all} $f\in \Ha^\infty(\strip{\omega})$.

\begin{lemma}\label{fcru.l.elem-prop}
Let $0 < \alpha < \omega$ and $f\in \elementary(\strip{\omega})$.
Then the following assertions hold:
\begin{aufzi}
\item\label{fcru.l.elem-prop-a} $\sup_{\abs{s}\le \alpha} \int_{-\infty}^\infty \abs{f( r + \ui s)}\, 
\ud{r} < \infty$.
\item\label{fcru.l.elem-prop-b}  
$f\in \elementary(\strip{\alpha}) \cap \Ce_0(\cls{\strip{\alpha}})$.

\item\label{fcru.l.elem-prop-c} $\int_{\rand\strip{\alpha}} f(z) \, \ud{z} = 0$.
\item\label{fcru.l.elem-prop-d} $f'\in \elementary(\strip{\alpha})$.
\end{aufzi}
\end{lemma}

\begin{proof}
For the proof of a) fix $\alpha < \omega' < \omega$. Then for $0 \le s \le \alpha$,
\begin{align*}
\int_{\rand\strip{s}} & |f(z)|\,\abs{\ud{z}} 
\le \; \tfrac1{2\pi}  \int_{\rand\strip{s}} \int_{\rand\strip{\omega'}} \bigl|f(\zeta) \tfrac{ \ue^{-(\zeta-z)^2} }{\zeta-z} \bigr|\,\abs{\ud{\zeta}} \,\abs{\ud{z}} \\
 =  & \; \tfrac1{2\pi}   \int_{\rand\strip{\omega'}} |f(\zeta)| 
\int_{\rand\strip{s}} \bigl|\tfrac{ \ue^{-(\zeta-z)^2} }{\zeta-z}\bigr| \,
\abs{\ud{z}} \,\abs{\ud{\zeta}}\\
\le  & \; \tfrac1{2\pi} \norm{ f }_{ \Ell{1}(\partial\strip{\omega'})}
\int_{\rand\strip{s}} \tfrac{ \ue^{-(\re z)^2} \ue^{(\omega' + \alpha)^2}}{\omega' - \alpha} \,
\abs{\ud{z}} = 
\tfrac{\ue^{(\omega' + \alpha)^2}}{\sqrt{\pi}(\omega' - \alpha)} \,
\norm{ f }_{ \Ell{1}(\partial\strip{\omega'})}.
\end{align*}
b) To see that
$\abs{f(z)} \to 0$ as
$\abs{\re z} \to \infty, \abs{\im z} \le \alpha$  one uses
the representation formula \eqref{eq:cauchy-formel} or
\eqref{eq:gauss-cauchy-formel} and the dominated convergence theorem. 

\smallskip
\noindent
c) By Cauchy's formula one has $0 = \int_{R_n} f(z)\, \ud{z}$ where
$R_n$ is the rectangle with corners at $\pm n \pm \ui \alpha$, $n \in \N$.
When letting $n \to \infty$ the upper and the lower side of the rectangle
approach $\partial\strip{\omega}$ and the integrals over the left and
right side vanish since $f\in \Ce_0(\cls{\strip{\omega'}})$ by b).

\smallskip
\noindent
d) Let $\alpha < \omega'< \omega$. Then by Cauchy's integral formula, 
\[ f'(z) = \tfrac1{2\pi \ui} \int_{\rand{\strip{\omega'}}}
\tfrac{f(\zeta)\, \ud{\zeta}}{(\zeta - z)^2} \quad (\abs{\im z} < \omega').
\] 
In particular
\[ \int_{\abs{\im z} = \alpha} \abs{f(z)} \abs{\ud{z}}
\le  
\Big(\int_{\rand{\strip{\omega'}}} \tfrac{\abs{f(\zeta)} \, \abs{\ud{\zeta}}}{2\pi}
\Big) \, \Big( \max_{\zeta = \pm \ui \omega'}\int_{\abs{\im z} = \alpha} 
  \tfrac{\abs{\ud{z}}}{\abs{\zeta  - z}^2}\Big) < \infty.
\]
\end{proof}

\subsubsection*{{\bf Operators of Strip Type}}
  A closed operator $A$ on a Banach space $X$ 
 is called of {\em strip type}
  $\alpha \ge 0$, if $\sigma(A) \subseteq \overline{\strip{\alpha}}$
  and if for all $\beta>\alpha$, the resolvent $R(\cdot, A)$ is
  uniformly bounded on $\C\ohne   \strip{\beta}$. 
  If for each $\beta > \alpha$ we have an estimate $\norm{R(\lambda, A)}\lesssim ( \abs{\im\lambda}-\beta)^{-1}$
on $\C\ohne \cls{\strip{\beta}}$, 
$A$ is called of {\em strong strip type} $\alpha$. 

\medskip

\noindent For an operator $A$ of strip type $\alpha \ge 0$, $\omega > \alpha$ and 
an elementary function $f \in \elementary(\strip{\omega})$ there is a natural
definition of the operator $f(A) \in \Lin(X)$ by
\begin{equation}\label{fcru.eq.fc-def}
  f(A) \overset{\text{def}}{=} \tfrac1{2\pi \ui}
\int_{\partial\strip{\omega'}} f(z) R(z, A)\,\ud{z},
\end{equation}
which is independent of $\omega' \in (\alpha, \omega)$ by Cauchy's theorem.
The mapping $\elementary(\strip{\omega})\to \Lin(X)$ given by
$f \mapsto f(A)$ is called
the elementary calculus for $A$. It is rather routine to show by virtue of
the resolvent identity, the residue theorem and contour deformation arguments that 
this is a homomorphism of algebras, and that
\[ \big(\tfrac{f(z)}{\lambda - z}\big)(A) = f(A)R(\lambda,A)
\quad\text{and}\quad 
\big(\tfrac{1}{(\lambda - z)(\mu - z)}\big)(A) = R(\lambda,A) R(\mu,A)
\]
for all $\lambda, \mu \in \C \ohne \cls{\strip{\omega}}$, cf.~\cite[Chapter 2]{HaaseFC} or
\cite{BatHaaMub2013}.

\medskip
\noindent
An operator $A$ of strip type $\alpha\in  [0, \omega)$ on a Banach space $X$
{\em has a bounded $\Ha^\infty(\strip{\omega})$-calculus}  if  there is
 a constant $C\ge 0$ such that
\[\norm{f(A)} \le C \norm{f}_{\Ha^\infty(\strip{\omega})} 
\quad\text{for all}\, f\in \elementary(\strip{\omega}).
\]

\begin{lemma}\label{fcru.l.bdd-hinf-char}
Let $0 \le \alpha < \omega$, and let  $A$  be a densely defined 
operator $A$ of strip type $\alpha\in  [0, \omega)$ on a Banach space $X$.
Then $A$ has a bounded $\Ha^\infty(\strip{\omega})$-calculus  if  and only
if the elementary calculus has an extension to a bounded
algebra homomorphism $\Phi: \Ha^\infty(\strip{\omega}) \to \Lin(X)$. 
In this case, such an extension is unique and $\norm{\Phi(f)}\le C \norm{f}_\infty$
holds for every $f\in \Ha^\infty(\strip{\omega})$ if it holds for every $f\in
\elementary(\strip{\omega})$.
\end{lemma}

\begin{proof}
Suppose first that the bounded algebra homomorphism
$\Phi: \Ha^\infty(\strip{\omega}) \to \Lin(X)$ extends
the elementary calculus. If $f\in \Ha^\infty(\strip{\omega})$ and 
$e\in \elementary(\strip{\omega})$ such that
$ef \in \elementary(\strip{\omega})$ and $e(A)$ is injective, then $(ef)(A) = \Phi(ef) = \Phi(e) \Phi(f) = e(A) \Phi(f)$. Hence 
\[ \Phi(f) = e(A)^{-1} (ef)(A)
\]
with the natural domain. Since each 
function $e(z) = (\lambda - z)^{-2}$ with $\abs{\im \lambda} > \omega$ is an instance, 
this shows uniqueness.

\smallskip
\noindent
Now suppose that 
$\norm{f(A)} \le C \norm{f}_\infty$ for each
$f\in \elementary(\strip{\omega})$. We consider the extension $\Phi$ to all of
$\Ha^{\infty}(\strip{\omega})$ by regularisation.
Let $f\in \Ha^\infty(\strip{\omega})$. 
Then for each $n \in \N$  the function $e_n(z) := \ue^{-(\mfrac{1}{n}) z^2}$
is elementary, hence also $f e_n$ is, and thus
\[
\norm{(e_nf)(A)} \le C \norm{e_nf}_\infty \le C \norm{e_n}_\infty \norm{f}_\infty
\le C \ue^{-\tfrac{1}{n}\omega^2} \norm{f}_\infty.
\]
On the other hand, it is easy to see that for $x$ in the dense subspace
$\dom(A^2)$ of $X$, one has $x\in \dom(f(A))$ and $(e_nf)(A)x \to \Phi(f)x$.
Hence $\Phi(f)$ is bounded by $C \norm{f}_\infty$, 
and since it is closed and densely defined,  $\Phi(f) \in \Lin(X)$ and 
$\norm{\Phi(f)} \le C \norm{f}_\infty$. 
\end{proof}

\subsubsection*{{\bf Sectorial Operators}}

A closed operator $A$ with dense domain and dense range on a Banach space $X$ 
is called {\em sectorial of angle}
$\alpha \in [0, \pi)$, if $\sigma(A) \subseteq \overline{\sector{\alpha}}$
and if for all $\beta\in (\alpha, \pi)$, the mapping  $z \mapsto z R(z, A)$ is
uniformly bounded on $\C\ohne  \strip{\beta}$. 

One can set up a functional calculus for sectorial operators on sectors analogously 
to the strip case. Namely, $f(A)$ is defined for an elementary function
$f\in \elementary(\sector{\omega})$ by means of \eqref{fcru.eq.fc-def} with $\rand\strip{\omega'}$
replaced by $\rand\sector{\omega'}$. For a general $f \in \Ha^\infty(\sector{\omega})$,
$f(A)$ is defined by regularisation as described above.
Then the sectorial  analogue of Lemma~\ref{fcru.l.bdd-hinf-char} holds.

It turns out \cite[Prop.~3.5.2]{HaaseFC}
that each sectorial operator $A$ of angle $\alpha$ 
has a logarithm $\log(A)$, which is of (strong) strip type $\alpha$. 
The functional calculi of these operators are linked via the exp/log-correspondence,
i.e., $f(\log A) = (f\nach \log)(A)$ for all $f\in \Ha^{\infty}(\strip{\omega})$,
see \cite[Theorem~4.2.4]{HaaseFC}. It is not true in general 
that every (strong) strip type operator is the logarithm
of a sectorial one \cite[Example 4.4.1]{HaaseFC}.
However, as long as one confines oneself to operators with bounded
$H^\infty$-calculus, the correspondence is perfect \cite[Prop.~5.5.3]{HaaseFC}, 
and hence it suffices to  consider in detail only one of these cases.

\section{\ELLONE-frame-bounded  sets}\label{app:l1-bdd}

Let $H$ be a Hilbert space.  A sequence $(f_\alpha)_{\alpha\in I}$ in $H$ is
called a {\em frame} for $H$ 
if there exist two constants $0<A<B$ such that 
\begin{equation}  \label{eq:def-frame}
   A^2 \norm{ h }_H^2 \le \sum_I  |\sprod{h}{f_\alpha}|^2 \le B^2 \norm{ h }_H^2
\qquad \text{for all
$h \in H$.}
\end{equation}
Equivalently, a frame is given by a pair of operators $(L, R)$ where
$R: H \to \ell_2(I)$ and $L: \ell_2(I) \to H$ such that $L\,R =
\text{Id}_H$. Indeed, in that case 
$f_\alpha:= R^*e_\alpha$, $\alpha \in I$, is a frame, 
where $(e_\alpha)_{\alpha \in I}$ is the canonical basis of 
$\ell_2(I)$. (One easily obtains (\ref{eq:def-frame}) with
$A=\norm{L}^{-1}$ and $B=\norm{R}$.)
Conversely, if $(f_\alpha)_{\alpha \in I}$ is a frame and 
$R: H \to \ell_2(I)$ is defined by $Rf := (\sprod{f}{f_\alpha})_{\alpha \in I}$,
then $R^*R$ is a selfadjoint, positive  and invertible operator, and 
hence $L:= (R^*R)^{-1}R^*$ satisfies $LR =\Id_H$. 
\vanish{
Indeed, denoting by $(\widetilde f_\alpha)$ the dual
frame, 
\[
h = \sum_{\alpha \in I} \sprod{h}{f_n} \widetilde f_n = \sum_{\alpha \in I} \sprod{h}{\widetilde f_n} f_n 
\]
which asserts that $R h = \sprod{h}{f_\alpha}_{\alpha\in I}$ and $L
(\lambda_\alpha) = \sum_{\alpha\in I} \lambda_\alpha \widetilde
f_\alpha$ provide such a pair of operators. Conversely, given a couple
$(L, R)$ of operators with $L\,R = \text{Id}_H$, let $f_\alpha =
R^*e_\alpha$ where $(e_\alpha)_{\alpha\in I}$ is an orthonormal basis
of $\ell_2(I)$. One easily obtains (\ref{eq:def-frame}) with
$A=\norm{L}^{-1}$ and $B=\norm{R}$.
}

Let $H$ be a Hilbert space.  A subset $M$ of $H$ is called 
\emph{\ELLONE-frame-bounded} if there exists a frame $(f_\alpha)_{\alpha\in I}$ 
of $H$ such that
\[
   \sup_{x \in M} \sum_{\alpha\in I} \abs{\sprod{x}{f_\alpha}} < \infty,
\]
In this case, in virtue of the above discussion, the 
{\em \ELLONE-frame-bound} of a subset $M \subseteq H$ is defined as
\begin{equation}  \label{eq:ell-eins-bdd}
 \abs{M}_1 := \inf  \;  \norm{L} \sup_{x \in M} \sum_{\alpha\in I} 
\abs{\sprod{Rx}{e_\alpha}} \}
\end{equation}
where the infimum is taken over all pairs of operators $(L, R)$ with
$R: H \to \ell_2(I)$ and $L: \ell_2(I) \to H$ such that $L R =
\Id_H$.  Let $X$ be a Banach space. An operator $T: X \to
H$, called {\em \ELLONE-frame-bounded} if $T$ maps the unit ball
of $X$ into an
\ELLONE-frame-bounded subset of $X$. In this case,
\[ 
   \abs{T}_{\ell_1} := \bigl| \{ Tx \suchthat \norm{x}_X\le 1\} \bigr|_1
\]
is called the {\em \ELLONE-frame-bound} of $T$.

\begin{rems}\label{lob.r.l1-bdd}
\begin{aufziii}
\item\label{item:l1-bdd-compact} \ELLONE-frame-bounded sets need not be compact. 

\item\label{item:l1-bdd-composition} Let $X, Y$ be Banach spaces.
  If $U: X \to H$ is \ELLONE-frame-bounded and $V: Y \to X$ is bounded, then
  $UV: Y \to H$ is \ELLONE-frame-bounded and
\[
     \abs{UV}_{\ell_1} \le \abs{U}_{\ell_1} \norm{V}.
\]

\item We point out that we do not know yet whether finite unions or simple
translates of 
  \ELLONE-frame-bounded sets are again \ELLONE-frame-bounded, something one
  would certainly expect to hold for a ``good'' boundedness concept.
  Consequently, we do not know whether the set of \ELLONE-frame-bounded
  operators $X \to H$ form a vector space.
\end{aufziii}
\end{rems}

\begin{lemma}\label{lob.l.l1-bdd-basic}
Let $H$ be any Hilbert space and $M \subseteq H$. Then the following
  assertions hold.
\begin{aufzi}
\item\label{item:label-lemma-a} If $M$ is \ELLONE-frame-bounded, then it is
  norm-bounded, with
   \[
     \sup_{x\in M} \norm{x} \le \abs{M}_1
   \]
   If $\spann (M)$ is finite-dimensional and $M$ is norm-bounded, then
   it is \ELLONE-frame-bounded.
 \item\label{item:label-lemma-b} If $M$ is \ELLONE-frame-bounded and $S: H
   \to K$ is an isomorphism into another Hilbert space, then $S(M)$ is
   \ELLONE-frame-bounded with
  \[ 
    \abs{S(M)}_1 \le \norm{S} \; \abs{M}_1
  \]
\item\label{item:label-lemma-c} If $M$ is \ELLONE-frame-bounded, then
  ${\cabsconv(M)}$ is \ELLONE-frame-bounded.
\end{aufzi}
\end{lemma}

\begin{proof}
  Parts~\ref{item:label-lemma-a}~and~\ref{item:label-lemma-b} are
  clear.  
  % (b)
  %  \tilde R:    S(M) \to     \ell_2
  %                y   \mapsto R \circ S^{-1} (y)
  % and \tilde L = L \circ S.
  %     \norm{ \tilde L} \sup_{y \in S(M)} \sum \dprod{ R S^{-1} y} {e_n}
  % \le \norm{S} {\norm{ L} \sup_{x \in M} \sum \dprod{ R x} {e_n}
  % now inf (R, L) and we're done.
  For the proof of~\ref{item:label-lemma-c} it suffices to
  notice that the closed unit ball of $\ell_1(I)$ is absolutely convex
  and closed in $\ell_2(I)$.
\end{proof}

\begin{rem}
Every \ELLONE-frame-bounded operator $T: X \to H$ factorises
through an $\ell_1$-space, but the converse is not true in general. 
Indeed, let $(f_n)_{n \in \N}$ be
a countable dense subset of the unit sphere 
$\{ f\in \ell_2 \suchthat \norm{f}_2 =1\}$ of $\ell_2$. Let $T: \ell_1 \to \ell_2$
be the operator defined by  $T(x_n)_n := \sum_n x_n f_n$. 
Then the image under $T$ of the unit ball of $\ell_1$ is
dense in the  unit ball of $\ell_2$, and hence $T$ is not \ELLONE-frame-bounded.
\end{rem}

For operators between Hilbert spaces, the class of
\ELLONE-frame-bounded operators coincides with the class of 
Hilbert--Schmidt operators.

\begin{lemma}\label{lob.l.l1-bdd-HS}
For an operator $T: K \to H$, $K$ and $H$ Hilbert spaces,
the following assertions are equivalent:
\begin{aufzii}
\item\label{item:ellone-bdd} $T$ is \ELLONE-frame-bounded.
\item\label{item:factors} $T$ factorises through an $\ell_1$-space.
\item\label{item:is-HS} $T$ is Hilbert-Schmidt.
\end{aufzii}
\end{lemma}

\begin{proof}
Suppose that~\ref{item:ellone-bdd} holds, i.e., $T$ is \ELLONE-frame-bounded.
Let $R: H \to \ell_2(I)$ and $L: \ell_2(I) \to H$ as in 
  (\ref{eq:ell-eins-bdd}). Then $T = L R T$ factors as $T = V\,U$ with
\[
U: \left\{
  \begin{array}{lcl}
      K &\pfeil  &\ell_1(I) \\
      x &\mapsto &\dprod{R Tx}{e_\alpha}
  \end{array}
  \right.
\qquad\text{and}\qquad
V: \left\{
  \begin{array}{lcl}
      \ell_1(I) &\pfeil  & H\\
      (\lambda_\alpha)_\alpha &\mapsto & \sum_\alpha \lambda_\alpha L e_\alpha
  \end{array}
  \right.
\]
and we have~\ref{item:factors}.  Next, recall that \cite[Corollary
4.12]{DieJarTon1995} asserts that Hilbert space operators are
Hilbert-Schmidt if and only if they factor through an $\Lin_1$-space.
Hence~\ref{item:factors} implies~\ref{item:is-HS}. Finally,
if $T: K \to H$ is Hilbert-Schmidt, the singular value decomposition
yields a representation
\[
  T = \sum_n \tau_n \konj{f_n} \otimes e_n 
\]
with orthonormal systems $(e_n)_n$ and $(f_n)_n$ and scalars $\tau =
(\tau_n)_n \in \ell_2(\N)$. %(see e.g. \cite[Theorem~4.1]{DieJarTon1995}).
We extend $(e_n)_n$ in some way to an orthonormal basis
$(e_\alpha)_{\alpha\in I}$ of $H$. Then
\[ \sum_{\alpha}  \abs{\dprod{Tf}{e_\alpha}}
= \sum_n \abs{\dprod{Tf}{e_n}} = \sum_n \abs{\tau_n} \abs{\dprod{f}{f_n}}
\le \norm{\tau}_{\ell_{2}} \norm{f}_K
\]
by the Cauchy-Schwarz and the Bessel inequalities. Hence $T$ is
\ELLONE-frame-bounded with $\abs{T}_{\ell_1} \le \norm{\tau}_{\ell_2}
= \norm{T}_{HS}$.
\end{proof}

Let us provide some other  examples of \ELLONE-frame-bounded sets/operators.

\begin{exas}\label{lob.exas.l1-bdd}
\begin{aufziii}
\item \label{item:ex-l1-bdd-1} The {\em Wiener algebra}
  $\mathrm{A}(\torus)$ is the set of continuous functions on $\torus =
  \{ z\in \C \suchthat \abs{z} =1\}$ that have absolutely summable
  Fourier coefficients.  Obviously, the embedding $\mathrm{A}(\torus)
  \subseteq \Ell{2}(\torus)$ is \ELLONE-frame-bounded.

\item  \label{item:ex-l1-bdd-2}
  As a consequence of the above item, every embedding into
  $\Ell{2}(\torus)$ that factors through the Wiener algebra is
  \ELLONE-frame-bounded. This implies, e.g., that the embedding $\Ce^s[0,1]
  \subseteq \Ell{2}[0,1]$ is \ELLONE-frame-bounded for $s > 1/2$

\item  \label{item:ex-l1-bdd-3} 
  The embeddings $\mathrm{B}_{pq}^s[0,1] \subseteq \Ell{2}[0,1]$
  and $\We_p^s[0,1] \subseteq \Ell{2}[0,1]$ are \ELLONE-frame-bounded
  whenever $s > 1/2$.
\end{aufziii}
\end{exas}

We do not know whether the continuous analogue of Example
\ref{item:ex-l1-bdd-1} is true, namely whether the embedding
\[ 
\mathrm{A}(\R) := \{ f\in \Ell{1}(\R) \suchthat \fourier{f} \in \Ell{1}(\R) \}
\subseteq \Ell{2}(\R)
\]
is \ELLONE-frame-bounded.  However, we have the following.

\begin{lemma}\label{lob.l.W12}
  The canonical embedding ${\mathrm W}_1^2(\R) \embeds \Ell{2}(\R)$ is
  \ELLONE-frame-bounded.
\end{lemma}

\begin{proof}
  Fix a function $0 \leq \eta \in \Ce^\infty(\R)$ with $\supp(\eta)
  \subseteq (\pi, \pi)$ and in such a way that with $\eta_k(t) :=
  \eta(t - k)$ for $k \in \Z$ one has
\[  
   1 = {\sum}_{k\in \Z} \eta_k.
\]
%Since the supports of the $\eta_k$ overlap, 
The double sequence $(f_{n,k})_{(n,k) \in \Z^2}$ given by $f_{n,k } :=
\eta_k \ue^{\ui n \cdot}$, forms a Gabor frame on $\Ell{2}(\R)$.  Let
$g\in \We^{2}_1(\R)$.  For $n=0$,
\[ 
 {\sum}_{k \in \Z} \abs{\int_{k-\pi}^{k + \pi} \eta_k(s) g(s) \,  \ud{s}} \le \norm{g}_{\Ell{1}}. 
\]
For $n \not= 0$, a twofold integration by parts (with vanishing
boundary terms) yields
\[ 
   \int_{k-\pi}^{k + \pi} \eta_k(s) g(s) \ue^{-\ui n s}\, \ud{s}
= - \tfrac{1}{n^2} \int_{k-\pi}^{k + \pi}[\eta_k(s) g(s)]'' \ue^{-\ui ns}\, \ud{s}.
\] 
Since $[\eta_k(s) g(s)]'' = \eta_k(s) g''(s) + 2 \eta_k'(s)g'(s) +
\eta_k'' g(s)$ and since the $\eta_k$'s are all translates of the same
function,
\[ 
   \bigl| [\eta_k(s) g(s)]'' \bigr| 
\lesssim \car_{(k - \pi, k + \pi)}(s)  \big( \abs{g(s)} +  \abs{g'(s)} + \abs{g''(s)} \big).
\]
Hence 
\[
   \sum_{n\in\Z^*, k\in\Z} \abs{\dprod{g}{f_{n,k}}} \lesssim \norm{g}_{{\mathrm W}_1^2}. \qedhere
\]
\end{proof}

Using interpolation techniques one can see that $\We^{\alpha}_1(\R)
\subseteq \Ell{2}(\R)$ is \ELLONE-frame-bounded for each $\alpha > 1$.

\def\cprime{$'$}


\begin{thebibliography}{10}

\bibitem{ArhancetLeMerdy:Ritt}
{\sc {Arhancet}, C., and {Le Merdy}, C.}
\newblock {Dilation of Ritt operators on L$^p$-spaces}.
\newblock {\tt http://arxiv.org/abs/1106.1513}, 2011.

\bibitem{BatMubVor2012}
{\sc Batty, C., Mubeen, J., and V{\"o}r{\"o}s, I.}
\newblock Bounded {$H^\infty$}-calculus for strip-type operators.
\newblock {\em Integral Equations Operator Theory 72}, 2 (2012), 159--178.

\bibitem{BatHaaMub2013}
{\sc Batty, C.~J., Haase, M., and Mubeen, J.}
\newblock The holomorphic functional calculus approach to operator semigroups.
\newblock {\em Acta Sci. Math. Univ. Szeged 79\/} (2013), 289--323.

\bibitem{BlaTarVid2000}
{\sc Blasco, O., Tarieladze, V., and Vidal, R.}
\newblock {$K$}-convexity and duality for almost summing operators.
\newblock {\em Georgian Math. J. 7}, 2 (2000), 245--268.

\bibitem{BoyDeL1992}
{\sc Boyadzhiev, K., and deLaubenfels, R.}
\newblock Semigroups and resolvents of bounded variation, imaginary powers and
  {$H^\infty$} functional calculus.
\newblock {\em Semigroup Forum 45}, 3 (1992), 372--384.

\bibitem{CowDouMcIYag1996}
{\sc Cowling, M., Doust, I., McIntosh, A., and Yagi, A.}
\newblock Banach space operators with a bounded {$H\sp \infty$} functional
  calculus.
\newblock {\em J. Austral. Math. Soc. Ser. A 60}, 1 (1996), 51--89.

\bibitem{Cowling}
{\sc Cowling, M.~G.}
\newblock Harmonic analysis on semigroups.
\newblock {\em Ann. of Math. (2) 117}, 2 (1983), 267--283.

\bibitem{DieJarTon1995}
{\sc Diestel, J., Jarchow, H., and Tonge, A.}
\newblock {\em Absolutely summing operators}, vol.~43 of {\em Cambridge Studies
  in Advanced Mathematics}.
\newblock Cambridge University Press, Cambridge, 1995.

\bibitem{FraMcI1998}
{\sc Franks, E., and McIntosh, A.}
\newblock Discrete quadratic estimates and holomorphic functional calculi in
  {B}anach spaces.
\newblock {\em Bull. Austral. Math. Soc. 58}, 2 (1998), 271--290.

\bibitem{FroWei2006}
{\sc Fr{\"o}hlich, A.~M., and Weis, L.}
\newblock {$H\sp \infty$} calculus and dilations.
\newblock {\em Bull. Soc. Math. France 134}, 4 (2006), 487--508.

\bibitem{HaaNeeVer2007}
{\sc Haak, B., van Neerven, J., and Veraar, M.}
\newblock A stochastic {D}atko-{P}azy theorem.
\newblock {\em J. Math. Anal. Appl. 329}, 2 (2007), 1230--1239.

\bibitem{HaaKun2006}
{\sc Haak, B.~H., and Kunstmann, P.~C.}
\newblock Admissibility of unbounded operators and wellposedness of linear
  systems in {B}anach spaces.
\newblock {\em Integral Equations Operator Theory 55}, 4 (2006), 497--533.

\bibitem{HaaseFC}
{\sc Haase, M.}
\newblock {\em The functional calculus for sectorial operators}, vol.~169 of
  {\em Operator Theory: Advances and Applications}.
\newblock Birkh\"auser Verlag, Basel, 2006.

\bibitem{Haa2009}
{\sc Haase, M.}
\newblock A transference principle for general groups and functional calculus
  on {UMD} spaces.
\newblock {\em Math. Ann. 345}, 2 (2009), 245--265.

\bibitem{Haa2011}
{\sc Haase, M.}
\newblock Transference principles for semigroups and a theorem of {P}eller.
\newblock {\em J. Funct. Anal. 261}, 10 (2011), 2959--2998.

\bibitem{HytNeePor2008}
{\sc Hyt{\"o}nen, T., van Neerven, J., and Portal, P.}
\newblock Conical square function estimates in {UMD} {B}anach spaces and
  applications to {$H^\infty$}-functional calculi.
\newblock {\em J. Anal. Math. 106\/} (2008), 317--351.

\bibitem{Hyt2007}
{\sc Hyt{\"o}nen, T.~P.}
\newblock Littlewood-{P}aley-{S}tein theory for semigroups in {UMD} spaces.
\newblock {\em Rev. Mat. Iberoam. 23}, 3 (2007), 973--1009.

\bibitem{KaiWei2008}
{\sc Kaiser, C., and Weis, L.}
\newblock Wavelet transform for functions with values in {UMD} spaces.
\newblock {\em Studia Math. 186}, 2 (2008), 101--126.

\bibitem{KalKunWei2006}
{\sc Kalton, N., Kunstmann, P., and Weis, L.}
\newblock Perturbation and interpolation theorems for the {$H^\infty$}-calculus
  with applications to differential operators.
\newblock {\em Math. Ann. 336}, 4 (2006), 747--801.

\bibitem{KalNeeVerWei2008}
{\sc Kalton, N., van Neerven, J., Veraar, M., and Weis, L.}
\newblock Embedding vector-valued {B}esov spaces into spaces of
  {$\gamma$}-radonifying operators.
\newblock {\em Math. Nachr. 281}, 2 (2008), 238--252.

\bibitem{KalWei2004}
{\sc Kalton, N., and Weis, L.}
\newblock The {${H}^\infty$}-functional calculus and square function estimates.
\newblock Unpublished manuscript, 2001.

\bibitem{KaltonPortal:Ritt}
{\sc Kalton, N.~J., and Portal, P.}
\newblock Remarks on {$\ell_1$} and {$\ell_\infty$}-maximal regularity for
  power-bounded operators.
\newblock {\em J. Aust. Math. Soc. 84}, 3 (2008), 345--365.

\bibitem{KriLeM2010}
{\sc Kriegler, C., and Le~Merdy, C.}
\newblock Tensor extension properties of {$C(K)$}-representations and
  applications to unconditionality.
\newblock {\em J. Aust. Math. Soc. 88}, 2 (2010), 205--230.

\bibitem{KunWei2004}
{\sc Kunstmann, P.~C., and Weis, L.}
\newblock Maximal {$L_p$}-regularity for parabolic equations, {F}ourier
  multiplier theorems and {$H^\infty$}-functional calculus.
\newblock In {\em Functional analytic methods for evolution equations},
  vol.~1855 of {\em Lecture Notes in Math.} Springer, Berlin, 2004,
  pp.~65--311.

\bibitem{LancienLeMerdy:Ritt}
{\sc Lancien, F., and Le~Merdy, C.}
\newblock On functional calculus properties of {R}itt operators.
\newblock {\tt http://arxiv.org/abs/1301.4875}, 2013.

\bibitem{LeM2004}
{\sc Le~Merdy, C.}
\newblock On square functions associated to sectorial operators.
\newblock {\em Bull. Soc. Math. France 132}, 1 (2004), 137--156.

\bibitem{LeMerdy:sqf-Ritt}
{\sc Le~Merdy, C.}
\newblock {$H^\infty$}-functional calculus and square function estimates for
  {R}itt operators.
\newblock {\tt http://arxiv.org/abs/1202.0768}, 2012.

\bibitem{LeMerdy:sharp-equivalence}
{\sc Le~Merdy, C.}
\newblock A sharp equivalence between {$H^\infty$}-functional calculus and
  square function estimates.
\newblock {\em J. Evol. Equ. 12}, 4 (2012), 789--800.

\bibitem{LeMXu2012}
{\sc Le~Merdy, C., and Xu, Q.}
\newblock Maximal theorems and square functions for analytic operators on
  {$L^p$}-spaces.
\newblock {\em J. Lond. Math. Soc. (2) 86}, 2 (2012), 343--365.

\bibitem{LinPie1974}
{\sc Linde, V., and Pi{\v{c}}~(Pietsch), A.}
\newblock Mappings of {G}aussian measures of cylindrical sets in {B}anach
  spaces.
\newblock {\em Teor. Verojatnost. i Primenen. 19\/} (1974), 472--487.

\bibitem{LindenstraussTzafriri2}
{\sc Lindenstrauss, J., and Tzafriri, L.}
\newblock {\em Classical {B}anach spaces. {II}}, vol.~97 of {\em Ergebnisse der
  Mathematik und ihrer Grenzgebiete [Results in Mathematics and Related
  Areas]}.
\newblock Springer-Verlag, Berlin, 1979.
\newblock Function spaces.

\bibitem{McI1986}
{\sc McIntosh, A.}
\newblock Operators which have an {$H\sb \infty$} functional calculus.
\newblock In {\em Miniconference on operator theory and partial differential
  equations (North Ryde, 1986)}, vol.~14 of {\em Proc. Centre Math. Anal.
  Austral. Nat. Univ.} Austral. Nat. Univ., Canberra, 1986, pp.~210--231.

\bibitem{Meyer-Nieberg}
{\sc Meyer-Nieberg, P.}
\newblock {\em Banach lattices}.
\newblock Universitext. Springer-Verlag, Berlin, 1991.

\bibitem{Nikolski:shift}
{\sc Nikol{\cprime}ski{\u\i}, N.~K.}
\newblock {\em Treatise on the shift operator}, vol.~273 of {\em Grundlehren
  der Mathematischen Wissenschaften [Fundamental Principles of Mathematical
  Sciences]}.
\newblock Springer-Verlag, Berlin, 1986.
\newblock Spectral function theory, With an appendix by S. V.
  Hru{\v{s}}{\v{c}}ev [S. V. Khrushch{\"e}v] and V. V. Peller, Translated from
  the Russian by Jaak Peetre.

\bibitem{Pisier:lust}
{\sc Pisier, G.}
\newblock Some results on {B}anach spaces without local unconditional
  structure.
\newblock {\em Compositio Math. 37}, 1 (1978), 3--19.

\bibitem{Rudin:RC}
{\sc Rudin, W.}
\newblock {\em Real and complex analysis}, third~ed.
\newblock McGraw-Hill Book Co., New York, 1987.

\bibitem{Schaefer:banach-lattices}
{\sc Schaefer, H.~H.}
\newblock {\em Banach lattices and positive operators}.
\newblock Springer-Verlag, New York, 1974.
\newblock Die Grundlehren der mathematischen Wissenschaften, Band 215.

\bibitem{Ste1958}
{\sc Stein, E.~M.}
\newblock On the functions of {L}ittlewood-{P}aley, {L}usin, and
  {M}arcinkiewicz.
\newblock {\em Trans. Amer. Math. Soc. 88\/} (1958), 430--466.

\bibitem{Ste1961a}
{\sc Stein, E.~M.}
\newblock On limits of seqences of operators.
\newblock {\em Ann. of Math. (2) 74\/} (1961), 140--170.

\bibitem{Ste1961b}
{\sc Stein, E.~M.}
\newblock On the maximal ergodic theorem.
\newblock {\em Proc. Nat. Acad. Sci. U.S.A. 47\/} (1961), 1894--1897.

\bibitem{Stein:HA}
{\sc Stein, E.~M.}
\newblock {\em Harmonic analysis: real-variable methods, orthogonality, and
  oscillatory integrals}, vol.~43 of {\em Princeton Mathematical Series}.
\newblock Princeton University Press, Princeton, NJ, 1993.
\newblock With the assistance of Timothy S. Murphy, Monographs in Harmonic
  Analysis, III.

\bibitem{Tol1981}
{\sc Tolokonnikov, V.~A.}
\newblock Estimates in the {C}arleson corona theorem, ideals of the algebra
  {$H^{\infty }$}, a problem of {S}z.-{N}agy.
\newblock {\em Zap. Nauchn. Sem. Leningrad. Otdel. Mat. Inst. Steklov. (LOMI)
  113\/} (1981), 178--198, 267.
\newblock Investigations on linear operators and the theory of functions, XI.

\bibitem{Uch1980}
{\sc Uchiyama, A.}
\newblock Corona theorems for countably many functions and estimates for their
  solutions.
\newblock preprint, UCLA, 1980.

\bibitem{Nee2010}
{\sc van Neerven, J.}
\newblock {$\gamma$}-radonifying operators---a survey.
\newblock In {\em The {AMSI}-{ANU} {W}orkshop on {S}pectral {T}heory and
  {H}armonic {A}nalysis}, vol.~44 of {\em Proc. Centre Math. Appl. Austral.
  Nat. Univ.} Austral. Nat. Univ., Canberra, 2010, pp.~1--61.

\bibitem{NeeWei2005}
{\sc van Neerven, J., and Weis, L.}
\newblock Stochastic integration of functions with values in a {B}anach space.
\newblock {\em Studia Math. 166}, 2 (2005), 131--170.

\bibitem{NeeWei2006}
{\sc van Neerven, J., and Weis, L.}
\newblock Invariant measures for the linear stochastic {C}auchy problem and
  {$R$}-boundedness of the resolvent.
\newblock {\em J. Evol. Equ. 6}, 2 (2006), 205--228.

\bibitem{NeeVerWei2007}
{\sc van Neerven, J. M. A.~M., Veraar, M.~C., and Weis, L.}
\newblock Stochastic integration in {UMD} {B}anach spaces.
\newblock {\em Ann. Probab. 35}, 4 (2007), 1438--1478.

\end{thebibliography}
\end{document}